\documentclass[12pt]{amsart}
\usepackage{amssymb}
\usepackage[margin=1in]{geometry}
\usepackage{enumerate}
\usepackage[pdftex,bookmarks=true]{hyperref}
\usepackage{fancyhdr}
\usepackage{amsmath}
\usepackage{enumitem}

\newcommand\err{\operatorname{err}}
\newcommand\supp{\operatorname{supp}}
\newcommand\sign{\operatorname{sign}}
\renewcommand\Re{\operatorname{Re}}
\renewcommand\Im{\operatorname{Im}}
\renewcommand\d{\operatorname{d}}
\newcommand\R{\mathbb {R}}
\newcommand\E{\mathbb {E}}
\renewcommand\P{\mathbb {P}}
\newcommand\C{\mathbb {C}}
\newcommand\Var{\textbf{Var }}
\newcommand\Cov{\textbf{Cov }}
\newcommand{\ep}{\varepsilon}
\newcommand{\mj}{\mathfrak J}
\newcommand{\trun}{\text{trun}}

\theoremstyle{plain}
  \newtheorem{theorem}{Theorem}[section]
  \newtheorem{corollary}[theorem]{Corollary}
  \newtheorem{lemma}[theorem]{Lemma}
  \newtheorem{proposition}[theorem]{Proposition}
\theoremstyle{definition}
  \newtheorem{remark}[theorem]{Remark}
  
\def\Xint#1{\mathchoice
	{\XXint\displaystyle\textstyle{#1}}%
	{\XXint\textstyle\scriptstyle{#1}}%
	{\XXint\scriptstyle\scriptscriptstyle{#1}}%
	{\XXint\scriptscriptstyle\scriptscriptstyle{#1}}%
	\!\int}
\def\XXint#1#2#3{{\setbox0=\hbox{$#1{#2#3}{\int}$ }
		\vcenter{\hbox{$#2#3$ }}\kern-.6\wd0}}

\def\dashint{\Xint-}

\begin{document}

\title{Random polynomials: central limit theorems for the real roots}

\author{Oanh Nguyen}
\address{Department of Mathematics, Princeton University, Princeton, NJ 08544, USA}
\address{Department of Mathematics, University of Illinois at Urbana--Champaign, Urbana, IL 61801, USA}
\email{onguyen@princeton.edu, onguyen@illinois.edu}

\author{Van Vu}
\address{Department of Mathematics, Yale University, New Haven, CT 06520, USA}
\email{van.vu@yale.edu}
\thanks{This work is partially supported by VIASM (Vietnam); O. Nguyen is supported by NSF DMS 1954174; V. Vu is partially supported by NSF  AWD0000777 and AWDD02154}
\begin{abstract}
The number of real roots has been a central subject in the theory of random polynomials and random functions since the fundamental papers of Littlewood-Offord and Kac in the 1940s. 
The main task here  is to determine the limiting distribution of  this random variable.

In 1974, Maslova famously proved a central limit theorem (CLT)  for the number of real roots of Kac polynomials. It has  remained   the only limiting theorem available for the 
number of real roots for more than four decades. 

In this paper,  using  a new approach, we derive a  general  CLT for the number of real roots of a large class of random polynomials with coefficients growing 
polynomially. Our result both generalizes and strengthens Maslova's theorem. 

\end{abstract}

\maketitle

\section{Introduction}\label{sec:intro}
Random polynomials, so simple to define but difficult to understand,  have attracted  generations of mathematicians. 
Typically, a random (algebraic) polynomial has the form 
$$ P_{n} (x) :=    c_n  \xi_n x^n + \dots  + c_1 \xi_1 x + c_0 \xi_0, $$ where $\xi_i$ are iid copies of an ({\it atom}) random variable $\xi$ with zero mean and unit variance, and $c_i$ 
are deterministic coefficients which may depend on both $n$ and $i$. Different  definitions of $c_i$  give rise to  different classes of random polynomials, which have  distinct behaviors. 

When $c_i =1$ for all $i$, the polynomial $P_{n}$ is often referred to as the \textit{Kac polynomial}.  Even this special class provides  great challenges,  which have led to rich literature  (see, for example, the books \cite{BS, Far} and the references therein).

Let  $N_n(\R)$ denote the number of real roots of $P_n$. A key problem in the theory 
of random polynomials is to understand the behavior of the random variable  $N_n(\R)$,  with $n$ tending to infinity. As a matter of fact, this is the problem that started the whole field, with fundamental works of Littlewood-Offord \cite{LO1, LO2, LO3} and Kac \cite{Kac1943average} from  the 1940s.

The first  natural  question is to determine the expectation of $N_n(\R)$. It took more than 20 years 
and the works of Kac \cite{Kac1943average}, Erd\H{o}s-Offord \cite{EO} and Ibragimov-Maslova \cite{Ibragimov1968average, Ibragimov1971average} to settle this problem for the Kac polynomial (the case $c_0 =\dots = c_n =1$). By now, the problem has been solved for many classes of random polynomials, with various choices for $c_i$ and under  very general assumptions for $\xi_i$ (see the introduction of \cite{nguyenvurandomfunction17}; also 
 \cite{HKPV, EK, MP, prosen1996, sodin2005zeroes, TVpoly, GKZ, pritsker1, pritsker2, DOV, soze1, soze2} and the references therein).

	The next, and more important, problem is to determine the variance and limiting distribution of $N_n(\R)$. 
This problem is much  harder  and our understanding is far from complete. In the 1970s, Maslova \cite{Mas2} proved the Central Limit Theorem (CLT) for the Kac polynomial. Here and later, $\xrightarrow{\text{ d }}$ means convergence in distribution; $\mathcal N(0,1)$ denotes
	the standard normal distribution, $\mu_{n}:=\E N_{n}(\R)$, $\sigma_{n}:=\sqrt {\Var N_n(\R)}$.

\begin{theorem} \label{theorem:Maslova} \cite{Mas1, Mas2}
	Let $\ep$ be a positive constant. Consider the \textbf{Kac polynomial} with the random variables $\xi_i$ being iid with mean zero, variance one, bounded $(2+\ep)$ moment, and $\P (\xi_i=0)=0$.   We have, as $n$ tends to infinity,
	\begin{equation*}
	\frac{N_n(\R)-\E N_n(\R)}{(\Var N_n(\R))^{1/2}}  \xrightarrow{\text{ d }} \mathcal N(0,1).
	\end{equation*}

	Furthermore, $\Var N_n(\R) = (K+o(1)) \log n $, where $K= \frac{4}{\pi} \left (1- \frac{2}{\pi}\right )$. 
 \end{theorem} 

The proof of Maslova relied heavily on explicit computation that requires all the $c_i$ to be equal. 
Only very recently, Central Limit Theorems have been established for other classes of polynomials, via new methods. 
 In 2015, Dalmao  \cite{dalmao2015} established the CLT  for binomial  polynomials (the case when $c_i = \sqrt { n \choose i} $), and in 2018,
 Do and the second author \cite{DV} handled Weyl polynomials ($c_i = \frac{1}{\sqrt {i ! } }$)). However, in both papers, the authors need to assume that the random variables $\xi_i$ are standard Gaussian and their arguments rely  strongly on special 
 properties of Gaussian processes. It remains a 
 major challenge to extend these results to other random variables $\xi_i$ (Rademacher, for example).  For related results concerning random trigonometric polynomials, see  \cite{gw2011, al2013, adl2016}.

The goal of this paper is to establish CLT for 
a large class of random polynomials where the deterministic  coefficients $c_i$ grow polynomially. We will only need a mild assumption on the $\xi_i$, which is satisfied by most random variables used in practice 
(in particular, this assumption is weaker than Maslova's).  In fact, we can also handle the more general setting when the 
$\xi_i$ are not  iid.

We consider 
$$
P_{n}(x)=\sum_{i=0}^{n}c_{i}\xi_{i}x^{i}
$$
where $\xi_{i}$ are independent random variables and $c_{i}$ are deterministic coefficients satisfying the following conditions for some positive constants $N_{0}, \tau_{1}, \tau_{2}, \varepsilon$ and some constant $\rho>-1/2$.
   
\begin{enumerate}[label=(A\arabic{*}), ref=A\arabic{*}]
\item  \label{cond_xi} The random variables $\xi_{i}$ are independent (but not necessarily identically distributed) real-valued random variables with unit variance and bounded $ (2+\varepsilon)$ moments, namely $\E |\xi_{i}|^{2+\varepsilon}\leq\tau_{2}$, 
\item   \label{cond_mean_0}$\E \xi_{i}=0$ for all $i\geq N_{0},$

\item  \label{cond_ci_rho} The coefficients $c_{i}$ are deterministic real numbers that grow polynomially, namely
$$|c_{i}|\leq\tau_{2} \quad\text{for all } 0\leq i<N_{0},$$
and
$$\tau_{1}i^{\rho}\leq|c_{i}|\leq\tau_{2}i^{\rho} \quad\text{for all }N_{0}\leq i<n.$$
\end{enumerate}

This class contains many interesting ensembles of polynomials including

\begin{itemize} 
	
	\item the Kac polynomial (all $c_i=1$)
	
	\item semi-Kac polynomials with  $c_{\delta n} = \dots =c_n =1 $ (for some constant $0 <  \delta  < 1$ ) and all other $c_i$  taking arbitrary values from  a fixed set of non-zero constants. (For example, we can have 
	$c_{n/2} = \dots =c_n=1$ and all $c_i, i < n/2$ are either 2 or 3 in arbitrary fashion.) 
	
	\item all derivatives of the Kac polynomial (the zeroes of these polynomials are thus the critical points of the Kac polynomial),
	
	
	\item hyperbolic polynomials $P_n(x) = \sum_{i=0}^{n}\sqrt{\frac{L(L+1)\dots(L+i-1)}{i!}}\xi_i x^{i}$ where $L$ is a positive constant (see \cite{HKPV, DOV, flasche2018real} and the references therein).
	
	\item $c_i$ has the form $f(i) +  g(i)$ where $f(i)$ is a polynomial in $i$ of a fixed degree $d >0$ and $g(i) $ is any function satisfying $|g(i) | = o( |f(i)| ) $.

		\end{itemize}

Our main result establishes the CLT for these random polynomials.

\begin{theorem}\label{thm:CLT_general}
	Assume that the polynomial $P_n$ satisfies Conditions \eqref{cond_xi}-\eqref{cond_ci_rho} and that $\Var N_n(\R) \ge c \log n$ for some constant $c>0$. 
	Then   $  \frac{N_{n}(\R)-\mu_{n}}{\sigma_{n}} \xrightarrow{\text{ d }} \mathcal N(0,1) $ where $\mu_{n}=\E N_{n}(\R)$, $\sigma_{n}=\sqrt {\Var N_n(\R)}$.
	\end{theorem}

The condition  $\Var N_n(\R) \ge c\log n $ is guaranteed  by the following lemma. 

\begin{lemma} \label{lemma:variance}  
	Assume that the polynomial $P_n$ satisfies Conditions \eqref{cond_xi}-\eqref{cond_ci_rho} and there exist constants $C, \ep >0$ such that for all 
	$i\in \left [n-n\exp\left (-\log^{1/5}n\right ), n-\exp\left (\log^{1/5}n\right )\right ]$,
\begin{equation}\label{cond:ci:kac}
\frac{|c_{i}|}{|c_n|} - 1\le C\exp\left (-\left (\log\log n\right )^{1+\varepsilon}\right ).\end{equation}
Then  $\Var N_n(\R) \ge c\log n $ for some constant $c>0$. 
\end{lemma}

The condition in this lemma is satisfied by all classes listed above. We obtain

\begin{corollary} The CLT holds for the Kac polynomial and its derivatives.  It also holds for hyperbolic polynomials. 
\end{corollary}

\begin{remark}  When restricted to  the Kac polynomial with $\xi_i$ being iid copies of an atom variable $\xi$, our result strengthens  Maslova's, as  the condition  $\P (\xi=0) =0$ in Theorem \ref{theorem:Maslova} is removed. 
\end{remark}

\begin{remark}  The real roots of $P_n (x) $ are the real solutions of the equation  $P_n(x)=0$. The flexibility in condition (A2) enables to 
	extend Theorem \ref{thm:CLT_general}  to the equation $P_n(x)= H(x)$, where $H(x)$ is any fixed polynomial with bounded degree. 
	In particular, taking $H(x)= L$ for a constant $L$, we conclude that CLT holds for any level set of $P_n$. 
\end{remark} 

{\bf Related literature.} Random polynomials with coefficients growing polynomially, also known as \textit{generalized Kac polynomials}, have attracted researches in different areas including Probability and Mathematical Physics. For example, we refer to Das \cite{das1972real}, Schehr--Majumdar \cite{schehr2007statistics, MP}, Do and the authors \cite{DOV}. It has been established in \cite{schehr2007statistics, MP} that the real roots of these polynomials are closely connected to zero crossing properties of the diffusion equation with random initial conditions. The connection has been applied in a paper by Dembo and Mukherjee \cite{dembo2015no} to study the probability that these random polynomials have no real roots, which is known as persistence probability as it is related to persistence properties of physically systems. See also \cite{dembo2002random} and the references therein. In \cite{dembo2015no}, the random variables are Gaussian and the persistence probability is $n^{-b+o(1)}$. It is shown that the power exponent $b$ is universal roughly in the sense that it depends on $\rho$ (in Condition \eqref{cond_ci_rho}) rather than the specific values of $c_i$. We also refer to Poplavskyi  and Schehr \cite{poplavskyi2018exact} for a recent development in finding the power exponent for the Kac polynomial. It would be interesting to see if the power exponent is universal in the sense presented in this paper that is if we replace the Gaussian distribution by other distributions.

{\bf Notations.}
We use standard asymptotic notations under the assumption that $n$ tends to infinity.  For two positive  sequences $(a_n)$ and $(b_n)$, we say that $a_n \gg b_n$ or $b_n \ll a_n$ if there exists a constant $C$ such that $b_n\le C a_n$. If $|c_n|\ll a_n$ for some sequence $(c_n)$, we also write $c_n\ll a_n$. 

If $a_n\ll b_n\ll a_n$, we say that $b_n=\Theta(a_n)$.  If $\lim_{n\to \infty} \frac{a_n}{b_n} = 0$, we say that $a_n = o(b_n)$.  If $b_n\ll a_n$, we sometimes employ the notations $b_n = O(a_n)$ and $a_n = \Omega(b_n)$ to make the idea intuitively clearer or the writing less cumbersome; for example, if $A$ is the quantity of interest, we may write $A = A'+ O(B)$ instead of $A - A' \ll B$, and $A = e^{O(B)}$ instead of $\log A\ll B$.

\section{The Universality Method}

The key ingredient of our proof is  the universality method. The general idea of this method is to show that  limiting laws do not depend too much on the distribution of the atom variable $\xi$ (or the variables $\xi_i$ in general, if they are not iid). 
Once universality has been established, then it suffices to prove the desired law for the case in which the $\xi_i$ are Gaussian, and here one can bring extra powerful tools such as 
  properties of Gaussian processes;  see
 \cite{HKPV, EK, MP, prosen1996, sodin2005zeroes, TVpoly, GKZ, sodin2005zeroes, ibragimov2013distribution, kabluchko2014asymptotic, pritsker1, pritsker2}.

The universality method has been powerful in studying local statistics such as the density or correlation functions concerning  the number of roots in a small region (where the expectation is of order $\Theta (1)$) (see, for example, \cite{TVpoly, NNV, DOV, nguyenvurandomfunction17}).  However, universality arguments are  tailored to the local settings and in order to use them to prove the global law in this paper, we need to perform a number of considerably technical 
steps, linking local statistics to the global one.  The proof for the Gaussian case itself also requires new ideas.

To study the real roots of  $P_n$, we divide the real line into two regions: a core region that contains most of the real roots and the remaining one that contains an insignificant number of real roots. Consider small numbers $0\le b_n<a_n<1$ that depend on $n$ and satisfy the following property for all constants $A>0$:
\begin{equation}\label{cond-a-log}
a_n \ll \log^{-A}n. 
\end{equation}
For example, $a_n = \exp\left (-(\log n)^{1/5}\right )$.
We define
\begin{equation}\label{def:mj}
\mj := \mj_{a_n, b_n}:=\pm (1-a_n, 1-b_n) \cup \pm (1-a_n, 1-b_n)^{-1}
\end{equation}
where for any given set $S$, we define $-S:= \{-x: x\in S\}$, $S^{-1} := \{x^{-1}: x\in S\}$, and $\pm S := -S\cup S$. 
For appropriate choices of $a_n$ and $b_n$, this will be our core region.

For a subset $S\subset \C$, let $N_n(S)=N_{P_n}(S) $ be the number of roots of $P_{n}$ in $S$. Let   $\tilde{\xi}_{i}$ be iid  standard Gaussian random variables and set 
$$\tilde{P}_{n}=\sum_{i=0}^{n}c_{i}\tilde{\xi}_{i}x^{i}.$$ 
We denote by $\tilde{N}_{n}(S)=N_{\tilde P _n}(S)$ the number of zeros of $\tilde{P}_n$ in $S$.

Our main result on global universality of the real roots states that on the core $\mj$, the distributions of the roots of $P_n$ and $\tilde P_n$ are approximately the same. 
\begin{theorem} \label{thm:distribution_universality}
Assume that the polynomial $P_n$ satisfies Conditions \eqref{cond_xi}-\eqref{cond_ci_rho}. There exist positive constants $C$ and $c$  such that for every $0\le b_n<a_n<1$ satisfying \eqref{cond-a-log}, for sufficiently large $n$ and every function $F:\mathbb{R}\rightarrow \mathbb{R}$ whose derivatives up to order $3$ are bounded by $1$, we have
$$
\left|\E F(N_n (\mj))-\E F\left (\tilde N_n(\mj)\right )\right|\leq Ca_n^{c}+ Cn^{-c}.
$$
\end{theorem}

Since $N_n(\mj)$ is always an integer, for every real number $a_0\in \mathbb R$, 
$$
\P \left (N_n(\mj)\leq a_{0}\right )=\P \left (N_n(\mj)\leq\lfloor a_{0}\rfloor\right )=\E (F(N_n(\mj)))
$$
where $F$ is any smooth function that takes values in $[0, 1]$ and $\mathbf{1}_{(-\infty,\lfloor a_{0}\rfloor]}\leq F\leq \mathbf{1}_{(-\infty,\lfloor a_{0}\rfloor+1)}$. Therefore, Theorem \ref{thm:distribution_universality} implies
\begin{equation}\label{eq:dist:N:tildeN}
\left|\P \left (N_n(\mj)\leq a_{0}\right )-\P \left (\tilde N_n(\mj)\leq a_{0}\right )\right|\leq Ca_n^{c} + Cn^{-c}.
\end{equation}

Using Theorem \ref{thm:distribution_universality} (not in the straightforward way), we  deduce the following corollary 

\begin{corollary} \label{cor:moment_universality}
Assume that the polynomial $P_n$ satisfies Conditions \eqref{cond_xi}-\eqref{cond_ci_rho}. Let $k\geq 1$ be an integer. There exist positive constants $C$ and $c$  such that for every $0\le b_n<a_n<1$ satisfying \eqref{cond-a-log} and for sufficiently large $n$, we have
$$
\left|\E \left (N_n^{k}(\mj)\right )-\E \left (\tilde N_n^{k}(\mj)\right )\right|\leq Ca_n^{c}+ Cn^{-c}.
$$
 In particular,
$$\left|\Var \bigg (N_n(\mj)\bigg ) -\Var  \bigg (\tilde N_n(\mj)\bigg )\right |\leq Ca_n^{c}+ Cn^{-c}.$$
\end{corollary}

Next, we show that the contribution outside of the core is negligible.
\begin{proposition}\label{prop:boundedge}
	Assume that the polynomial $P_n$ satisfies Conditions \eqref{cond_xi}-\eqref{cond_ci_rho}. Let $k\ge 2$ be an integer. There exists a positive constant $C$ such that for every $0\le b_n<a_n<1$ satisfying \eqref{cond-a-log} and for sufficiently large $n$, we have
\begin{equation}\label{eq:prop:boundedge}
	\E N_{n}^{k}\left (\R\setminus \mj\right )\le \begin{cases}
C\left ( \left (\log a_n \right )^{2k} + \log^{k} (nb_n)\right )\quad \text{if } b_n\ge 1/n,\\
C \left (\log a_n \right )^{2k} \quad \text{if } b_n < 1/n.
\end{cases}
\end{equation}
 \end{proposition}

To prove Theorem \ref{thm:CLT_general} and Lemma \ref{lemma:variance}, we use the universality results stated in Theorem \ref{thm:distribution_universality}, Corollary \ref{cor:moment_universality} and Proposition \ref{prop:boundedge} to reduce to the Gaussian case (i.e., the case in which the $\xi_i$ are iid standard Gaussian) with 
roots restricted to  the core $\mj$. In particular, we prove

\begin{lemma} \label{lm:CLT_gau}
Assume that the polynomial $P_n$ satisfies Conditions \eqref{cond_xi}-\eqref{cond_ci_rho}. Let $c<1$ be any positive constant, then for any $a_n, b_n$ satisfying 
\begin{equation}\label{eq:cond:ab:CLT}
	(\log n)^{2}/n\le b_n<a_n\le \exp\left (-(\log n)^{c}\right ), \quad \log\frac{a_n}{b_n}=\Theta(\log n), \quad \text{and} \quad \Var \tilde N_n(\mj) \gg \log n,
\end{equation}
we have
$$\frac{\tilde{N}_{n}(\mj)-\E \tilde{N}_{n}(\mj)}{\sqrt{\Var\tilde{N}_{n}(\mj)}}   \xrightarrow{\text{ d }}  \mathcal N(0,1).$$ 
\end{lemma}
And we also prove the following special case of Lemma \ref{lemma:variance} for Gaussian.
\begin{lemma} \label{lm:var_gau}
	Assume that the polynomial $P_n$ satisfies Conditions \eqref{cond_xi}-\eqref{cond_ci_rho} and there exist constants $C, \ep >0$ such that for all 
	$i\in \left [n-n\exp\left (-\log^{1/5}n\right ), n-\exp\left (\log^{1/5}n\right )\right ]$,
	\begin{equation} 
		\frac{|c_{i}|}{|c_n|} - 1\le C\exp\left (-\left (\log\log n\right )^{1+\varepsilon}\right ).\nonumber
	\end{equation}
	Then
\begin{equation}
 \Var\tilde{N}_{n}(\R) \gg \log n.\nonumber
\end{equation}
\end{lemma}

To  illustrate the method of universality, we include here the short proof of Theorem \ref{thm:CLT_general} and Lemma \ref{lemma:variance} assuming the Gaussian case (Lemma \ref{lm:CLT_gau} and Lemma \ref{lm:var_gau}) together with the universality results (Corollary \ref{cor:moment_universality} and Proposition \ref{prop:boundedge}). 
\begin{proof}[Proof of  Lemma \ref{lemma:variance}]
	We first choose $a_n$ and $b_n$ that satisfy all the conditions in Corollary \ref{cor:moment_universality} and make the right-hand side of \eqref{eq:prop:boundedge} as small as $o(\log n)$ when $k=2$. In particular, we let 
	$$a_n = \exp\left (-\log^{1/5}n\right), \quad  b_n = \frac{1}{na_n},$$
	and
	\begin{equation} 
	\mj =\pm (1-a_n, 1-b_n) \cup \pm (1-a_n, 1-b_n)^{-1}\nonumber.
	\end{equation}
	By the triangle inequality on the 2-norm, we obtain
	\begin{equation}\label{eq:var:diff}
	\left |\sqrt{\Var N_n(\R)}-\sqrt{\Var N_n(\mj)} \right |\le \sqrt{\Var N_n(\R\setminus \mj)} \le \sqrt{\E N_n^{2}(\R\setminus \mj)} = o\left (\sqrt{\log n}\right )
	\end{equation}
	where in the last equation, we used Proposition \ref{prop:boundedge}.
	Since $\tilde P_n$ is just a special case of $P_n$ (where the random variables $\xi_i$ are iid Gaussian), we also have
	\begin{equation} 
	\left |\sqrt{\Var \tilde N_n(\R)}-\sqrt{\Var \tilde N_n(\mj)} \right |=o\left (\sqrt{\log n}\right ).\nonumber
	\end{equation}
	Combining this with Lemma \ref{lm:var_gau}, we obtain
\begin{equation}\label{eq:var:mj:gau}
\sqrt{\Var \tilde N_n(\mj)} = \sqrt{\Var \tilde N_n(\R)} + o\left (\sqrt{\log n}\right )\gg \sqrt{\log n}.
\end{equation}
Applying Corollary \ref{cor:moment_universality} and \eqref{eq:var:mj:gau} yields
$$\Var N_{n}(\mj)=\Var \tilde N_{n}(\mj) + O(a_n^{c}) = \Var \tilde N_{n}(\mj) + o(\log n) \gg \log n.$$
From this and \eqref{eq:var:diff}, 
$$\sqrt{\Var N_n(\R)} = \sqrt{\Var N_n(\mj)}  + o\left (\sqrt{\log n}\right )\gg \sqrt{\log n}.$$
 This completes the proof.
\end{proof}

\begin{proof}[Proof of Theorem \ref{thm:CLT_general}]
	Let $a_n, b_n$ and $\mj$ be as in the proof of Lemma \ref{lemma:variance}. By the assumption that $\sigma_n = \sqrt{\Var N_n(\R)}\gg \sqrt{\log n}$ and by \eqref{eq:var:diff}, we have
	$$\sqrt{\Var N_n(\mj)} = \sigma_n(1+o(1))\gg \sqrt{\log n}.$$
	 By this and Corollary \ref{cor:moment_universality}, we also have 
	 \begin{equation}\label{eq:verify}
	 \sqrt{\Var \tilde N_n(\mj)} = \sqrt{\Var N_n(\mj)} + o(1) = \sigma_n(1+o(1))\gg \sqrt{\log n}.
	 \end{equation}
	 Thus, \eqref{eq:cond:ab:CLT} holds and so we can apply Lemma \ref{lm:CLT_gau} to get
	$$
	\frac{\tilde N_{n}(\mj)-\E \tilde{N}_{n}(\mj)}{\sqrt{\Var \tilde{N}_{n}(\mj)}}\xrightarrow{\text{ d }}\mathcal{N}(0,1) .
	$$
	Hence,
	$$
	\frac{N_{n}(\mj)-\E \tilde{N}_{n}(\mj)}{\sqrt{\Var \tilde{N}_{n}(\mj)}}\xrightarrow{\text{ d }}\mathcal{N}(0,1)
	$$
	because by \eqref{eq:dist:N:tildeN}, for any fixed $a\in \R$,
	$$\P\left (\frac{N_{n}(\mj)-\E \tilde{N}_{n}(\mj)}{\sqrt{\Var \tilde{N}_{n}(\mj)}}\le a\right ) = \P\left (\frac{\tilde N_{n}(\mj)-\E \tilde{N}_{n}(\mj)}{\sqrt{\Var \tilde{N}_{n}(\mj)}}\le a\right ) + o(1) \xrightarrow [n\to\infty]{} \P(\mathcal{N}(0,1) \le a).$$
	By Corollary \ref{cor:moment_universality}, 
	$$\E N_{n}(\mj)-\E \tilde{N}_{n}(\mj) = o(1).$$

	Combining these with \eqref{eq:verify}, we get
\begin{equation}\label{clt:mj}
\frac{N_{n}(\mj)-\E N_{n}(\mj)}{\sigma_{n}}\xrightarrow{\text{ d }}\mathcal{N}(0,1) .
\end{equation}
	From Proposition \ref{prop:boundedge}, we have
$$\E N_{n}(\R\setminus \mj) \ll \log^{2/5}n.$$
By Markov's inequality, for any fixed $a>0$, we have
\begin{eqnarray}
	\P\left (\left |\frac{N_{n}(\R\setminus \mj)-\E N_{n}(\R\setminus \mj)}{\sigma_{n}}\right |\ge a\right ) &\le& \frac{1}{a\sigma_n} \E\left |N_{n}(\R\setminus \mj)-\E N_{n}(\R\setminus \mj)\right |\nonumber\\
	& \ll& \frac{\log^{2/5}n}{a\log^{1/2}n}\xrightarrow [n\to\infty]{} 0.\nonumber
\end{eqnarray}
	Thus,
\begin{equation}\label{clt:mj:0}
\frac{N_{n}(\R\setminus \mj)-\E N_{n}(\R\setminus \mj)}{\sigma_{n}}\xrightarrow{\text{ d }}0.
\end{equation}
Adding \eqref{clt:mj} and \eqref{clt:mj:0} completes the proof. 
 	\end{proof}

 In Section \ref{sec:prop:var_gau},  we use universality again to prove Lemma \ref{lm:var_gau}. But in this case, we will  reduce general coefficients $c_i$ to the case
 when  $c_i=1$. In other words,  we  could swap random variables with different means or variances. This deviates   significantly from standard swapping arguments that swap random variables with the  same mean and variance.

The rest of the paper is organized as follows. Section \ref{sec:distribution_universality} is devoted to the proof of Theorem \ref{thm:distribution_universality}, Section \ref{sec:cor:moment_universality} for Corollary \ref{cor:moment_universality}, Section \ref{sec:prop:boundedge} for Proposition \ref{prop:boundedge}, Section \ref{sec:lm:CLT_gau} for Lemma \ref{lm:CLT_gau}, and Section \ref{sec:prop:var_gau} for Lemma \ref{lm:var_gau}.

 \section{Proof of Theorem \ref{thm:distribution_universality}} \label{sec:distribution_universality}
Under the hypothesis of Theorem \ref{thm:distribution_universality}, we need to show that
\begin{equation}\label{eq:f1:1}
\left|\E F(N_n (\mj))-\E F\left (\tilde N_n(\mj)\right )\right|\leq Ca_n^{c}+ Cn^{-c}.
\end{equation}

We first restrict to the interval $(0, 1)$ and prove that
\begin{equation}\label{eq:f1:2}
\left|\E F(N_n (\mj\cap (0, 1)))-\E F\left (\tilde N_n(\mj\cap (0, 1))\right )\right|\leq Ca_n^{c}+ Cn^{-c}.
\end{equation}

The proof of \eqref{eq:f1:1} follows from the same arguments with some (merely technical) modifications explained in Section \ref{section:reduction:01}.  We choose to start by presenting the proof of \eqref{eq:f1:2} as it already captures all of the ideas without having to deal with the tedious technical and notational complications detailed in Section \ref{section:reduction:01}. This way, it would make the proofs clearer and easier to follow.

\subsection{Partition into dyadic intervals and preliminary results} \label{subsec:partition}

Recall that 
$$\mj\cap (0, 1) = (1-a_n, 1-b_n).$$
We start by partitioning the interval $(1-a_n, 1-b_n)$ into dyadic intervals: $(1-a_n, 1-a_n/2), [1-a_n/2, 1-a_n/4), \dots$ To be more specific, let $\delta_{i}:=a_n/2^{i}$ for $i=0, \dots, M-1$ where $M$ is the smallest number such that $a_n/2^{M}\leq \max\{1/n, b_n\}$. Let $\delta_M := \max\{1/n, b_n\}$. Note that $M\ll \log n$. For each $i\le M-1$, let $N_{i}$ be the number of real roots of $P_{n}$ in the interval $[1-\delta_{i-1}, 1-\delta_{i})$. Let $N_{M}$ be the number of real roots of $P_{n}$ in the interval $[1-\delta_{M-1}, 1-b_n)$. We have, $N_{n} (1-a_n, 1-b_n)=N_{1}+\cdots+N_{M}$.\quad \footnote{ If $a_n\le 1/n$, we set $M = 1, \delta_0 = \delta_1 = 1/n$ and $N = N_1$ to be the number of real roots of $P_n$ in the interval $(1-a_n, 1-b_n)$.\newline Generally, there is no difference in our proof if an interval of interest includes one of its endpoints or not. So, for example, if one cares about $N_n[1-a_n, 1-b_n)$ instead of $N_n(1-a_n, 1-b_n)$, one can use the exact same analysis.}

For a dyadic interval $(1 - \delta, 1 - \delta/2)$, we can control the moments of the number of roots. More generally, the following result works not just for dyadic intervals but also for balls on the complex plane.
	\begin{lemma}[Bounded number of roots]\label{lm:bddness} For any positive constants $A$ and $k$, there exists a constant $C$ such that for every $n \ge C$, every $ 1/n\le \delta\le 1/C$ and $z\in \C$ with $1-2\delta\le |z|\le 1-\delta+1/n$, we have
	\begin{equation}\label{eq:tail}
	\P\left  (N_{n}\left (B\left (z, \delta/2\right )\right )\ge C\log (1/\delta)\right )\le C\delta^{A},
	\end{equation}
	and
	\begin{equation}
	\E N_{n}^{k}\left (B\left (z, \delta/2\right )\right )\le C\log^{k} (1/\delta)\label{eq:secondmoment}
	\end{equation}
			where $B(z, R)$ is the disk with center $z$ and radius $R$ in the complex plane.
\end{lemma}

As a consequence, for $ 1/n\le \delta\le 1/C$ and for the dyadic interval $[1 - \delta, 1 - \delta/2]$, applying Lemma \ref{lm:bddness} for $z = 1 - 3\delta/2$, we obtain
\begin{equation} 
\P\left  (N_{n}\left (1 - \delta, 1 - \delta/2\right )\ge C\log (1/\delta)\right )\le C\delta^{A},\nonumber
\end{equation}
and
\begin{equation}
\E N_{n}^{k}\left (1 - \delta, 1 - \delta/2\right )\le C\log^{k} (1/\delta).\nonumber
\end{equation}
\begin{proof}[Proof of Lemma \ref{lm:bddness}]

We shall prove that for a large constant $C$ and for every $a\in [1, n\delta]$, 
\begin{equation}
\P\left (N_{n}\left (B(z, \delta/2)\right ) \ge Ca-C\log \delta\right )\ll a^{-A}\delta^{A}\label{eq:bddness:1}
\end{equation}
where the implicit constant only depends on $A$ and $C$.
Setting $a=1$, we obtain \eqref{eq:tail}. Setting $A = 2k$, letting $a$ run from $1$ to $n\delta$ and using the fact that $N_{n}(B(z, \delta/2))\le n$ with probability 1, we obtain 
\begin{eqnarray}
\E N_{n}^{k}\left (B\left (z, \delta/2\right )\right )&\le& \left (C-C\log \delta\right )^{k}+\sum_{a = 1}^{n\delta} \left (C(a+1)-C\log \delta\right )^{k} \left (a^{-2k}\delta^{2k}\right ) + n^{k}(n\delta)^{-2k}\delta^{2k}\nonumber\\
&\ll& \log^{k} (1/\delta)\nonumber
\end{eqnarray}
where in the first inequality, the first term bounds $\E N_{n}^{k}\left (B\left (z, \delta/2\right )\right )$ on the event $$N_{n}\left (B\left (z, \delta/2\right )\right )\le C-C\log \delta,$$ the second term comes from the events $Ca-C\log \delta\le N_{n}\left (B\left (z, \delta/2\right )\right )< C(a+1)-C\log \delta$ for each $a$, and the third term comes from the event $N_{n}\left (B\left (z, \delta/2\right )\right )\ge C(n\delta+1)-C\log \delta$.
This proves \eqref{eq:secondmoment}, completing the proof. 

It remains to prove \eqref{eq:bddness:1}. To that end, we use the following version of Jensen's inequality which asserts that for every entire function $f$, every $z\in \C$ and $0<r< R$, 
\begin{equation}
N_{f}\left (B(z, r)\right )\le \frac{\log \frac{M_1}{M_2}}{\log \frac{R^{2}+r^{2}}{2Rr}} \label{eq:Jensen}
\end{equation}
where $M_1 = \sup_{w\in B(z, R)}|f(w)|$ and $M_2 = \sup_{w\in B(z, r)}|f(w)|$. This is a consequence of the classical Jensen's formula (see, for example, \cite{Ru}). We add a proof of this inequality in Appendix \ref{app:jensen} for completeness.

Applying Jensen's inequality to the polynomial $P_n$ gives
\begin{equation}
N_{n}\left (B(z, \delta/2)\right )\ll \log \frac{M_1}{M_2}\label{eq:bddness:0}
\end{equation}
where $M_1=\sup_{w\in B(z, 2\delta/3)}|P_n(w)|$ and $M_2 = \sup_{w\in B(z, \delta/2)}|P_n(w)|$.

From \eqref{eq:bddness:0}, to prove \eqref{eq:bddness:1}, it suffices to show that
\begin{equation}
\P\left (M_1\ge \exp\left (Ca-C\log \delta\right )\right )\ll a^{-A}\delta^{A}\label{eq:bddness:2}
\end{equation}
and 
\begin{equation}
\P\left (M_2\le \exp\left (-Ca+C\log \delta\right )\right )\ll a^{-A}\delta^{A}\label{eq:bddness:3}.
\end{equation}

Since 
$$M_1\le \sum_{i=0}^{n} |c_i||\xi_i||z|^{i},$$
it follows that $\E M_1\ll \delta^{-O(1)}$ by Conditions \eqref{cond_xi} and \eqref{cond_ci_rho}. The bound \eqref{eq:bddness:2} then follows from Markov's inequality.

For \eqref{eq:bddness:3}, writing $z = re^{i\theta}$ and observing that the set $\{w=re^{i\theta'}: \theta'\in [\theta-\delta/10, \theta+\delta/10]\}$ is a subset of $B(z, \delta/2)$, we have
\begin{eqnarray}
&&\P\left (M_2 \le \exp\left (-Ca+C\log \delta\right )\right )  \nonumber\\
&& \quad \le\P\left (\sup_{\theta'\in [\theta-\delta/10, \theta+\delta/10]}  \left |\sum_{j=0}^{n} c_j \xi_j r^{j}e^{ij\theta'}  \right | \le \exp\left (-Ca+C\log \delta\right )\right ) \nonumber.
\end{eqnarray}
By taking the supremum outside, the right-hand side is at most 
$$\sup_{\theta'\in [\theta-\delta/10, \theta+\delta/10]}\P\left (\left |\sum_{j=0}^{n} c_j \xi_j r^{j}e^{ij\theta'}\right | \le \exp\left (-Ca+C\log \delta\right )\right )$$
and hence, by projecting onto the real line and conditioning on the random variables $(\xi_j)_{j\notin [1, a/\delta]}$, it is bounded by 
$$\sup_{\theta'\in [\theta-\delta/10, \theta+\delta/10]}\sup_{Z\in \R}\P\left (\left |\sum_{j=1}^{a/\delta} c_j \xi_j r^{j}\cos(j\theta')-Z\right | \le \exp\left (-Ca+C\log \delta\right )\right ).$$

We use the following anti-concentration lemma from \cite{nguyenvurandomfunction17}.
\begin{lemma}\cite[Lemma 9.2]{nguyenvurandomfunction17} \label{lmanti_concentration}
	Let $\mathcal E$ be an index set of size $M\in \mathbb N$, and let $(\xi_j)_{j\in \mathcal E}$ be independent random variables satisfying Condition \eqref{cond_xi}. Let $(e_j)_{j\in \mathcal E}$ be deterministic (real or complex) coefficients with $|e_j|\ge \bar e$ for all $j$ and for some number $\bar e\in \R_+$. Then for any constant $B\ge 1$, any interval $I\subset\R$ of length at least $M^{-B}$, there exists $\theta'\in I$ such that
	$$\sup_{Z\in \R}\P\left (\left| \sum_{j\in \mathcal E} e_j \xi_j \cos(j\theta')-Z\right |\le \bar e M^{-16B^{2}}\right )\ll M^{-B/2},$$  
	where the implicit constant depends only on $B$ and the constants in Condition \eqref{cond_xi}. 
\end{lemma}
Applying Lemma \ref{lmanti_concentration} with $B = 2A, \mathcal E = [1, a/\delta], M = a/\delta$, $I = [\theta-\delta/10, \theta+\delta/10]$, $e_j = c_jr^{j}$ and $\bar e = \frac{\delta}{ae^{3a}}$ (where we use Condition \eqref{cond_ci_rho} and the assumption that $r = |z|\ge 1-2\delta$ to get  $|e_j|\ge \bar e$), we obtain $\theta'\in [\theta-\delta/10, \theta+\delta/10]$ such that for a sufficiently large constant $C$,
\begin{equation}
\sup_{Z\in \R}\P\left (\left |\sum_{j=1}^{a/\delta} c_j \xi_j r^{j}\cos(j\theta')-Z\right | \le \exp\left (-Ca+C\log \delta\right )\right )  \ll (a/\delta)^{-A}= a^{-A}\delta^{A}\nonumber
\end{equation}
which gives \eqref{eq:bddness:3} and completes the proof of Lemma \ref{lm:bddness}.
\end{proof}

\subsection{A generalization}
Theorem \ref{thm:distribution_universality} is deduced from the following more general result that can be of independent interest.
\begin{proposition}\label{prop:univ:1M}
Let $\hat{F}$ : $\mathbb{R}^{M}\rightarrow \mathbb{R}$  be any function whose every partial derivative up to order $3$ is bounded by $1$. We have
$$
\left |\E \hat{F}(N_{1}, \ldots, N_{M})-\E \hat{F}\left (\tilde{N}_{1}, \ldots,\tilde{N}_{M}\right )\right |\ll \delta_0^{c}.
$$
\end{proposition}

To deduce \eqref{eq:f1:2} (which is essentially Theorem \ref{thm:distribution_universality} as mentioned at the beginning of this section), let  $\hat{F}$ be the function defined by $\hat{F}(x_{1}, \ldots, x_{M})=F(x_{1}+\cdots+x_{M})$. It is easy to check that  $\left\Vert\partial^{(3)}\hat{F}\right\Vert_{\infty}\le 1$ where $\left\Vert\partial^{(3)}\hat{F}\right\Vert_{\infty} = \max_{\alpha: |\alpha|\le 3} ||\partial^{\alpha}\hat F||_{\infty}$ being the supremum of all partial derivatives up to order 3 of $F$. By applying Proposition \ref{prop:univ:1M} to this $\hat F$, \eqref{eq:f1:2} follows. \qed

The rest of this section is devoted to the proof of Proposition \ref{prop:univ:1M}.
 
 \subsection{Approximate the indicator function by smooth functions}
 To apply analytical tools, we first approximate the indicator function in counting the number of real roots by smooth functions.

 Let $\alpha$ be a sufficiently small positive constant. Let $\varphi_{0}$ be a smooth function taking values in $[0, 1]$, supported on $[-1, 1]$ and equal $1$ at $0$ with $\Vert \phi_{0}^{(a)}\Vert_{\infty}=O(1)$ for all $0\le a\le 3$. For example, we can take the classical bump function
 $$\varphi_0(x) = \begin{cases}
 \exp\left (-\frac{x^{2}}{1 - x^{2}}\right ), \quad\text{if } x\in (-1, 1)\\
 0, \quad\text{otherwise}.
 \end{cases}$$
 
 For $1\leq i\leq M$, let $\phi_{i}$ be a smooth function taking values in $[0, 1]$, supported on $[1-\delta_{i-1}-\delta_{i}^{1+\alpha}, 1-\delta_{i}+\delta_{i}^{1+\alpha}]$ and equal $1$ on $[1-\delta_{i-1}, 1-\delta_{i}]$ with $\Vert \phi_{i}^{(a)}\Vert_{\infty}=O\left (\delta_{i}^{-a(1+\alpha)}\right )$ for all $0\leq a\leq 3$. An example of $\phi_i$ can be obtained by translating and scaling $\phi_0$ as follows.
 $$\varphi_i(x) = \begin{cases}
 1, \quad\text{if } x\in [1-\delta_{i-1}, 1-\delta_{i}]\\
 0, \quad\text{if } x\in [1-\delta_{i-1}-\delta_{i}^{1+\alpha}, 1-\delta_{i}+\delta_{i}^{1+\alpha}], \\
 \phi_0\left (\frac{x - (1-\delta_i)}{\delta_i^{1+\alpha}}\right ), \quad\text{if } x\in [1-\delta_{i}, 1-\delta_{i}+\delta_{i}^{1+\alpha}], \\
 \phi_0\left (\frac{x - (1-\delta_{i-1})}{\delta_i^{1+\alpha}}\right ), \quad\text{if } x\in [1-\delta_{i-1}-\delta_{i}^{1+\alpha}, 1-\delta_{i-1}]. \\
 \end{cases}$$

The indicator of the dyadic interval  $(1-\delta_{i-1}, 1-\delta_{i}]$ shall be approximated by the following function defined on the complex plane
$$\varphi_{i}(z):=\phi_{i}({\rm Re}(z))\varphi_{0}\left (\frac{\mbox{Im}(z)}{\delta_{i}^{1+\alpha}}\right ).$$ 
In other words, the number of roots in $(1-\delta_{i-1}, 1-\delta_{i}]$, which is just $N_i$, is approximated by 
$$\sum_{j=1}^{n}\varphi_{i}(\zeta_{j})$$
where $(\zeta_{j})_{j=1}^{n}$ be the roots of $P_{n}$.

To control the error terms in this approximation, note that for all $0\leq a\leq 3$, 
\begin{equation}\label{eq:der:varphi}
\left\Vert\partial^{(a)}\varphi_{i}\right\Vert_{\infty}=O\left (\delta_{i}^{-a(1+\alpha)}\right ).
\end{equation}

 The following lemma estimates the error term in approximating $N_i$ by $ \sum_{j=1}^{n}\varphi_{i}(\zeta_{j})$.
 
\begin{lemma} \label{lemma:aproximation1} We have
\begin{equation}\label{eq:smooth1}
 \E \hat{F}(N_{1}, \ldots, N_{M})-\E  \hat F\left(\sum_{j=1}^{n}\varphi_{1}(\zeta_{j}), \ldots,\sum_{j=1}^{n}\varphi_{M}(\zeta_{j})\right) \ll \delta_0^{\alpha/8} .
\end{equation}

\end{lemma} 

\begin{proof}[Proof of Lemma \ref{lemma:aproximation1}] By the derivative assumption on $\hat F$, we have 
	$$
	\hat{F}(N_{1}, \ldots, N_{M})- \hat F\left(\sum_{j=1}^{n}\varphi_{1}(\zeta_{j}), \ldots,\sum_{j=1}^{n}\varphi_{M}(\zeta_{j})\right) \ll  \sum_{i=1}^{M}\left|N_{i}-\sum_{j=1}^{n}\varphi_{i}(\zeta_{j})\right|.
	$$
	
	For each $i\leq M, |N_{i}-  \sum_{j=1}^{n}\varphi_{i}(\zeta_{j})|$ is bounded by the number of roots of $P_{n}$ in the union of the sets $S_{1}, S_2$, $S_3$ where $S_1$ is the set of all complex numbers whose real part lies in $[1-\delta_{i-1}-\delta_{i}^{1+\alpha}, 1-\delta_{i}+\delta_{i}^{1+\alpha}]$ and imaginary part in $[-\delta_{i}^{1+\alpha}, \delta_{i}^{1+\alpha}]\setminus \{0\}$, $S_{2} :=[1-\delta_{i-1}-\delta_{i}^{1+\alpha}, 1-\delta_{i-1}]$, and $S_{3}:=[1-\delta_{i}, 1-\delta_{i}+\delta_{i}^{1+\alpha}].$
	
	Thus, Lemma \ref{lemma:aproximation1} follows by proving that for $k=1, 2, 3,$
	$$\E N_n(S_k)\ll \delta_0^{\alpha/8} .$$
	
	To this end, we use the following lemma from \cite{DOV}.
	\begin{lemma}\label{lm:repulsion} \cite[Lemma 5.1]{DOV}
		There exists a constant $\alpha_0>0$ such that for all $0<\alpha\le \alpha_0$ and all  $x\in \R$ with $|x|\in  [1-\delta_{i-1}-\delta_{i}^{1+\alpha}, 1-\delta_{i}+\delta_{i}^{1+\alpha}]$, 
		$$\P  \left(N_{n}(B(x,2\delta_i^{1+\alpha}))\ge 2\right) \ll  \delta_i^{3\alpha/2}.$$
	\end{lemma}
	
	To show that $\E N_n(S_1)\ll \delta_0^{\alpha/8}$, we note that $S_1$ is contained in a union of $\Theta(\delta_i^{-\alpha})$ small balls of radius $2\delta_i^{1+\alpha}$. By Lemma \ref{lm:repulsion}, the union bound and the fact that the complex roots come in conjugate pairs, the probability that $N_n(S_1)$ is nonzero is in fact negligible
	\begin{eqnarray} 
	\P\left (N_n(S_1)>0\right ) &\le& \sum_{\text{small balls}} \P\left (\text{number of roots in a small ball is at least 2}\right )\nonumber\\
	&\ll& \delta_i^{-\alpha}  \delta_i^{3\alpha/2} = \delta_{i}^{\alpha/2}.\nonumber
	\end{eqnarray}
	
	Thus, $N_{n}(S_{1})=0$ except on an event, named $\mathcal A_1$, of probability at most $O\left (\delta_{i}^{\alpha/2}\right ).$

The expectation on the tail event $\mathcal A_1$ is controlled as the higher moments of $N_n(S_1)$ are bounded by Lemma \ref{lm:bddness}. More specifically, since $S_1\subset B\left (1 - \frac{3\delta_i}{2}, \frac{\delta_i}{2}+ \delta_i^{1+\alpha}\right )$, applying \eqref{eq:secondmoment} to this ball and H\"older's inequality, we obtain 
	\begin{equation}\label{eq:s:1}
	\E N_{n}(S_{1})= \E N_{n}(S_{1})\textbf{1}_{\mathcal A_1}\le \left (\E N_n^{2} (S_1)\right )^{1/2} \left (\P(\mathcal A_1)\right )^{1/2} \ll \delta_i^{\alpha/4}\log \frac{1}{\delta_i} \ll \delta_{i}^{\alpha/8}.
	\end{equation}
	
	For $S_2\cup S_3$, \cite[Theorem 2.4]{DOV} implies that for such intervals as $S_2$ and $S_3$, the expected number of real roots is universal in the sense that 
	$$\E N_{n}(S_{2} )=\E \tilde N_{n}(S_{2} )+O\left (\delta_{i}^{\alpha/2}\right )\quad \text{and}\quad \E N_{n}(S_{3} )=\E \tilde N_{n}(S_{3} )+O\left (\delta_{i}^{\alpha/2}\right )$$ 
	where $\tilde N_n(S)$ is the number of real roots of $\tilde P_n$ in a set $S$. It thus remains to show that $\E \tilde N_{n}(S_{2}\cup S_{3})\ll \delta_0^{\alpha/8}$. To this end, we  use  Kac-Rice formula (see \cite{Kac1943average, EK}; here we use \cite[Formula 3.12]{Far}),
	\begin{eqnarray} \label{eq:KacRice} 
	\E \tilde N_n(a, b) &=&  \frac{1	}{\pi}\int_{a} ^{b}\frac{\sqrt{\Var P_n(t)\Var P'_n(t) - (\Cov (P_n(t), P'_n(t))^{2}}}{(\Var P_n(t))^{2}}dt \nonumber\\
	&=&   \frac{1	}{\pi}\int_{a} ^{b} \frac{\sqrt{\sum_{i=0}^{n}\sum_{j=i+1}^{n} c_i^{2} c_j^{2}(j-i)^{2} t^{2i+2j-2}}}{\sum_{i=0}^{n} c_i^{2}t^{2i}}dt.
	\end{eqnarray}    
	Algebraic manipulations show that 
	\begin{equation}\label{eq:S23}
	\E \tilde N_{n}(S_{2}\cup S_{3})\ll \delta_{i}^{\alpha/2}.
	\end{equation}
	We add the verification of this estimate in Appendix \ref{app:proof:S23} for completeness.
	Putting the bounds together gives
	$\E  N_{n}(S_{2}\cup S_{3})\ll \delta_{i}^{\alpha/2}$.
	
	By combining this with \eqref{eq:s:1}, it follows that  the left-hand side of \eqref{eq:smooth1} is bounded by $O\left(  \sum_{i=1}^{M}\delta_{i}^{\alpha/8}\right)=O\left (\delta_0^{\alpha/8}\right )$, proving \eqref{eq:smooth1} and Lemma \ref{lemma:aproximation1}.
	\end{proof}

\subsection{Reducing to an explicit function of the random polynomial: log$|P_n|$}\qquad\newline 
 In light of Lemma \ref{lemma:aproximation1}, to show Proposition \ref{prop:univ:1M}, it remains to show that
\begin{equation}\label{smooth2}
\E \hat F \left(\sum_{j=1}^{n}\varphi_{1}(\zeta_{j}), \ldots,\sum_{j=1}^{n}\varphi_{M}(\zeta_{j})\right)=\E \hat F\left(\sum_{j=1}^{n}\varphi_{1}\left (\tilde{\zeta}_{j}\right ), \ldots,\sum_{j=1}^{n}\varphi_{M}\left (\tilde{\zeta}_{j}\right ) \right )+O(\delta_0^{\alpha})
\end{equation}
where $\tilde{\zeta}_{j}$ are the roots of $\tilde{P}_{n}.$

In this section, we reduce the sums $ \sum_{j=1}^{n}\varphi_{i}(\zeta_{j})$ to an explicit function of $P_n$. The starting point for this reduction is to apply the Green's second identity to the compactly supported function $\varphi_i$
\begin{equation}
\sum_{j=1}^{n}\varphi_i (\zeta_{j})= \frac{1}{2\pi}\int_{\mathbb C}\log |P_n(z)| \triangle \varphi (z)dz  =  \frac{1}{2\pi}\int_{B(1- 3\delta_{i}/2, 2\delta_{i}/3)}\log |P_n(z)|\triangle \varphi (z)dz\label{eq:green:0}
\end{equation}
where we note that $\varphi_i$ is supported on $B(1- 3\delta_{i}/2, 2\delta_{i}/3)$.

We show that the integral on the right-most side is well approximated by its Riemann sum. In particular, we prove that for $m_{i}:=\delta_{i}^{-11\alpha},$
\begin{equation}\label{rc1}
\left|  \sum_{j=1}^{n}\varphi_{i}(\zeta_{j})-\frac{2\delta_{i}^{2}}{9m_{i}}\sum_{k=1}^{m_{i}}\log|P_n(w_{ik})|\triangle\varphi_{i}(w_{ik})\right|=O(\delta_i^{\alpha})
\end{equation}
with probability at least $1-O\left (\delta_i^{\alpha}\right )$, where $w_{ik}$ are chosen independently, uniformly at random from the ball $B(1- 3\delta_{i}/2, 2\delta_{i}/3)$ and are independent of all previous random variables. 

\begin{proof}[Proof of equation \eqref{rc1}] This proof is based on \cite[Equation (4.20)]{DOV}.

For notational convenience, we skip the subscript $i$ and write $\delta :=\delta_i$, $\varphi:=\varphi_i$, and $m:=m_i$. Let $x_0 = 1-3\delta_i/2$, the center of the ball.

Since $\varphi$ is compactly supported in $B(x_0, 2\delta/3)$, by the Green's second identity, it holds that
\begin{equation}
\sum_{j=1}^{n}\varphi (\zeta_{j})= \frac{1}{2\pi}\int_{\mathbb C}\log |P_n(z)| \triangle \varphi (z)dz  =  \frac{1}{2\pi}\int_{B(x_0, 2\delta/3)}\log |P_n(z)|\triangle \varphi (z)dz.\label{eq:green}
\end{equation}
 
 We shall think about the integral on the right-most side of \eqref{eq:green} as an expectation with respect to $dz$, up to some rescale; so in approximating an expectation by a sample mean (which is the second summation in \eqref{rc1}), it is sufficient to control the variance. To this end, we would need to require that $\log|P_n(z)|$ is bounded above and below. For that purpose, we introduce the good event  $\mathcal T$ on which the following hold for $c_1 = \alpha/2$:
\begin{enumerate} [label=(T\arabic{*}), ref=T\arabic{*}]
	\item \label{Tupper}$\log|P_n(z)|\le \frac{1}{2}\delta^{-c_1}$ for all $z \in B\left ( x_0, \frac{4\delta}{5}\right ) $.
	\item \label{Tlower}$\log|P_n(x_1)|\ge -\frac{1}{2} \delta^{-c_1}$ for some $x_1\in B\left ( x_0, \frac{\delta}{100}\right ) $.
\end{enumerate}
By Jensen's inequality \eqref{eq:Jensen}, these conditions imply 
\begin{equation}\label{Troot}
N_n \left (B\left ( x_0, \frac{3\delta}{4}\right ) \right )\ll  \delta^{-c_1}.
\end{equation}
We will show later that 
\begin{equation}\label{eq:prob:T}
\P\left (\mathcal T\right ) = 1-O\left (\delta^{\alpha}\right ).
\end{equation}
Assuming \eqref{eq:prob:T}, it suffices to show that \eqref{rc1}, conditioned on $\mathcal T$, holds with probability $1-O\left (\delta^{\alpha}\right )$. The following lemma provides the required variance bound, conditioned on $\mathcal T$. 

\begin{lemma}\label{f5} On the event $\mathcal T$, we have
	\begin{equation}\label{norm2}
	\int_{B( x_0, 2\delta/3)}\left (\log |P_n(z)|\right )^{2}dz\ll \delta^{-8c_1+2}.
	\end{equation}
\end{lemma}

Assuming Lemma \ref{f5} and the fact that $\left\Vert \triangle \varphi \right\Vert_{\infty}\ll \delta ^{-2(1+\alpha)}$ by the definition of $\varphi$, we conclude that on the event $\mathcal T$,
\begin{equation} \label{eq:variance:MC}
\dashint_{B(x_0, 2\delta/3)}\log^{2} |P_n(z)|\cdot |\triangle \varphi(z)|^{2}dz \ll \delta^{-4-8\alpha} 
\end{equation}
where $\dashint_{B(x_0, 2\delta/3)} f(z)dz: = \frac{1}{|{B(x_0, 2\delta/3)}|}\int_{B(x_0, 2\delta/3)} f(z)dz$ is the average of $f$ on the domain of integration.

Having bounded the $2$-norm, we now use the following sampling lemma which is a direct application of Chebyshev's inequality.
\begin{lemma}[Monte Carlo sampling Lemma] (\cite[Lemma 38]{TVpoly})
	Let $(X, \mu)$ be a probability space, and $F: X\to \C$ be a square integrable function. Let $m\ge 1$, let $x_1, \dots, x_m$ be drawn independently at random from $X$ with distribution $\mu$, and let $S$ be the empirical average 
	$$S: =  \frac{1}{m}\left (F(x_1)+\dots+ F(x_m)\right ).$$
	Then $S$ has mean $\int_{X} Fd\mu$ and variance $\frac{1}{m}\int_{X} \left (F-\int_{X} Fd\mu\right )^{2}d\mu$. In particular, by Chebyshev's inequality, we have
	$$\P\left (\left |S-\int_{X} Fd\mu\right |\ge \lambda\right )\le \frac{1}{m\lambda^{2}}\int_{X} \left (F-\int_{X} Fd\mu\right )^{2}d\mu.$$
\end{lemma}

Conditioning on $\mathcal T$ and applying this sampling lemma with $\lambda = \delta^{\alpha-2}$ together with \eqref{eq:variance:MC}, we obtain 
\begin{equation} 
\dashint_{B(x_0, 2\delta/3)}\log |P_n(z)|\triangle \varphi(z)dz - \frac{1}{m}\sum_{k=1}^{m}\log|P_n(w_{ik})|\triangle\varphi(w_{ik}) \ll \delta^{\alpha-2}\nonumber
\end{equation}
with probability at least $1-\frac{\delta^{-4-8\alpha}}{m\delta^{2\alpha-4}} = 1-\delta^{\alpha}$ where we recall that $m = \delta^{-11\alpha}$.

Combining this with \eqref{eq:green} gives \eqref{rc1}, conditioned on $\mathcal T$ as claimed. 
\end{proof}

It is left to verify \eqref{eq:prob:T} and Lemma \ref{f5}.
\begin{proof}[Proof of \eqref{eq:prob:T}] 
	Since
	$$\sup_{z\in B(x_0, 4\delta/5)} |P_n(z)|\le \sum_{i=0}^{n} |c_i||\xi_i|\left (1-  7\delta/10\right )^{i}$$
	has mean at most $\delta^{-O(1)}$, applying Markov's inequality to the random variable $$\sum_{i=0}^{n} |c_i||\xi_i|\left (1-  7\delta/10\right )^{i},$$ we conclude that the event \eqref{Tupper} happens with probability at least $1 - O_A\left (\delta^A\right )$ for any constant $A>0$.
	
	For \eqref{Tlower}, writing $x_0 = re^{i\theta}$ and observing that the set $\{w = re^{i\theta'}: \theta'\in [\theta-\delta/100, \theta+\delta/100]\}$ is a subset of $B(x_0, \delta/100)$, we have
	\begin{eqnarray}
	&&\P\left (\text{\eqref{Tlower} fails}\right )  \le \P\left (\sup_{\theta'\in [\theta-\delta/100, \theta+\delta/100]}\left |\sum_{j=0}^{n} c_j \xi_j r^{j}e^{ij\theta'}\right | \le \exp\left (-\delta^{-c_1}/2\right )\right ).
	\end{eqnarray}
	
	By taking the supremum outside, the right-hand side is at most
	$$ \sup_{\theta'\in [\theta-\delta/100, \theta+\delta/100]}\P\left (\left |\sum_{j=0}^{n} c_j \xi_j r^{j}e^{ij\theta'}\right | \le \exp\left (-\delta^{-c_1}/2\right )\right )$$
	and hence is bounded by
	$$  \sup_{\theta'\in [\theta-\delta/100, \theta+\delta/100]}\sup_{Z\in \R}\P\left (\left |\sum_{j=1}^{1/\delta} c_j \xi_j r^{j}\cos(j\theta')-Z\right | \le \exp\left (-\delta^{-c_1}/2\right )\right )$$
	by projecting onto the real line and conditioning on the random variables $(\xi_i)_{i\notin [1, 1/\delta]}$.

	Applying the anti-concentration Lemma \ref{lmanti_concentration} with $B = 4\alpha, \mathcal E = [1, 1/\delta], M = 1/\delta$, $I = [\theta-\delta/100, \theta+\delta/100]$, $e_j = c_jr^{j}$ and $\bar e = \delta$ (where we use Condition \eqref{cond_ci_rho} and the assumption that $r = |z|\ge 1-3\delta$ to get  $|e_j|\ge \bar e$), we obtain $\theta'\in [\theta-\delta/100, \theta+\delta/100]$ such that for all $Z\in \C$,
	\begin{equation}
	\P\left (\left |\sum_{j=1}^{1/\delta} c_j \xi_j r^{j}\cos(j\theta')-Z\right | \le \exp\left (-\delta^{-c_1}/2\right )\right )  \ll \delta^{2\alpha}.\nonumber
	\end{equation}
	This proves that \eqref{Tlower} holds with probability at least $1- O\left (\delta^{2\alpha}\right )$, concluding the proof of \eqref{eq:prob:T}. 
\end{proof}

\begin{proof} [Proof of Lemma \ref{f5}] This proof is based on \cite[Lemma 4.8]{DOV}.
	Since $x_1\in B(x_0, \delta/100)$, it suffices to show that 
	\begin{equation}
	\int_{B( x_1, 2\delta/3+\delta/100)}\left (\log |P_n(z)|\right )^{2}dz\ll \delta^{-8c_1+2}.\nonumber
	\end{equation}
	By \eqref{Troot}, there exists an $r\in [2\delta/3+\delta/100, 3\delta/4-\delta/100]$ such that $P_n$ does not have zeros in the (closed) annulus $B(x_1, r+\eta)\setminus B(x_1, r-\eta)$ with center at $x_1$ and radii $r\pm \eta$ where $\eta \gg \delta^{1+c_1}$.  
	
	It is now sufficient to show that 
	\begin{equation}\label{norm2:1}
	\int_{B( x_1, r )}{\log ^2| P_n(z)|}dz  \ll \delta^{-8c_1+2}.
	\end{equation}
	Let $\zeta_1,\dots, \zeta_k$ be all zeros of $P_n$ in $B(x_1, r-\eta)$, then $k \ll \delta^{-c_1}$ and $P_n(z) = (z-\zeta_1)\dots(z-\zeta_k)g(z)$ where $g$ is a polynomial having no zeros on the closed ball $B(x_1, r+\eta)$.
	By triangle inequality,
	\begin{eqnarray}
	\left (\int_{B( x_1, r )}\log ^2| P_n(z)| dz \right )^{1/2}&\le& \sum_{i=1}^k 	\left (\int_{B( x_1, r )}\log ^2|z-\zeta_i|dz \right )^{1/2} + \left (\int_{B( x_1, r )}\log ^2|g(z)| dz \right )^{1/2}\nonumber\\
	&\ll& \delta^{1 - 2c_1}+ \left (\int_{B( x_1, r )}\log ^2|g(z)| dz \right )^{1/2},\label{a1}
	\end{eqnarray}
	where in the last inequality, we used
	\begin{eqnarray}
	\int_{B( x_1, r)}\log^2 |z-\zeta_i|\text{d}z \le \int_{B(0,3\delta/2)} \log ^2 |z|\text{d}z \ll \delta^{2-2c_1}. \nonumber
	\end{eqnarray}

	Next, we will bound $\int_{B( x_1, r)}{\log ^2| g(z)|}dz$ by finding a uniform upper bound and lower bound for $\log |g(z)|$. Since $\log |g(z)|$ is harmonic in $B(x_1, r)$, it attains its extrema on the boundary. Thus, 
	\begin{equation}\label{a2}
	\left (\int_{B( x_1, r )}\log ^2| g(z)| dz \right )^{1/2}  \ll \delta\max_{z\in \partial B(x_1, r)}|\log|g(z)||.
	\end{equation} 
	Notice that $\log|g(z)|$ is also harmonic on the ball $B(x_1, r+\eta)$. For the upper bound of $\log |g(z)|$, we claim that for all $z$ in $B(x_1, r+\eta)$,  
		\begin{eqnarray}
		\log |g(z)| \le \delta^{-2c_1}\label{f3}.
		\end{eqnarray}
 
	Indeed, since a harmonic function attains its extrema on the boundary, we can assume that $z\in\partial B(x_1, r+\eta)$.
		By Condition \eqref{Tupper}, $\log|P_n(z)|\le \delta^{-c_1}$. Additionally, by noticing that $|z-\zeta_i|\ge 2\eta$ for all $1\le i\le k$, we get
		\begin{eqnarray}
		\log |g(z)| = \log|P_n(z)| - \sum_{i=1}^k \log |z-\zeta_i|\le \delta^{-c_1} - k\log(2\eta) \le \delta^{-2c_1}
		\end{eqnarray}
		as claimed.

	As for the lower bound, let $u(z) = \delta^{-2c_1} - \log |g(z)|$, then $u$ is a non-negative harmonic function on the ball $B(x_1, r+\eta)$. By Harnack's inequality (see \cite[Chapter 11]{Ru}) for the subset $B(x_1, r)$ of the above ball, we have that for every $z\in B(x_1, r)$,
	\[\alpha u(x_1)\le u(z)\le \frac{1}{\alpha} u(x_1),
	\]
	where $\alpha =\frac{\eta}{2r+\eta}\gg \delta^{c_1}$.
	Hence,
	\[\alpha \left(\delta^{-2c_1} - \log |g(x_1)|\right)\le \delta^{-2c_1} - \log |g(z)|\le \frac{1}{\alpha} \left(\delta^{-2c_1} - \log |g(x_1)|\right).
	\]
	
	And so,
	\begin{equation}\label{a3}
	|\log |g(z)||\le \frac{1}{\alpha}|\log |g(x_1)||+ \frac{1}{\alpha }\delta^{-2c_1} \ll \delta^{-c_1}|\log |g(x_1)|| + \delta^{-3c_1}.
	\end{equation}
	Thus, we reduce to bounding $|\log |g(x_1)||$.
	From Lemma \ref{f3} and condition \eqref{Tlower}, we have
	\begin{eqnarray}
	\delta^{-2c_1}\ge \log |g(x_1)| &=& \log |P(x_1)| - \sum_{i=1}^k\log |x_1 - \zeta_i|\ge \log |P(x_1)| \ge -\frac{1}{2}\delta^{-c_1}.\nonumber
	\end{eqnarray}
	And so, $|\log|g(x_1)||\le \delta^{-2c_1}$, which together with \eqref{a3} give
	\begin{equation}\label{a3-1}
	|\log |g(z)||\ll \delta^{-3c_1}.
	\end{equation}
	
	From \eqref{a1}, \eqref{a2}, and \eqref{a3-1}, we obtain \eqref{norm2:1} and hence Lemma \ref{f5}. 
\end{proof}

 \subsection{Universality of log $|P_n|$}

Since $\sum_{i=1}^{M} \delta_i^{\alpha}\ll \delta_0^{\alpha}$, by applying \eqref{rc1}, the left-hand side of \eqref{smooth2} equals
\begin{equation}
\E K(\log |P_n(w_{ik})|)_{\substack{i=1,\ldots,M\\k=1,\dots,m_{i}}}+O(\delta_0^{\alpha})\nonumber
\end{equation}
and the right-hand side of \eqref{smooth2} equals 
\begin{equation} 
\E K(\log |\tilde P_n(w_{ik})|)_{\substack{i=1,\ldots,M\\k=1,\dots,m_{i}}}+O(\delta_0^{\alpha})\nonumber
\end{equation}
where 
\begin{equation}\label{eq:K:def}
K(x_{ik})_{\substack{i=1,\ldots,M\\k=1,\dots ,m_{i}}} := \hat F\left(\frac{2\delta_{1}^{2}}{9m_{1}}\sum_{k=1}^{m_{1}} x_{1k}\triangle\varphi_{1}(w_{1k}), \ldots, \frac{2\delta_{M}^{2}}{9m_{M}}\sum_{k=1}^{m_{M}} x_{Mk}\triangle\varphi_{M}(w_{Mk})\right).
\end{equation}

In this section, we show that the difference between these two identities is small (Lemma \ref{lm:log-universality-general}). Before stating the result, note that by \eqref{eq:der:varphi} and the assumption on the derivatives of $\hat F$, it holds that
\begin{equation}\label{eq:derivativebound}
\begin{gathered}
\Vert K\Vert_{\infty}=O(1) ,\quad \left\Vert\frac{\partial K}{\partial x_{ik}}\right\Vert_{\infty}=O\left (\delta_{i}^{-2\alpha}\right ) ,\quad \left\Vert\frac{\partial^{2}K}{\partial x_{ik}\partial x_{i'k'}}\right\Vert_{\infty}=O\left (\delta_{i}^{-2\alpha}\delta_{i'}^{-2\alpha}\right ) ,\\
\text{and } \left\Vert\frac{\partial^{3}K}{\partial x_{ik}\partial x_{i'k'}\partial x_{i''k''}}\right\Vert_{\infty}=O\left (\delta_{i}^{-2\alpha}\delta_{i'}^{-2\alpha}\delta_{i''}^{-2\alpha}\right )
\end{gathered}
\end{equation}
for all $i, i', i'', k, k', k''$.

\begin{lemma} [Universality of log $|P_n|$] \label{lm:log-universality-general}
There exists a constant $\alpha_{0}>0$ such that for every constant $\alpha\in (0,  \alpha_{0}]$, every function $K : \mathbb{R}^{m_{1}+\cdots+m_{M}}\rightarrow \mathbb{R}$  that satisfies \eqref{eq:derivativebound}  and every $w_{ik}$  in $B(1-3\delta_{i}/2,2\delta_{i}/3)$,  we have
$$
|\E K(\log|P_n(w_{ik})|)_{ik}-\E K(\log|\tilde P_n(w_{ik})|)_{ik}|=O(\delta_0^{\alpha}).
$$
 \end{lemma}
 
In order to prove  Lemma \ref{lm:log-universality-general}, we first prove the following smooth version where the log function is replaced by a smooth function. The proof of Lemma \ref{lm:log-universality-general} follows by a routine smoothening argument that we defer to Appendix \ref{app:log_universality}. Both proofs are based on \cite{TVpoly}.

\begin{lemma} \label{lm:log-universality-smooth}
	There exists a constant $\alpha_{0}>0$  such that for every $\alpha\in (0,  \alpha_{0}]$, every smooth \footnote{By ``smooth", we mean that $L$ has continuous derivatives up to order 3.} function $L : \mathbb{C}^{m_{1}+\cdots+m_{M}}\rightarrow \mathbb{R}$  that satisfies \eqref{eq:derivativebound} and every $w_{ik}$  in $B(1-3\delta_{i}/2,2\delta_{i}/3)$, we have
\begin{equation}\label{eq:lm:log-universality-smooth}
\left|\E L\left (\frac{P_n(w_{ik})}{\sqrt{V(w_{ik})}}\right )_{\substack{i=1,\ldots,M\\k=1,\dots,m_{i}}}-\E L\left (\frac{\tilde{P}_n(w_{ik})}{\sqrt{V(w_{ik})}}\right )_{\substack{i=1,\ldots,M\\k=1,\dots,m_{i}}}\right|=O(\delta_0^{\alpha}) ,
\end{equation}
where $V(w):=  \sum_{j=N_{0}}^{n}|c_{j}|^{2}|w|^{2j}$ and $N_{0}$  is the constant in Conditions \eqref{cond_mean_0} and \eqref{cond_ci_rho}.
\end{lemma}
\begin{proof}[Proof of Lemma \ref{lm:log-universality-smooth}.] We use the Lindeberg swapping argument. Let $P_{i_{0}}(z)=\sum_{i=0}^{i_{0}-1}c_{i}\tilde{\xi}_{i}z^{i}+\sum_{i=i_{0}}^{n}c_{i}\xi_{i}z^{i},$ for $0 \le i_0 \le n+1$. 
	Then  $P_{0}=P_n$ and $P_{n+1}=\tilde{P}_n$ and $P_{i_0 +1} $ is obtained from $P_{i_0} $ by replacing the random variable $\xi_{i_0} $ by $\tilde \xi_{i_0} $.  Let 
$$
I_{i_{0}}:=\left|\E L\left(\frac{P_{i_{0}}(w_{ik})}{\sqrt{V(w_{ik})}}\right)_{ik}-\E L\left(\frac{P_{i_{0}+1}(w_{ik})}{\sqrt{V(w_{ik})}}\right)_{ik}\right|.
$$

The left-hand side of \eqref{eq:lm:log-universality-smooth} is bounded by $\sum_{i_{0}=0}^{n}I_{i_{0}}$.
Fix $i_{0}\in[N_{0}, n+1]$ (where $N_0$ is the constant in Conditions \eqref{cond_mean_0} and \eqref{cond_ci_rho}) and let 
$$Y_{ik}:= \frac{P_{i_{0}}(w_{ik})}{\sqrt{V(w_{ik})}}-\frac{c_{i_{0}}\xi_{i_{0}}w_{ik}^{i_{0}}}{\sqrt{V(w_{ik})}}$$
 for $1\leq i\leq M, 1\leq k\leq m_{i}$. We have $$ \frac{P_{i_{0}+1}(w_{ik})}{\sqrt{V(w_{ik})}}=Y_{ik}+\frac{c_{i_{0}}\tilde{\xi}_{i_{0}}w_{ik}^{i_{0}}}{\sqrt{V(w_{ik})}}. $$  Conditioned on the $\xi_{j}$ and $\tilde{\xi}_{j}$ for all $j\neq i_{0}$, the $Y_{ik}$ are fixed. To bound $I_{i_0}$, we reduce to bounding 
\begin{equation}
d_{i_{0}}:=\left|\E _{\xi_{i_{0}},\tilde{\xi}_{i_{0}}}\hat{L}\left(\frac{c_{i_{0}}\xi_{i_{0}}w_{ik}^{i_{0}}}{\sqrt{V(w_{ik})}}\right)_{ik}-\E _{\xi_{i_{0}},\tilde{\xi}_{i_{0}}}\hat{L}\left(\frac{c_{i_{0}}\tilde{\xi}_{i_{0}}w_{ik}^{i_{0}}}{\sqrt{V(w_{ik})}}\right)_{ik}\right|.\label{def_di0}
\end{equation}
where  $\hat{L}=\hat{L}_{i_{0}}(x_{ik})_{ik} :=L(Y_{ik}+x_{ik})_{ik}$. Note that this function $\hat{L}$ also satisfies \eqref{eq:derivativebound} because $L$ does.

Let $a_{ik,i_{0}}= \frac{c_{i_{0}}w_{ik}^{i_{0}}}{\sqrt{V(w_{ik})}}$. By Condition \eqref{cond_ci_rho}, we have
\begin{equation}\label{eq:bound:V:below}
V(w_{ik}) \gg  \sum_{j=\delta_i^{-1}}^{2\delta_i^{-1}} j^{2\rho} \big(1-13\delta_i/6\big)^{2j}  \gg \delta_i^{-1-2\rho}
\end{equation}
and
\begin{equation}\label{eq:bound:cw}
\left |c_{i_0} w_{ik}^{i_{0}}\right |\ll i_0^{\rho} \big(1- \delta_i/6\big)^{i_0}\ll  i_0^{\rho} \exp\left (- i_0\delta_i/6\right ) \ll \max\{1,  \delta_i^{-\rho}\}.
\end{equation}
Since $\rho>-1/2$, we have from \eqref{eq:bound:V:below} and \eqref{eq:bound:cw} that 
\begin{equation}\label{eq:delocalization:1}
|a_{ik,i_{0}}|\ll \delta_{i}^{\alpha_{1}}
\end{equation}
 for some constant $\alpha_{1}>0$. Taylor expanding $\hat{L}$ around the origin, we obtain 
\begin{equation}\label{jul1}
\hat{L}(a_{ik,i_{0}}\xi_{i_{0}})_{ik}=\hat{L}(0)+\hat{L}_{1}+\err_{1},
\end{equation}
where
$$
\hat{L}_{1}=\left.\frac{\d\hat{L}(a_{ik,i_{0}}\xi_{i_{0}}t)_{ik}}{\d t}\right|_{t=0}=\sum_{ik}\frac{\partial\hat{L}(0)}{\partial\Re(z_{ik})}\Re(a_{ik,i_{0}}\xi_{i_{0}})+\sum_{ik}\frac{\partial\hat{L}(0)}{\partial \Im(z_{ik})}
\Im(a_{ik,i_{0}}\xi_{i_{0}}).
$$
Since $\hat{L}$ satisfies \eqref{eq:derivativebound}, we have 
\begin{align*}
|\err_{1}| &\quad\leq\quad  \sup_{t\in[0,1]} \left|\frac{1}{2}\frac{\d^{2}\hat{L}(a_{ik,i_{0}}\xi_{i_{0}}t)_{ik}}{\d t^{2}}\right|\\
&\quad\ll \quad  |\xi_{i_{0}}|^{2}\sum_{ik,i'k'}  |a_{ik,i_{0}}||a_{i'k',i_{0}}|\delta_{i}^{-2\alpha}\delta_{i'}^{-2\alpha} \ll |\xi_{i_{0}}|^{2}\left(\sum_{ik}|a_{ik,i_{0}}|\delta_{i}^{-2\alpha}\right)^{2}.
\end{align*}

Expanding to the next derivative, we have, in a similar manner, 
\begin{equation} \label{eq:Taylor}
\hat{L}(a_{ik,i_{0}}\xi_{i_{0}})_{ik}=\hat{L}(0)+\hat{L}_{1}+\frac{1}{2}\hat{L}_{2}+\err_{2},
\end{equation}
where $ \hat{L}_{2}=\left.\frac{\d ^{2}\hat{L}(a_{ik,i_{0}}\xi_{i_{0}}t)_{ik}}{\d t^{2}}\right|_{t=0}$ and 
\begin{equation}
| \err_{2}| \ll |\xi_{i_{0}}|^{3}\left(\sum_{ik}|a_{ik,i_{0}}|\delta_{i}^{-2\alpha}\right)^{3}.\nonumber 
\end{equation}

By definition,  $|  \err_{2}|=\left |\err_{1}-\frac{1}{2}\hat{L}_{2}\right |\ll |\xi_{i_{0}}|^{2}(\sum_{ik}|a_{ik,i_{0}}|\delta_{i}^{-2\alpha})^{2}$.  Using interpolation, H\"older's inequality and $m_i = \delta_i^{-11\alpha}$, we get
\begin{align*}
|\err_{2}| & \ll    |  \xi_{i_{0}}|^{2+\varepsilon}\left(\sum_{ik}|a_{ik,i_{0}}|\delta_{i}^{-2\alpha}\right)^{2+\varepsilon}
& \ll  |\xi_{i_{0}}|^{2+\varepsilon}M^{1+\varepsilon}\sum_{i=1}^{M}\delta_{i}^{-50\alpha}\left(\sum_{k=1}^{m_{i}}|a_{ik,i_{0}}|^{2}\right)^{(2+\varepsilon)/2}.
\end{align*}
 
All of these estimates also hold for $\tilde{\xi}_{i_0}$ in place of $\xi_{i_0}$. Since $\xi_{i_0}$ and $\tilde \xi_{i_0}$ have the same first and second moments and they both have bounded $(2+\varepsilon)$ moments, we get 
$$
d_{i_{0}}=|\E \err_{2}| \ll M^{1+\varepsilon}\sum_{i=1}^{M}\delta_{i}^{-50\alpha}\left(\sum_{k=1}^{m_{i}}|a_{ik,i_{0}}|^{2}\right)^{(2+\varepsilon)/2}.
$$

Taking expectation with respect to the remaining variables shows that the same upper bound holds for $I_{i_{0}}$ 
for all $N_{0}\leq i_{0}\leq n+1$. By \eqref{eq:delocalization:1}, choosing $\alpha$ to be sufficiently small compared to $\alpha_{1}$, we have $\left(  \sum_{k=1}^{m_{i}}|a_{ik,i_{0}}|^{2}\right)^{\varepsilon/2} \ll \delta_{i}^{(2\alpha_{1}-11\alpha)\varepsilon /2} \ll \delta_{i}^{100\alpha}$.
Hence,
$$
\sum_{i_{0}=N_{0}}^{n+1}I_{i_{0}}\ll M^{1+\ep}\sum_{i_{0}=N_{0}}^{n+1}\sum_{i=1}^{M}\delta_{i}^{50\alpha}\sum_{k=1}^{m_{i}}|a_{ik,i_{0}}|^{2} \ll \log^{2}n\sum_{i=1}^{M}\delta_{i}^{2\alpha}\ll (\log^{2}n)  \delta_0^{2\alpha}\ll \delta_0^{\alpha},
$$
where we used $M\ll \log n$, $  \sum_{i_0=N_{0}}^{n+1} |a_{ik,i_{0}}|^{2}=1$ and \eqref{cond-a-log}.

For $0\leq i_{0}<N_{0}$, instead of \eqref{jul1} and \eqref{eq:Taylor}, we use mean value theorem to get a rough bound
$$
\hat{L}(a_{ik,i_{0}}\xi_{i_{0}})_{ik}=\hat{L}(0)+O\left(|\xi_{i_{0}}|\sum_{ik}\delta_{i}^{-2\alpha}|a_{ik,i_{0}}|\right) ,
$$
which by the same arguments as above gives
$$
I_{i_{0}}\ll M^{1/2}\left(\sum_{i=1}^{M}\delta_{i}^{-50\alpha}\sum_{k=1}^{m_{i}}|a_{ik,i_{0}}|^{2}\right)^{1/2}\ll \log^{1/2}n\sqrt{\sum_{i=1}^{M}\delta_{i}^{2\alpha_{1}-61\alpha}}\ll \delta_0^{\alpha}.
$$
Taking all these bounds together, we get $\sum_{i_{0}=0}^{n+1}I_{i_{0}} \ll \delta_0^{\alpha}$. This completes the proof  of Lemma \ref{lm:log-universality-smooth}.
\end{proof}

\subsection{Finishing the proof of Proposition \ref{prop:univ:1M}}\label{subsection:proof:univ:1M}
In Lemma \ref{lemma:aproximation1}, we approximated the number of real roots in dyadic intervals, $N_i$, by the sums $\sum_{j=1}^{n}\varphi_{i}(\zeta_{j})$ and estimated the error term to be
	 		\begin{equation}\label{eq:smooth1:rep}
		\E \hat{F}(N_{1}, \ldots, N_{M})-\E  \hat F\left(\sum_{j=1}^{n}\varphi_{1}(\zeta_{j}), \ldots,\sum_{j=1}^{n}\varphi_{M}(\zeta_{j})\right) \ll \delta_0^{\alpha/8}
		\end{equation}
as in \eqref{eq:smooth1}.
	
	Applying this bound for the specific Gaussian case yields
	\begin{equation}\label{eq:smooth1:rep:2}
	\E \hat{F}(\tilde N_{1}, \ldots, \tilde N_{M})-\E  \hat F\left(\sum_{j=1}^{n}\varphi_{1}(\tilde \zeta_{j}), \ldots,\sum_{j=1}^{n}\varphi_{M}(\tilde \zeta_{j})\right) \ll \delta_0^{\alpha/8} .
	\end{equation}

	It remains to show that
	\begin{equation}\label{smooth2:rep}
	\E \hat F \left(\sum_{j=1}^{n}\varphi_{1}(\zeta_{j}), \ldots,\sum_{j=1}^{n}\varphi_{M}(\zeta_{j})\right)=\E \hat F\left(\sum_{j=1}^{n}\varphi_{1}\left (\tilde{\zeta}_{j}\right ), \ldots,\sum_{j=1}^{n}\varphi_{M}\left (\tilde{\zeta}_{j}\right ) \right )+O(\delta_0^{\alpha})
	\end{equation}
	where $\tilde{\zeta}_{j}$ are the roots of $\tilde{P}_{n}.$
	
	By \eqref{rc1}, the sums $\sum_{j=1}^{n}\varphi_{i}(\zeta_{j})$ can be well approximated by the sample sum
	\begin{equation}\label{rc1:rep}
	\left|  \sum_{j=1}^{n}\varphi_{i}(\zeta_{j})-\frac{2\delta_{i}^{2}}{9m_{i}}\sum_{k=1}^{m_{i}}\log|P_n(w_{ik})|\triangle\varphi_{i}(w_{ik})\right|=O(\delta_i^{\alpha})
	\end{equation}
	with probability at least $1-O\left (\delta_i^{\alpha}\right )$, where $w_{ik}$ are chosen independently, uniformly at random from the ball $B(1- 3\delta_{i}/2, 2\delta_{i}/3)$ and are independent of all previous random variables.

	Since $\sum_{i=1}^{M} \delta_i^{\alpha}\ll \delta_0^{\alpha}$, by applying \eqref{rc1:rep}, the left-hand side of \eqref{smooth2} equals
	\begin{equation} 
	\E K(\log |P_n(w_{ik})|)_{\substack{i=1,\ldots,M\\k=1,\dots,m_{i}}}+O(\delta_0^{\alpha})\nonumber
	\end{equation}
	and the right-hand side of \eqref{smooth2} equals 
	\begin{equation} 
	\E K(\log |\tilde P_n(w_{ik})|)_{\substack{i=1,\ldots,M\\k=1,\dots,m_{i}}}+O(\delta_0^{\alpha})\nonumber
	\end{equation}
	where $K$ is the function defined in \eqref{eq:K:def}.
		
For any fixed $w_{ik}$, Lemma \ref{lm:log-universality-general} asserts that the difference between these two identities is small 
		$$
		|\E K(\log|P_n(w_{ik})|)_{ik}-\E K(\log|\tilde P_n(w_{ik})|)_{ik}|=O(\delta_0^{\alpha}).
		$$
  
  This gives Proposition \ref{prop:univ:1M}.\qed

\subsection{From $\mj\cap (0,1)$ to $\mj$}\label{section:reduction:01}

In this section, we detail the modifications needed to prove Theorem \ref{thm:distribution_universality} from the proof of \eqref{eq:f1:2}. For Theorem \ref{thm:distribution_universality}, we need to show that
\begin{equation} 
\left|\E F(N_n (\mj))-\E F\left (\tilde N_n(\mj)\right )\right|\leq Ca_n^{c}+ Cn^{-c}.\nonumber
\end{equation}

Inequality \eqref{eq:f1:2} restricts to the interval $(0, 1)$ and says that
\begin{equation} 
\left|\E F(N_n (\mj\cap (0, 1)))-\E F\left (\tilde N_n(\mj\cap (0, 1))\right )\right|\leq Ca_n^{c}+ Cn^{-c}.\nonumber
\end{equation}

	We first decompose $N_{n}(\mj)$ and $\tilde N_n(\mj)$ into the sum of the numbers of real roots in the intervals $\mj\cap (0, 1)$, $\mj\cap (-1, 0)$, $\mj\cap (1, \infty)$, and $\mj\cap (-\infty, -1)$ and denote by $N^{(i)}_n$ and $\tilde N^{(i)}_n$ with $i=1, \dots, 4,$  the corresponding number of real roots. For example, 
	$$N^{(1)}_n = N_n(\mj\cap(0, 1))\quad\text{and}\quad \tilde N^{(1)}_n = \tilde N_n(\mj\cap(0, 1)),$$
	$$N^{(2)}_n = N_n(\mj\cap(-1, 0))\quad\text{and}\quad \tilde N^{(2)}_n = \tilde N_n(\mj\cap(-1, 0)),$$
	and so on.
	
	Note that
	$$N_n(\mj) = \sum_{i=1}^{4} N^{(i)}_n\quad\text{and}\quad \tilde N_n(\mj) = \sum_{i=1}^{4} \tilde N^{(i)}_n.$$
	
	It has been shown in proving \eqref{eq:f1:2} how to deal with $N^{(1)}_n$. To deal with $N^{(2)}_n$, note that there is a one-to-one correspondence between the real roots of $P_n(z)$ in $(-1, 0)$ and the real roots in $(0, 1)$ of the polynomial $P_n(-z)=\sum_{i=0}^{n} (-1)^{i}c_i\xi_i z^{i}$. Denote this new polynomial by $P^{(2)}_n(z)$ and the original polynomial $P_n(z)$ by $P_n^{(1)}(z)$. All arguments that have been used for $P_n^{(1)}$ to handle $N^{(1)}_n$ can be applied to $P^{(2)}_n$ to handle $N^{(2)}_n$.

	For $N^{(3)}_n$, there is a one-to-one correspondence between the roots of $P_n^{(1)}(z)$ in $(1, \infty)$ and the roots in $(0, 1)$ of the polynomial $ \frac{z^{n}}{c_n}P_n(z^{-1})= \sum_{i=0}^{n} \frac{c_{n-i}}{c_n}\xi_{n-i} z^{i}=: P^{(3)}_n(z)$. The coefficients $\frac{c_{n-i}}{c_n}$ of $P^{(3)}_n$ satisfy Condition \ref{cond_ci_rho} with $\rho = 0$, except for a negligible number of $i$, and hence the same arguments as for $P^{(1)}_n$ also apply for $P^{(3)}_n$.

	Similarly, for $N^{(4)}_n$, there is a one-to-one correspondence between the roots of $P_n^{(1)}(z)$ in $(-\infty, -1)$ and the roots in $(0, 1)$ of the polynomial $ P^{(3)}_n(-z)= \sum_{i=0}^{n} (-1)^{i}\frac{c_{n-i}}{c_n}\xi_{n-i} z^{i}=: P^{(4)}_n(z)$. All arguments that work for $P^{(3)}_n(z)$ also works for $P^{(4)}_n(z)$.
 
	In Section \ref{subsec:partition}, $N^{(1)}_n$ is partitioned into a sum of $N_i$, which is the number of roots of $P^{(1)}_n$ in a dyadic interval $(1- \delta_{i-1}, 1-\delta_i]$. Denote these $N_i$ by $N^{(1)}_i$. Denote by $N^{(2)}_i, N^{(3)_i}$, and $N^{(4)}_i$ the number of roots in the same interval of $P^{(2)}_n, P^{(3)}_n,$ and $P^{(4)}_n$, respectively. We have
	$$N_n(\mj) = \sum_{j=1}^{4}\sum_{i=1}^{M} N^{(j)}_i.$$
All other steps in the proof of \eqref{eq:f1:2} can now be written for the proof of Theorem \ref{thm:distribution_universality} by placing these $N^{(j)}_i$ in place of $N_i$. For example, Proposition \ref{prop:univ:1M} becomes
$$
\left |\E \hat{F}(N^{(1)}_{1}, \ldots, N^{(1)}_{M}, N^{(2)}_{1}, \ldots, N^{(2)}_{M}, \dots)-\E \hat{F}\left (\tilde{N}^{(1)}_{1}, \ldots,\tilde{N}^{(1)}_{M}, \tilde{N}^{(2)}_{1}, \ldots,\tilde{N}^{(2)}_{M}, \dots\right )\right |\ll \delta_0^{c}
$$
where $\hat{F}$ : $\mathbb{R}^{4M}\rightarrow \mathbb{R}$ is any function whose every partial derivative up to order $3$ is bounded by $1$.

\section{Proof of Corollary \ref{cor:moment_universality}}\label{sec:cor:moment_universality}
We define $\delta_0, \dots, \delta_M, N_{1}, \ldots , N_{M}$ as in the beginning of the proof of Theorem \ref{thm:distribution_universality}. Note that
$\delta_i\ge \delta_M\ge 1/n$ and $\delta_0^{c} = \Theta\left (a_n^{c}+n^{-c}\right )$.

To prove the first part of Corollary \ref{cor:moment_universality}, we first reduce to the interval $[1-a_n, 1-b_n)$ as explained in Section \ref{section:reduction:01}, namely, it suffices to show that  
\begin{equation}\label{eq:cor:moment:uni:restriction}
\left|\E \left (N_n^{k}[1-a_n, 1-b_n)\right )-\E \left (\tilde N_n^{k}[1-a_n, 1-b_n)\right )\right|\leq C\delta_0^{c}.
\end{equation}
We write $N := N_{n} [1-a_n, 1-b_n)$, $\tilde N := \tilde N_{n} [1-a_n, 1-b_n)$. 
Let $\mathcal{A}$ be the event on which $N\leq\log^{4}n$ (here, $4$ can be replaced by any large constant). Let $F$ be a smooth function that is supported on the interval $[-1, \log^{4}n+1]$ and $F(x)=x^{k}$ for all $x\in[0, \log^{4}n]$. Since $N$ is always an integer, it holds that $N^{k}1_{\mathcal{A}}=F(N)$. The function $F$ can be chosen such that all of its derivatives up to order $3$ are bounded by $O\left(\log^{4k}n\right)$. Applying Theorem \ref{thm:distribution_universality} to the rescaled function $(\log n)^{-4k} F$, we obtain
$$
\left|\E N^{k}\mathbf{1}_{\mathcal{A}}-\E \tilde{N}^{k}\mathbf{1}_{\tilde{\mathcal{A}}}\right|=\left|\E F(N)-\E F\left (\tilde{N}\right )\right|\ll \delta_0^{2c}\log^{-4k}n\ll \delta_0^{c}
$$
for some small constant $c$ where $\tilde {\mathcal A}$ is the corresponding event on which $\tilde N\leq\log^{4}n$.

To finish the proof of the first part, we show that the contribution from the complement of $\mathcal{A}$ is negligible, i.e., 
$$\E N^{k}\mathbf{1}_{\mathcal{A}^{c}}\ll \delta_0^{c}.$$
Since $M\ll \log n\ll \delta_0^{-c/2}$ by \eqref{cond-a-log} and since $N^{k} \leq M^{k}\sum_{i=0}^{M}N_{i}^{k}$, it suffices to show that for all $i$, $\E N_{i}^{k}\mathbf{1}_{\mathcal{A}^{c}}\ll \delta_0^{2c}$. Let $\mathcal{A}_{i}$ be the event on which $N_{i}\leq\log^{3}(1/\delta_{i})$. Note that $  \bigcap_{i=1}^{M}\mathcal{A}_{i}\subset \mathcal{A}$. Let $A$ be a large constant. By \eqref{eq:tail} of Lemma \ref{lm:bddness}, $\P (\mathcal{A}_{i}^{c})\ll \delta_i^{A}$. Thus, 
$$
\P (\mathcal{A}^{c})\leq\sum_{i=1}^M\P (\mathcal{A}_{i}^{c})\ll \sum_{i=1}^M\delta_{i}^{A}\ll \delta_0^{A}.
$$
This together with \eqref{eq:secondmoment} of Lemma \ref{lm:bddness} give
\begin{eqnarray}
\E N^{k}\mathbf{1}_{\mathcal{A}^{c}} &\ll& \log^{k} n\sum_{i=1}^{M} \E N_{i}^{k}1_{\mathcal{A}^{c}}\nonumber\\
&\ll&  \log^{k} n\sum_{i=1}^{M} \left (\E N_{i}^{2k}\right )^{1/2}\left (\P(\mathcal{A}^{c})\right )^{1/2}\ll  \log^{k}n\sum_{i=1}^{M}\delta_0^{A/2}\left (\log\frac{1}{\delta_i}\right )^{2k}.\nonumber
\end{eqnarray}
Since $\delta_i\ge 1/n$, the right most side is at most $(\log^{4k}n) \delta_0^{A/2}\ll \delta_0^{A/2-1}\ll \delta_0^{c}$ by \eqref{cond-a-log} and by choosing $A\ge 3$.

The second part of Corollary \ref{cor:moment_universality} follows from the first part by observing that
$$
\left  (\E N_n(\mj)\right   )^{2}-\left  (\E \tilde{N}_n(\mj)\right  )^{2} \ll \delta_0^{c}\left (2\E N_n(\mj)+O(\delta_0^{c})\right )\ll \delta_0^{c}\log^{2} n\ll \delta_0^{c/2} 
$$
where in the first inequality, we used the first part of Corollary \ref{cor:moment_universality} for $k=1$, in the second inequality, we used \eqref{eq:secondmoment} to get that
$$ \E N_n(\mj)\ll \sum_{i=1}^{M} \log(1/\delta_i)\le \sum_{i=1}^{M} \log n\ll \log^{2}n,$$
and in the last inequality, we used \eqref{cond-a-log}.  follows immediately by observing that
$$
\left  (\E N_n(\mj)\right   )^{2}-\left  (\E \tilde{N}_n(\mj)\right  )^{2} \ll \delta_0^{c}\left (2\E N_n(\mj)+O(\delta_0^{c})\right )\ll \delta_0^{c}\log^{2} n\ll \delta_0^{c/2} 
$$
where in the first inequality, we used the first part of Corollary \ref{cor:moment_universality} for $k=1$, in the second inequality, we used \eqref{eq:secondmoment} to get that
$$ \E N_n(\mj)\ll \sum_{i=1}^{M} \log(1/\delta_i)\le \sum_{i=1}^{M} \log n\ll \log^{2}n,$$
and in the last inequality, we used \eqref{cond-a-log}.  This completes the proof of Corollary \ref{cor:moment_universality}.

\section{Proof of Proposition \ref{prop:boundedge}}\label{sec:prop:boundedge} 
\subsection{Probability of multiple roots} We start by proving a useful tool that control the probability that the polynomial $P_{n}$ has many roots in a small interval. For any $x, y\in \mathbb{R}$, let
\begin{equation}\label{def:V}
V(x)   :=\Var P_n(x)=\sum_{i=0}^{n}c_{i}^{2}x^{2i} 
\end{equation}
and
\begin{equation}\label{def_r}
r(x, y):=\frac{\E P_n(x)P_n(y)}{\sqrt{V(x)}\sqrt{V(y)}}=\frac{\sum_{i=0}^{n}c_{i}^{2}x^{i}y^{i}}{\sqrt{(\sum_{i=0}^{n}c_{i}^{2}x^{2i})(\sum_{i=0}^{n}c_{i}^{2}y^{2i})}}.
\end{equation}
 \begin{lemma}\label{multipleroot} Assume that the random variables $\xi_i$ are iid standard Gaussian. 
	There exists a constant $C_{0}$ such that for any $0<s<1$, any $k, l\geq 2$, $ 1-  \frac{1}{C_{0}}\leq x<t<1$ and $y, z\in (x, t)$ satisfying
	$$
	\log\frac{1-x}{1-y}=\log\frac{1-z}{1-t}=\delta
	$$
	for some $\delta\in(0,1/2C_{0}]$, we have
	\begin{equation}\label{ine:multiple1}
	\P (N_n(x, y)\geq k)\ll (C_{0}\delta)^{ks} 
	\end{equation}
	and
	\begin{equation}\label{ine:multiple2}
	\P (N_n(x, y)\geq k, N_n(z, t)\geq l)\ll (C_{0}\delta)^{2ks}+(C_{0}\delta)^{2ls}+\frac{(C_{0}\delta)^{(k+l)s}}{\sqrt{1-r^{2}(y,t)}},
	\end{equation}
	where the implicit constants depend only on $s$, not on $k,l, x, y, \delta$.
\end{lemma} 

\begin{proof}[Proof of Lemma \ref{multipleroot}] We start by proving \eqref{ine:multiple2}. By Rolle's theorem and the fundamental theorem of calculus, if $P$ has at least $k$ zeros in the interval $(x, y)$ then
	$$
	|P_n(y)|\leq\int_{x}^{y}\int_{x}^{y_{1}}\cdots\int_{x}^{y_{k-1}}|P_n^{(k)}(y_{k})|dy_{k}\ldots dy_{1}=:I_{x,y}.
	$$

	Therefore,
	\begin{align*}
	\P (N_n(x, y)\geq k, N_n(z, t)\geq l) \quad\leq\quad& \P \left(I_{x,y}\geq\varepsilon_{1}\sqrt{V(y)}\right)+\P \left(I_{z,t}\geq\varepsilon_{2}\sqrt{V(t)}\right)\\
	&+\quad \P \left(|P_n(y)|\leq\varepsilon_{1}\sqrt{V(y)}, |P_n(t)|\leq\varepsilon_{2}\sqrt{V(t)}\right)
	\end{align*}
	where $\varepsilon_{1}:=(C_{0}\delta)^{ks}$, $\varepsilon_{2}:=(C_{0}\delta)^{ls}$ and 
	\begin{equation}\label{varbound:1}
	V(x)   =\Var P_n(x)=\sum_{i=0}^{n}c_{i}^{2}x^{2i}.
	\end{equation}
	By \eqref{cond_ci_rho}, we have the following estimate whose proof is deferred to Appendix \ref{app:proof:varbound} as it is merely algebraic:
	\begin{equation}\label{varbound}
	V(x)  = \frac{\Theta(1)}{(1-x+1/n)^{2\rho+1}}\quad \forall x\in(1-1/C, 1).
	\end{equation}

	Since $\left( \frac{P_n(y)}{\sqrt{V(y)}}, \frac{P_n(t)}{\sqrt{V(t)}}\right)$ is a Gaussian vector with mean 0 and covariance matrix \newline $\left[\begin{array}{cc}
	1 & r(y,t)\\
	r(y,t) & 1
	\end{array}\right]$, we have
	$$
	\P \left(|P_n(y)|\leq\varepsilon_{1}\sqrt{V(y)}, |P_n(t)|\leq\varepsilon_{2}\sqrt{V(t)}\right)\ll \frac{\varepsilon_{1}\varepsilon_{2}}{\sqrt{1-r^{2}(y,t)}}.
	$$
	It remains to show that
	\begin{equation}\label{intbound}
	\P \left(I_{x,y}\geq\varepsilon_{1}\sqrt{V(y)}\right)\ll (C_{0}\delta)^{2ks}.
	\end{equation}
	
	Since  $0<s<1$, there exists $h>0$ such that $s=  \frac{2+h}{4+h}$. By Markov's inequality, we have
	\begin{eqnarray}
	&&\left(\varepsilon_{1}\sqrt{V(y)}\right)^{2+h}\P \left(I_{x,y}\geq\varepsilon_{1}\sqrt{V(y)}\right) \le  \E \left(\int_{x}^{y}\int_{x}^{y_{1}}\cdots\int_{x}^{y_{k-1}}|P_n^{(k)}(y_{k})|dy_{k}\ldots dy_{1}\right)^{2+h}.\nonumber
	\end{eqnarray}
	By H\"older's inequality, the right-hand side is at most
	$$\left(\frac{(y-x)^{k}}{k!}\right)^{1+h}\E \int_{x}^{y}\int_{x}^{y_{1}}\cdots\int_{x}^{y_{k-1}}|P_n^{(k)}(y_{k})|^{2+h}dy_{k}\ldots dy_{1}$$
	and so
	\begin{eqnarray}
	\left(\varepsilon_{1}\sqrt{V(y)}\right)^{2+h}\P \left(I_{x,y}\geq\varepsilon_{1}\sqrt{V(y)}\right) \le \left(\frac{(y-x)^{k}}{k!}\right)^{2+h}\sup_{w\in(x,y)}\E |P_n^{(k)}(w)|^{2+h}. \label{intbd2}
	\end{eqnarray}
	
	For each $w\in (x, y)$, since $P_n^{(k)}(w)$ is a Gaussian random variable, using the hypercontractivity inequality for the Gaussian distribution (see, for example, \cite[Corollary 5.21]{massartconcentrationbook}), we have for some constant $C$, 
	\begin{eqnarray}
	\E |P^{(k)}(w)|^{2+h}&\ll& \left(\E |P_n^{(k)}(w)|^{2}\right)^{\frac{2+h}{2}}
	\ll  \left (\frac{C^{k}(k!)^{2}}{(1-y+1/n)^{2\rho+2k+1}}\right )^{\frac{2+h}{2}}\nonumber
	\end{eqnarray}
	where in the last inequality, we used an estimate similar to \eqref{varbound}.

	Plugging this and \eqref{varbound} into \eqref{intbd2}, we obtain
	\begin{equation*}
	\P \left(I_{x,y}\geq\varepsilon_{1}\sqrt{V(y)}\right)\ll  \frac{ (1-y+1/n)^{\frac{(2\rho+1)(2+h)}{2}}}{\varepsilon_{1}^{2+h}}\left(\frac{(y-x)^{k}}{k!}\right)^{2+h}\left(\frac{C^{k}(k!)^{2}}{(1-y+1/n)^{2\rho+2k+1}}\right)^{\frac{2+h}{2}}
	\end{equation*}
	which gives
	\begin{align*}
	\P \left(I_{x,y}\geq\varepsilon_{1}\sqrt{V(y)}\right) \ll \frac{ 1}{\varepsilon_{1}^{2+h}}\left(C\frac{y-x}{1-y+1/n}\right)^{k(2+h)}\ll  \frac{ 1}{\varepsilon_{1}^{2+h}}\left(C\frac{y-x}{1-y}\right)^{k(2+h)}.
	\end{align*}
	Using $\ep_1 = (C_{0}\delta)^{ks}$, $\frac{y-x}{1-y} = \frac{1-x}{1-y}-1 = e^{\delta}-1\le 2\delta$ for $\delta\le \frac{1}{2C_0}$ and $s = \frac{2+h}{4+h}$, we get
	\begin{equation} 
	\P \left(I_{x,y}\geq\varepsilon_{1}\sqrt{V(y)}\right) \ll\frac{1}{(C_{0}\delta)^{ks(2+h)}}(2C\delta)^{k(2+h)}\ll (C_{0}\delta)^{2ks} \nonumber
	\end{equation}
	by choosing $C_0\ge 2C$. This proves \eqref{ine:multiple2}.

	The inequality \eqref{ine:multiple1} is obtained by the same reasoning:
	\begin{align*}
	\P (N_n(x, y)\geq k) &\quad\leq\quad \P \left(I_{x,y}\geq\varepsilon_{1}\sqrt{V(y)}\right)+\P \left(|P_n(y)|\leq\varepsilon_{1}\sqrt{V(y)}\right).
	\end{align*}
	Thus
	\begin{equation}
	 \P (N_n(x, y)\geq k) \ll   (C_{0}\delta)^{2ks} +\varepsilon_{1}\ll (C_{0}\delta)^{ks}.\nonumber
	\end{equation}
	
	This completes the proof of Lemma \ref{multipleroot}.
\end{proof}

\subsection{Partition into pieces}
 Let $\mathfrak A$  be the right-hand side of \eqref{eq:prop:boundedge}:
$$\mathfrak{A} :=\begin{cases}
 \left (\log a_n\right )^{2k} + \log^{k} (nb_n) \quad \text{if } b_n\ge 1/n,\\
 \left (\log a_n\right )^{2k} \quad \text{if } b_n < 1/n.
\end{cases}.$$

Writing $\R\setminus \mj$ as a union of four sets $T_1:= [0, 1]\setminus \mj$, $T_2:=[-1, 0]\setminus \mj$, $T_3:=(1, \infty)\setminus \mj$ and $T_4:=(-\infty, -1)\setminus \mj$ and using triangle inequality, we reduce Proposition \ref{prop:boundedge} to showing that for each $1\le i\le 4$, 
\begin{equation}\label{eq:prop:boundedge:01}
\E N_{P_n}^{k}(T_i)=\E N_n^{k}(T_i)\ll \mathfrak A.
\end{equation}

We only prove \eqref{eq:prop:boundedge:01} for $i=1$; the proofs for the remaining $i=2, 3, 4$ are similar.
Since $\mathfrak A\gg 1$, by triangle inequality, \eqref{eq:prop:boundedge:01} follows from showing that for some large constant $C$, 
\begin{equation}\label{eq:root1}
\E N^{k}_{n}[0, 1-1/C]\ll 1,
\end{equation}
\begin{equation}\label{eq:root2}
\E N^{k}_{n} (1-1/C, 1-a_n)\ll \left (\log a_n^{-1}\right )^{2k},
\end{equation}
\begin{equation}\label{eq:root3}
\E N^{k}_{n} \left (1-\frac{a_n}{n}, 1\right )\ll 1,
\end{equation}
and 
\begin{equation}\label{eq:root4}
\E N^{k}_{n} \left (1-b_n, 1-\frac{a_n}{n}\right )\ll \mathfrak A
\end{equation}
where we note that if $1-b_n> 1-\frac{a_n}{n}$ then the interval $\left (1-b_n, 1-\frac{a_n}{n}\right )$ is empty and \eqref{eq:root4} is vacuously true.

The bound \eqref{eq:root1} is precisely the content of \cite[Lemma 2.5]{DOV}.  In the following sections, we show \eqref{eq:root2}, \eqref{eq:root3}, and \eqref{eq:root4}.

\subsection{Proof of \eqref{eq:root2}} Dividing the interval $(1-1/C, 1-a_n)$ into dyadic intervals $I_0:=\left (1-\frac{1}{C}, 1-\frac{1}{2C}\right ]$,  $I_1:=\left (1-\frac{1}{2C}, 1-\frac{1}{4C}\right ]$, $\dots$, $I_m:=\left (1-\frac{1}{2^{m}C}, 1-a_n\right )$ (where $\frac{1}{2^{m}C}\ge a_n > \frac{1}{2^{m+1}C}$) and applying triangle inequality together with \eqref{eq:secondmoment}, we obtain 
$$\left (\E N_n^{k} (1-1/C, 1-a_n)\right )^{1/k}\le \sum_{i=0}^{m} \left (\E N_n^{k} (I_i)\right )^{1/k}\ll \sum_{i=0}^{m} \log (2^{i}C)\ll (\log a_n^{-1})^{2}.$$
Thus, 
$$\E N_n^{k} (1-1/C, 1-a_n) \ll (\log a_n^{-1})^{2k}$$
proving \eqref{eq:root2}.

\subsection{Reducing to Gaussian}
To prove \eqref{eq:root3} and \eqref{eq:root4}, applying \eqref{eq:cor:moment:uni:restriction} to the intervals \newline $\left (1-\frac{a_n}{n}, 1\right )$ and $\left (1-b_n, 1-\frac{a_n}{n}\right )$, we get
$$
\left|\E N_n^{k}\left (1-\frac{a_n}{n}, 1\right ) -\E  \tilde N_n^{k}\left (1-\frac{a_n}{n}, 1\right ) \right|\ll n^{-c}\ll 1
$$
and 
$$
\left|\E N_n^{k}\left (1-b_n, 1-\frac{a_n}{n}\right ) -\E   \tilde N_n^{k}\left (1-b_n, 1-\frac{a_n}{n}\right ) \right|\leq Cb_n^{c}+ Cn^{-c}\ll 1\ll \mathfrak A.
$$
Thus, it remains to prove \eqref{eq:root3} and \eqref{eq:root4} when the random variables $\xi_i$ are iid standard Gaussian. So for the rest of this proof, we assume that it is the case.  

\subsection{Proof of \eqref{eq:root3}}
For \eqref{eq:root3}, we use H\"older's inequality and \eqref{eq:secondmoment} to conclude that 
\begin{eqnarray}
\E N^{k}_{n} \left (1-\frac{a_n}{n}, 1\right ) &\le& \left (\E N^{2k-1}_{n} \left (1-\frac{a_n}{n}, 1\right )\right )^{1/2}\left (\E N_{n} \left (1-\frac{a_n}{n}, 1\right )\right )^{1/2}\nonumber\\
&\ll& (\log n)^{2k-1} \left (\E N_{n} \left (1-\frac{a_n}{n}, 1\right )\right )^{1/2}.\label{eq:root3:KR}
\end{eqnarray}
Using the Kac-Rice formula \eqref{eq:KacRice}, we get 
\begin{equation}\label{eq:root3:KR1}
\E N_n\left (1-\frac{a_n}{n}, 1\right )=   \frac{1	}{\pi}\int_{1-\frac{a_n}{n}} ^{1} \frac{\sqrt{\sum_{i=0}^{n}\sum_{j=i+1}^{n} c_i^{2} c_j^{2}(j-i)^{2} t^{2i+2j-2}}}{\sum_{i=0}^{n} c_i^{2}t^{2i}}dt \ll  \int_{1-\frac{a_n}{n}} ^{1} n dt = a_n. 
\end{equation}
where we used $|j-i|\le n+1$ and $\sum_{i=0}^{n}\sum_{j=i+1}^{n} c_i^{2} c_j^{2} t^{2i+2j} = \left (\sum_{i=0}^{n} c_i^{2}t^{2i}\right )^{2}$.
Plugging this into \eqref{eq:root3:KR} and using \eqref{cond-a-log}, we obtain
\begin{equation}
\E N^{k}_{n} \left (1-\frac{a_n}{n}, 1\right ) \ll (\log n)^{2k-1}a_n^{1/2}\ll 1.
\end{equation}

\subsection{Proof of \eqref{eq:root4}}
	By \eqref{ine:multiple1}, for every interval $[x, y]$ with 
	\begin{equation}\label{eq:such:interval}
	1-b_n\le x<y<1-\frac{a_n}{n}\quad \text{and}\quad \log\frac{1-x}{1-y} = \frac{1}{C},
	\end{equation}
	where $C$ is a sufficiently large constant,	we have
	$$\E N^{k}_n(x, y) \le \E N_n(x, y) + \sum _{j=2}^{\infty} j^{k} \P\left (N_n(x, y)= j\right )\ll \E N_n(x, y) + \sum _{j=2}^{\infty} j^{k} 2^{-j} = \E N_n(x, y) + O(1).$$
	Dividing the interval $\left (1-b_n, 1-\frac{a_n}{n}\right )$ into $O\left (\log \frac{n}{a_n} + \log b_n\right ) = O\left (\log \frac{nb_n}{a_n}\right )$ intervals that satisfy \eqref{eq:such:interval}, we obtain
 \begin{equation}\label{eq:root4:0}
\E N^{k}_{n} \left (1-b_n, 1-\frac{a_n}{n}\right )\ll \left (\log \frac{nb_n}{a_n}\right )^{k-1} \E N_n \left (1-b_n, 1-\frac{a_n}{n}\right ) + \left (\log \frac{nb_n}{a_n}\right )^{k-1} .
\end{equation}
	So, \eqref{eq:root4} follows from \eqref{eq:root4:0} and the following
	\begin{equation}\label{eq:mean:root}
	\E N_n \left (1-b_n, 1-\frac{a_n}{n}\right ) \ll \max\{1, \log(nb_n)\}
	\end{equation}
which can be deduced from
		\begin{equation}\label{eq:mean:root:C:1}
		\E N_n \left (1-b_n, 1-\frac{C}{n}\right ) \ll  \log(nb_n) \quad\text{if}\quad b_n\ge C/n
	\end{equation}
and
	\begin{equation}\label{eq:mean:root:C:2}
	\E N_n \left (1-\frac{C}{n}, 1-\frac{a_n}{n}\right ) \ll 1.
\end{equation}

To prove \eqref{eq:mean:root:C:1}, let $c_{i, \rho}:= \sqrt{\frac{(2\rho+1)\dots(2\rho+i)}{i!}}$. We have $c_{i, \rho}= \Theta(c_i)$ for all $i\ge N_0$ thanks to  assumption \ref{cond_ci_rho}.

	Using Kac-Rice formula \eqref{eq:KacRice}, we have 
	\begin{equation}\label{eq:KR1:1}
	\E N_n\left (1-b_n, 1-\frac{C}{n}\right )  \ll   \int_{1-b_n}^{1-\frac{C}{n}} \frac{\sqrt{\sum_{i=0}^{n}\sum_{j=i+1}^{n} c_{i, \rho}^{2} c_{j, \rho}^{2}(j-i)^{2} t^{2i+2j-2}}}{\sum_{i=0}^{n} c_{i, \rho}^{2}t^{2i}}dt.
	\end{equation}   
	We use \cite[Lemma 10.3]{DOV} with $h(k) = c_{i, \rho}^{2}$ which estimates the above integrand uniformly over the interval $\left (1-b_n, 1-\frac{C}{n}\right )$ and asserts that
	\begin{eqnarray}
	\frac{\sqrt{\sum_{i=0}^{n}\sum_{j=i+1}^{n} c_{i, \rho}^{2} c_{j, \rho}^{2}(i-j)^{2} t^{2i+2j-2}}}{\sum_{i=0}^{n} c_{i, \rho}^{2}t^{2i}} &\ll & \frac{\sqrt{2\rho+1}}{2\pi(1-t)} +   (1-t)^{\rho-1/2} + \frac{1}{n(1-t)^{2}}\nonumber.		
	\end{eqnarray}
	which is $\ll \frac{1}{1-t}$ by the assumption $\rho>-1/2$.
	
	That gives \eqref{eq:mean:root:C:1} because
\begin{equation}\label{eq:root4:1}
\E N_n \left (1-b_n, 1-\frac{C}{n}\right ) \ll  \int_{1-b_n}^{1-\frac{C}{n}} \frac{1}{1-t}dt\ll \log n + \log b_n = \log(nb_n).
\end{equation}

	For \eqref{eq:mean:root:C:2}, we use the same bound as in \eqref{eq:root3:KR1} to obtain
 \begin{equation}\label{eq:root4:2}
\E N_n \left (1-\frac{C}{n}, 1 - \frac{a_n}{n}\right ) \ll  \int_{1-\frac{C}{n}}^{1 - \frac{a_n}{n}}n dt \ll 1.
\end{equation}
This proves \eqref{eq:mean:root:C:2} and completes the proof of \eqref{eq:root4}.

\section{Proof of Lemma \ref{lm:CLT_gau}}\label{sec:lm:CLT_gau}
Since the lemma only involves Gaussian random variables $\tilde \xi_i$, we simplify the notation and write $\xi_i$ for $\tilde \xi_i$ and $N_n(S)$ for $\tilde N_n(S)$ (this helps us to avoid 
multiple superscripts later on). Thus,  for this section, $\xi_{i}\sim N(0, 1)$ for all $i$. 

We will adapt the argument in Maslova \cite{Mas2}, which is to approximate the number of roots by a sum of independent random variables. 
 Since the random variables $\xi_i$ are now standard Gaussian, numerous technical steps in \cite{Mas2}, which may be impossible to reproduce without having $c_0= \dots =c_n=1$, 
 can be greatly simplified and applied to our general setting thanks to special properties of Gaussian variables.

\subsection{Approximate the number of real roots by the number of sign changes.}\label{subsection:approximate-sign}
Let $V$ and $r$ be defined as in \eqref{def:V} and \eqref{def_r}. Lemma \ref{multipleroot} asserts that in a small interval, it is  unlikely that the polynomial $P_n$ has more than $1$ root. If $P_n$ has at most $1$ root in an interval $(a, b)$ and does not vanish at $a$ and $b$ then $N_{n}(a, b)=1$ if $P_n(a)$ and $P_n(b)$ have different signs and $N_{n}(a, b)=0$ otherwise. Hence, on a small interval $(a,b)$, it is reasonable to 
approximate  $N_{n}(a, b)$ by the number of sign changes:
\begin{equation}\label{def_Nsign}
N_{n}^{\sign}(a,   b)=\frac{1}{2}-\frac{1}{2} \sign (P_{n}(a)P_{n}(b))
\end{equation}
where 
$$\sign (x): = \begin{cases}
1\quad \text{if } x> 0,\\
0\quad \text{if } x= 0,\\
-1\quad \text{if } x>  0.
\end{cases}$$

The following lemma estimates the accuracy  of this approximation for a long interval. 

\begin{lemma}[Approximate by sign changes] \label{lmm:tosignchange} Assume that the $\xi_i$ are iid standard Gaussian. For any positive constant $\ep$, there exist constants $C, C'$  such that the following holds. Let $T>1/C$ and $a, b$  be such that $1-a_n\leq a<b\leq 1-b_n$ and $\log\frac{1-a}{1-b}=T$. Let $j_0=\delta^{-1}\log(1-a)^{-1}$  and $j_{1}=\delta^{-1}\log(1-b)^{-1}$ where $\delta$ is any number with 
	$$\exp(-(\log\log n)^{1+\varepsilon})<\delta<1/C.$$ 
	Assume (without loss of generality) that $j_0$ and $j_{1}$  are integers and let  $x_{j}=1-\exp(-j\delta)$ for all  $j=j_{0}, \ldots , j_{1}.$ Let
$$S=S_{a,b,\delta}= N_n[a, b) =\sum_{j=j_{0}}^{j_{1}-1}N_n[x_{j}, x_{j+1})\quad \text{and}\quad S^{\sign}=S_{a,b,\delta}^{\sign}=  \sum_{j=j_{0}}^{j_{1}-1}N_n^{\sign}(x_{j}, x_{j+1}).$$
 Then
$$
\E (S-S^{\sign})^{2}\leq C'T^{2}\delta^{1-\ep}.
$$
\end{lemma}

\begin{proof}[Proof of  Lemma \ref{lmm:tosignchange}]
Note that $x_{j_{0}}=a, x_{j_{1}}=b$, and $j_{1}-j_{0}=\delta^{-1}T.$

We have
$$
\E (S-S^{\sign})^{2}=\sum_{i=j_{0}}^{j_{1}-1}\E (N_{i}-N_{i}^{\sign})^{2}+2\sum_{j_{0}\leq i<j\leq j_{1}-1}\E (N_{i}-N_{i}^{\sign})(N_{j}-N_{j}^{\sign})
$$
where $N_{j} :=N_n[x_{j}$, $x_{j+1}$) and $N_{j}^{\sign}:=N_n^{\sign}(x_{j}, x_{j+1}$).

By Lemma \ref{multipleroot}, we have

\begin{equation}\label{st2}
\sum_{i=j_{0}}^{j_{1}-1}\E (N_{i}-N_{i}^{\sign})^{2}\leq \sum_{i=j_{0}}^{j_{1}-1}\sum_{k=2}^{n}k^{2}\P (N_{i}=k)\ll  T\delta^{-1}\sum_{k=2}^{n}k^{2}(C_{0}\delta)^{k(1-\ep/2)}\ll T\delta^{1-\ep}.
\end{equation}

For each $j_{0}\leq i<j\leq j_{1}-1$, we have
\begin{align*}
\E (N_{i}-N_{i}^{\sign})(N_{j}-N_{j}^{\sign})\quad  \leq\quad & \sum_{k,l=2}^{n}kl\P (N_{i}=k, N_{j}=l).\nonumber
\end{align*}
Let $k_{0}:=\delta^{-1/100}$. We split the right-hand side into three sums: $2\le k, l\le k_0$ for the first sum, $k_0<k\le n$ and $2\le l\le n$ for the second sum, and $2\le k\le n$ and $k_0< l\le n$ for the third sum, and denote the corresponding sums by $K_1, K_2, K_3$, respectively.

By Lemma \ref{multipleroot}, letting $r_{ij}:=r(x_{i+1}, x_{j+1})$ gives
\begin{equation}
K_{1}\ll k_{0}^{2}\left[(C  \delta)^{4(1-\ep)}+\frac{(C\delta)^{4(1-\ep)}}{\sqrt{1-r_{ij}^{2}}}\right]\ll \delta^{3}+\frac{\delta^{3}}{\sqrt{1-r_{ij}^{2}}}.
\end{equation}
For $K_{2}$, we use H\"older's inequality to get
\begin{align*}
K_{2} \quad \leq\quad&  \E \left(N_{i}N_{j}\mathbf{1}_{N_{i}\geq k_{0}+1}\mathbf{1}_{N_{j}\geq 2}\right)\leq\left(\E N_{i}^{2}\mathbf{1}_{N_{i}\geq k_{0}+1}\right)^{1/2}\left(\E N_{j}^{2}\mathbf{1}_{N_{j}\geq 2}\right)^{1/2}\\
\leq\quad & k_{0}^{-h+1}\left(\E N_{i}^{2h}\mathbf{1}_{N_{i}\geq 2}\right)^{1/2}\left(\E N_{j}^{2}\mathbf{1}_{N_{j}\geq 2}\right)^{1/2} \ll k_{0}^{-h+1} \delta^{2-\ep}\ll \delta^{3}
\end{align*}
where $h$ is a sufficiently large constant and in the next to last inequality, we used Lemma \ref{multipleroot} in a similar way as in \eqref{st2}.
%
%
Similarly, $K_{3}\ll \delta^{3}$. Hence,
$$
\E (N_{i}-N_{i}^{\sign})(N_{j}-N_{j}^{\sign})\ll  \delta^{3}+\frac{\delta^{3}}{\sqrt{1-r_{ij}^{2}}},
$$
and so
\begin{equation}\label{st}
\E (S-S^{\sign})^{2}\ll T\delta^{1-\ep}+\sum_{j_{0}\leq i<j<j_{1}} \left (\delta^{3}+\frac{\delta^{3}}{\sqrt{1-r_{ij}^{2}}}\right )\ll T \delta^{1-\ep}+\delta^{3}\sum_{j_{0}\leq i<j<j_{1}}\frac{1}{\sqrt{1-r_{ij}^{2}}}.
\end{equation}  

To complete the proof of the lemma, it remains to bound $1-r^{2}_{ij}$ from below. 
For each $0\le k\le n$, let 
$$c_{k, \rho} = \sqrt{\frac{(k+2\rho)\dots (1+2\rho)}{k!}}.$$
By Condition \eqref{cond_ci_rho}, $c_k = \Theta(c_{k, \rho})$ for all $k\ge N_0$ and thus, for all $x, y\in [a, b]$, 
$$V(x) = \sum_{k=0}^{n} c_{k}^{2}x^{2k} = \Theta\left (\sum_{k=0}^{n} c_{k, \rho}^{2}x^{2k}\right )$$
and
\begin{equation}\label{eq:1-r:rho}
1  -r^{2}(x, y) = \frac{\sum_{0\le i<k\le n} c_i^{2}c_k^{2} (x^{i}y^{k} - x^{k}y^{i})^{2}}{\left (\sum_{k=0}^{n} c_{k}^{2}x^{2k}\right )\left (\sum_{k=0}^{n} c_{k}^{2}y^{2k}\right )} = \Theta\left (\frac{\sum_{0\le i<k\le n} c_{i, \rho}^{2} c_{k, \rho}^{2} (x^{i}y^{k} - x^{k} y^{i})^{2}}{\left (\sum_{k=0}^{n} c_{k, \rho}^{2}x^{2k}\right )\left (\sum_{k=0}^{n} c_{k, \rho}^{2}y^{2k}\right )}\right ).
\end{equation}
Therefore, in order to bound $1 - r^{2}_{ij}$ from below, it suffices to assume that $c_k = c_{k, \rho}$ for all $0\le k\le n$ for the rest of the proof of Lemma \ref{lmm:tosignchange}.

For $c_k = c_{k, \rho}$, we have for every $x\in[1-a_n, 1-b_n]$,  
\begin{equation}\label{boundvariance}
V(x)=  \frac{1+O(\ep_0)}{(1-x^{2})^{2\rho+1}}.
\end{equation}
where $ \varepsilon_0=\exp\left (-(\log\log n)^{1+2\varepsilon}\right )$. We defer the simple verification of \eqref{eq:1-r:rho} and \eqref{boundvariance} to Appendix \ref{proof:boundavariance}. 

Letting $x=x_{i+1}$ and $y=x_{j+1}$ yields
$$
r_{ij}=\frac{V(\sqrt{xy})}{\sqrt{V(x)V(y)}}=(1+O(\ep_0))\left(\frac{\sqrt{(1-x^{2})(1-y^{2})}}{(1-xy)}\right)^{2\rho+1}.
$$

Let $s_{ij}:=  \frac{\sqrt{(1-x^{2})(1-y^{2})}}{(1-xy)}$. 
To estimate $1-r_{ij}^{2}$, let us first estimate $1-s_{ij}^{2}$. We have
$$
1-s_{ij}^{2}=\frac{(x-y)^{2}}{(1-x+x(1-y))^{2}}=\frac{\left(e^{(j-i)\delta}-1\right)^{2}}{\left(e^{(j-i)\delta}+x\right)^{2}}\geq\frac{(j-i)^{2}\delta^{2}}{\left(e^{(j-i)\delta}+1\right)^{2}}.
$$
Thus, if $(j-i)\delta\leq 1$ then $1-s_{ij}^{2}\gg (j-i)^{2}\delta^{2}$ and if $(j-i)\delta\geq 1$ then $1-s_{ij}^{2}=  \frac{\left(e^{(j-i)\delta}-1\right)^{2}}{\left(e^{(j-i)\delta}+x\right)^{2}}\gg 1$.
Combining this with the assumption that $\delta\ge \exp\left (-(\log\log n)^{1+\varepsilon}\right )$, we have $\ep_0 =o\left (1-s_{ij}^{2}\right )$ for all $i<j$. 
This implies
\begin{align*}
1-r_{ij}^{2}=1-s_{ij}^{2(2\rho+1)}+o\left (1-s_{ij}^{2}\right )=\Theta\left (1-s_{ij}^{2}\right ) = \Theta\left (\frac{(x-y)^{2}}{(1-xy)^{2}}\right )
\end{align*}
and
\begin{eqnarray}
\sum_{j_{0} \leq  i<j\leq j_{1}-1}\frac{1}{\sqrt{1-r_{ij}^{2}}} 
&\ll&  \sum_{i<j<i+\delta^{-1}}\frac{1}{\sqrt{1-s_{ij}^{2}}}+ \sum_{j\geq i+\delta^{-1}}\frac{1}{\sqrt{1-s_{ij}^{2}}}\nonumber\\
& \ll  &  \sum_{i<j<i+\delta^{-1}}\frac{1}{(j-i)\delta}+ \sum_{j\geq i+\delta^{-1}}1
\ll  T\delta^{-2}\log\delta^{-1}+T^{2}\delta^{-2} .\nonumber
\end{eqnarray}
Plugging this into \eqref{st}, we obtain 
\begin{equation}
\E (S-S^{\sign})^{2}\ll T\delta^{1-\ep}+T^{2}\delta+T\delta\log\delta^{-1}\ll T^{2}\delta^{1-\ep},\nonumber
\end{equation}
completing the proof of Lemma \ref{lmm:tosignchange}.
\end{proof}

\subsection{Truncate the polynomial $P_n$ to get independence.}\label{subsection:truncate}
We now show that  $N^{\sign}(x, y)$ and $N^{\sign}(z, t)$  (in some rough sense) are independent,  whenever the intervals $(x, y)$ and $(z, t)$ are relatively far apart. This allows us to approximate  $N_n(\mj)$ by a  sum of independent random variables, from which we can derive a  Central Limit Theorem. 

For any $x \in[1-a_n, 1-b_n]$, let
\begin{equation}\label{def:mx:Mx}
A_{x}=\log(1-x)^{-1}, m_{x}=(1-x)^{-1}A_{x}^{-\alpha}, \quad \text{and}\quad  M_{x}=\alpha(1-x)^{-1}\log A_{x}
\end{equation}
where $\alpha$ is a large constant to be chosen.

Define a truncated version of $P_n$ by
$$
Q(x)=\sum_{j=m_{x}}^{M_{x}}c_{j}\xi_{j}x^{j}.
$$

We get $Q$ from $P_n$ by a truncation in which the truncation points $m_x$ and $M_x$ depend on the value of $x$. 
Let
$$
\rho'=\min\{1, 1+2\rho\}>0.
$$

The following lemma asserts that $Q$ is a good approximation of $P_n$ and that $Q(x)$ and $Q(y)$ are independent when $x$ and $y$ are far apart. 
\begin{lemma}\label{lm:obser1}
  For every $x\in[1-a_n, 1-b_n]$, it holds that
\begin{equation}\label{eq:truncation:PQ}
0\leq \Var P_n (x)-\Var Q(x) = \Var \left (P_n(x)-Q(x)\right ) \ll A_{x}^{-\alpha\rho'}\E P_n^{2}(x),
\end{equation}
 Moreover, if $1-a_n\leq x<y\leq 1-b_n$ and if $  \log\frac{1-x}{1-y}\geq 2\alpha\log\log n$ then $Q(x)$ and $Q(y)$ are independent because
$$
M_{x}<m_{y}.
$$
\end{lemma}
\begin{proof}[Proof of Lemma \ref{lm:obser1}] 
Since $b_n\ge 1/n$, for all $x\in [1-a_n, 1-b_n]$, $1-x\ge b_n\ge 1/n$. We write $x = 1-\frac{1}{L}$. 

By \eqref{varbound}, on the right-most side of \eqref{eq:truncation:PQ}, we have
$$ \Var P_n(x) = \frac{\Theta(1)}{(1-x)^{2\rho+1}} = \Theta\left (L^{2\rho+1}\right ).$$
On the other side, we have
$$\Var P_n (x) - \Var Q_n (x)\le \sum_{i=0}^{N_0} c_i^{2} x^{2i}+\sum_{i=N_0}^{m_x} c_i^{2} x^{2i} + \sum_{i=M_x}^{n} c_i^{2} x^{2i}\ll 1 + \sum_{N_0}^{m_x} i^{2\rho} + \sum_{i=M_x}^{n} i^{2\rho} e^{-2i/L}.$$
By the same argument as in \eqref{eq:varbound:L:3}, the right-most sum is at most 
\begin{eqnarray}
\sum_{i=M_x}^{n} i^{2\rho} e^{-2i/L}&\ll& L^{2\rho+1}\int_{M_x/L-1}^{\infty}t^{2\rho} e^{-2t} dt \nonumber\\
 &\ll& L^{2\rho+1} e^{-\alpha \log\log L} \ll (\log L)^{-\alpha} \E P_n^2(x) \ll  A_x^{-\alpha \rho'}\E P_n^2(x)\nonumber
\end{eqnarray}
where we used $M_x = \alpha L\log \log L$ and $A_x = \log L$ by the definition of $M_x$ and $A_x$.
Thus, 
$$\Var P_n (x) - \Var Q_n (x)\ll 1 + m_x^{2\rho+1} + A_x^{-\alpha \rho'}\E P_n^2(x)\ll A_x^{-\alpha \rho'}\E P_n^2(x)$$
where we used $m_x = L(\log L)^{-\alpha}$ by the definition of $m_x$. This proves \eqref{eq:truncation:PQ}.

As for the second part of Lemma \ref{lm:obser1}, writing $x = 1-\frac{1}{L}$ and $y = 1 - \frac{1}{K}$, we have $1\ll L\le K\le n$ and $\log \frac{K}{L}\ge 2\alpha\log\log n$, so
$$M_x=\alpha  L\log \log L\le L \log^{\alpha}n \le K\log^{-\alpha}n\le K\log^{-\alpha}K = m_y.$$
This proves Lemma \ref{lm:obser1}.
\end{proof}

\subsection{Approximating sign changes of $P_n$ by those of $Q$: short intervals}
Let 
\begin{equation}\label{def:Ntrun}
N^{\trun}_{P_n}(x, y)=N^{\trun}(x, y):=\frac{1}{2}-\frac{1}{2}\sign(Q(x)Q(y))
\end{equation}
be the sign change of $Q$ on the interval $(x, y)$. In the next lemma, we show that $N^{\trun}$ is a good approximation of the corresponding sign change $N_n^{\sign}$ of $P_n$ defined in \eqref{def_Nsign}.

\begin{lemma}[Approximation  by truncation I]\label{lm:truncation} Assume that the $\xi_i$ are iid standard Gaussian. Let $C$ be any positive constant. Let $1-a_n\leq x<y\leq 1-b_n$ with $\log\frac{1-x}{1-y}\le 1/C$. Then
$$
\E \left(N_n^{\sign}(x, y)-N^{\trun}(x, y)\right)^{2}\ll A_{x}^{-\alpha\rho'/3}.
$$
\end{lemma}
\begin{proof}[Proof of Lemma \ref{lm:truncation}] 
 Using the formula
 $$\sign(a)=\frac{1}{\pi}\int_{\mathbb{R}}t^{-1} \sin(ta)dt,$$
 we have
\begin{eqnarray}
 &&N_n^{\sign}(x, y)-N^{\trun}(x, y)\nonumber\\
 &&\quad =\frac{1}{2\pi^{2}}\int_{\mathbb{R}}\int_{\mathbb{R}}t^{-1}u^{-1}\left(\sin(t\bar{Q}(x))\sin(u\bar{Q}(y))-\sin(t\bar{P}_n(x))\sin(u\bar{P}_n(y))\right)dtdu\nonumber
\end{eqnarray}
 where
 $$
 \bar{Q}(x):=\frac{Q(x)}{\sqrt{V(x)}},\quad \bar{Q}(y):=\frac{Q(y)}{\sqrt{V(y)}}, \quad \bar{P_n}(x):=\frac{P_n(x)}{\sqrt{V(x)}}\quad \text{and}\quad \bar{P_n}(y):=\frac{P_n(y)}{\sqrt{V(y)}}.
 $$
 
 Decompose the plane $\mathbb{R}\times \mathbb{R}$ of $(t, u)$ into two regions: the square  \newline $\left\{(t, u): A_{x}^{-\alpha\rho'/6}\leq|t|, |u|\leq A_{x}^{\alpha\rho'/3}\right\}$ and its complement. 
 We denote the corresponding integrals on these regions by $I_{1}$ and $I_{2}$, respectively.
 
 First, we show that the contribution from $I_{2}$ is negligible. Indeed, using the estimates
 \begin{align}
 \left|  \int_{|t|\leq\varepsilon}t^{-1}\sin(ta)dt\right|\ll & \min\left \{|a\varepsilon|, 1\right \} ,\nonumber \\
 \left|\int_{|t|\geq M}t^{-1}\sin(ta)dt\right|\ll & \min\left \{\frac{1}{|aM|}, 1\right \},\label{eq:sin:bound2} 
 \end{align}
 we obtain
 \begin{align*}
 |I_{2}| \quad \ll \quad  & (|  \bar{P_n}(x)|+|\bar{P_n}(y)|+|\bar{Q}(x)|+\bar{Q}(y)|)A_{x}^{-\alpha\rho'/6} \\
 &+\min\{1, |\bar{P_n}(x)|^{-1}A_{x}^{-\alpha\rho'/3}\} +\min\{1, |\bar{P_n}(y)|^{-1}A_{x}^{-\alpha\rho'/3}\}\\
 &+\min\{1, |\bar{Q}(x)|^{-1}A_{x}^{-\alpha\rho'/3}\}+\min\{1, |\bar{Q}(y)|^{-1}A_{x}^{-\alpha\rho'/3}\} .
 \end{align*}
 
 From this and the Gaussianity of $\bar P_n$ and $\bar Q$, we have
 $$
 \E I_{2}^{2}\ll A_{x}^{-\alpha\rho'/3}+\E \min\{1, Z^{-2}A_{x}^{-2\alpha\rho'/3} \}\ll A_{x}^{-\alpha\rho'/3}
 $$
 where $Z\sim \mathcal  N(0,1)$.
 
 For $I_{1}$, we need to make use of the cancellation between $P_n$ and $Q$. We rewrite $I_1$ as 
 \begin{eqnarray*}
 I_{1} &= & \frac{1}{\pi^{2}}\int_{A_{x}^{-\alpha\rho/6}}^{A_{x}^{\alpha\rho'/3}}\int_{A_{x}^{-\alpha\rho/6}}^{A_{x}^{\alpha\rho'/3}}t^{-1}u^{-1}\sin(t\bar{Q}(x))\cos\left (u\frac{\bar{Q}(y)+\bar{P_n}(y)}{2}\right )\sin\left (u\frac{\bar{Q}(y)-\bar{P_n}(y)}{2}\right )dtdu\\
 &+&\frac{1}{\pi^{2}}\int_{A_{x}^{-\alpha\rho/6}}^{A_{x}^{\alpha\rho'/3}}\int_{A_{x}^{-\alpha\rho/6}}^{A_{x}^{\alpha\rho'/3}}t^{-1}u^{-1}\sin(u\bar{P_n}(y))\cos\left (t\frac{\bar{Q}(x)+\bar{P_n}(x)}{2}\right )\sin\left (t\frac{\bar{Q}(x)-\bar{P_n}(x)}{2}\right )dtdu.
 \end{eqnarray*}
 Using $\left | \int_{b}^{c}t^{-1}\sin(ta)dt\right |\ll 1$ for all $0<b<c$, $\left|  \frac{\sin(a)}{a}\right|\leq 1$ for all $a\neq 0$, and \eqref{eq:truncation:PQ}, we get
 %
 \begin{eqnarray}
 \E I_{1}^{2} &\ll & \E\left [ \int_{A_{x}^{-\alpha\rho'/6}}^{A_{x}^{\alpha\rho'/3}}|\bar{Q}(y)-\bar{P_n}(y)|+|\bar{Q}(x)-\bar{P_n}(x)|dt\right ]^{2}\nonumber\\
 &\ll& A_{x}^{2\alpha\rho'/3} (\E |\bar{Q}(y)-\bar{P_n}(y)|^{2}+\E |\bar{Q}(x)-\bar{P_n}(x)|^{2})\ll A_{x}^{-\alpha\rho'/3} \nonumber
 \end{eqnarray}
 where we used Lemma \ref{lm:obser1} (recalling that the random variables $\xi_i$ are iid standard Gaussian and hence have mean 0) to get 
 $$\E |\bar{Q}(y)-\bar{P_n}(y)|^{2} = \Var (\bar{Q}(y)-\bar{P_n}(y)) \ll A_{y}^{-\alpha\rho'}  \ll A_{x}^{-\alpha\rho'}.$$
 
 This completes the proof of Lemma \ref{lm:truncation}.
 \end{proof}

 \subsection{Approximating sign changes of $P_n$ by those of $Q$: long intervals}
\begin{lemma}[Approximation  by truncation II]\label{lm:totruncation} Assume that the $\xi_i$ are iid standard Gaussian. There exist constants $C, C'$  such that the following holds. Let $T>1/C$ and $a, b$  be such that $1-a_n\leq a<b\leq 1-b_n$ and $\log\frac{1-a}{1-b}=T$. Let $j_0=\delta^{-1}\log(1-a)^{-1}$  and $j_{1}=\delta^{-1}\log(1-b)^{-1}$ where $\delta$ is any number in $(0, 1/C)$. Assume (without loss of generality) that $j_0$ and $j_{1}$  are integers and let 
	$x_{j}=1-\exp(-j\delta)$ for all $j=j_{0}, \ldots , j_{1}.$ 
Let
$$S^{\sign}=S_{a,b,\delta}^{\sign}=  \sum_{j=j_{0}}^{j_{1}-1}N_n^{\sign}(x_{j}, x_{j+1}) \quad\text{and} \quad   S^{\trun}=S^{\trun}_{a,b,\delta}=\sum_{j=j_{0}}^{j_{1}-1}N^{\trun}(x_{j}, x_{j+1}).$$
Then
$$
\E (S^{\trun}-S^{\sign})^{2}\leq C'\delta^{-2}T^{2}\left (\log \frac{1}{a_n}\right )^{-\alpha\rho'/3}=C'\delta^{-2}\left (\log\frac{1-a}{1-b}\right )^{2}\left (\log \frac{1}{a_n}\right )^{-\alpha\rho'/3}.
$$
\end{lemma}
\begin{proof}[Proof of Lemma \ref{lm:totruncation}]  

By Lemma \ref{lm:truncation}, we have
\begin{eqnarray}
\E \left(S^{\trun}-S^{\sign}\right)^{2} &\le& \left (\sum_{j=j_{0}}^{j_{1}-1}\left [\E (N^{\trun}(x_{j}, x_{j+1})-N_n^{\sign}(x_{j}, x_{j+1}))^{2}\right ]^{1/2}\right)^{2}\nonumber\\
&\ll& \left(\sum_{j=j_{0}}^{j_{1}-1}A_{x_{j}}^{-\alpha\rho'/6}\right)^{2}\ll \left(\sum_{j=j_{0}}^{j_{1}-1}(j\delta)^{-\alpha\rho'/6}\right)^{2}\nonumber.
\end{eqnarray}
By the definition of $j_0$ and $j_1$, we get
\begin{eqnarray}
\sum_{j=j_{0}}^{j_{1}-1}(j\delta)^{-\alpha\rho'/6} \le \delta^{-\alpha\rho'/6}(j_1-j_0)  j_0^{-\alpha\rho'/6} = \delta^{-1-\alpha\rho'/6}T j_0^{-\alpha\rho'/6}  \le  \delta^{-1}T  \left (\log \frac{1}{a_n} \right ) ^{-\alpha\rho'/6}\nonumber
\end{eqnarray}
proving Lemma \ref{lm:totruncation}.
\end{proof}

\subsection{Control of the fourth moment}
The following lemma controls the fourth moment of $S^{\trun}$.
\begin{lemma}[Bounded forth moment]\label{lm:boundforthmoment}
Under the setting of Lemma \ref{lm:totruncation} and an additional assumption that $\delta\geq \left (\log  \frac{1}{a_n}\right )^{-\alpha\rho'/24}$,  we have
	\begin{equation}\label{st5}
 \E \left (S^{\trun}_{a,b,\delta}-\E S^{\trun}_{a,b,\delta}\right )^{4}\ll T^{2}(\log\log n)^{2}= \left (\log\frac{1-a}{1-b}\right )^{2}(\log\log n)^{2}.
	\end{equation}
\end{lemma}
\begin{proof}[Proof of Lemma \ref{lm:boundforthmoment}]
Let $C_{0}$ be the constant in Lemma \ref{multipleroot}. 

\textbf{Case 1.} $ T\leq 1$. Since $T\gg 1$, it suffices to show that
\begin{equation}\label{st6}
\E (S^{\trun}_{a,b,\delta})^{4}\ll 1.
\end{equation}
For simplicity, we write $S^{\trun}$ for $S^{\trun}_{a,b,\delta}$. Let $S^{\sign}=S_{a,b,\delta}^{\sign}$ as in the setting of Lemma \ref{lm:totruncation}. 
By the definition of sign changes, we have  with probability $1$,
$$
S^{\trun} \ll j_1-j_0\ll \delta^{-1} \quad\text{and}\quad S^{\sign}\ll j_1-j_0\ll \delta^{-1}.
$$
%
Hence, by Lemma \ref{lm:totruncation}, H\"older's inequality, and the assumption that $\delta\geq \left (\log  \frac{1}{a_n}\right )^{-\alpha\rho'/24}$, we have
$$
\left |\E (S^{\trun})^{4}-\E (S^{\sign}) ^4\right |\ll \delta^{-3}\E \left |S^{\trun}-S^{\sign}\right |\ll \delta^{-4}\left (\log  \frac{1}{a_n}\right )^{-\alpha\rho'/6}\ll 1.
$$
Thus, it suffices to show that $\E (S^{\sign}) ^4\ll 1$. Since $N_n^{\sign} (x, y) \le N_n(x, y)$ for any interval $(x, y)$, 
$$\E (S^{\sign}) ^4\le  \E N_n^{4}(a, b).$$
Partition the interval $(a, b)$ into smaller intervals $(x, y)$ such that $\log\frac{1-x}{1 - y} = \frac{1}{2C_0}$. Since $\log\frac{1-a}{1 - b} = T$, the number of such sub-intervals is $2C_0T$. By \eqref{ine:multiple1}, for each of these intervals $(x, y)$, we have
$$\E N_n^{4}(x, y) \ll \sum_{k=1}^{\infty} k^{4} 2^{-k/2}\ll 1.$$
Using this and the assumption that $T\le 1$ of \textbf{Case 1}, we have $\E N_n^{4}(a, b)\ll 1$ as desired.
%

\textbf{Case 2.} $ T>1$. We decompose the sum in $S^{\trun}-\E S^{\trun}$ into blocks of size $\mu := \delta^{-1}$ of the form
\begin{align*}
X_{k}&=\sum_{j=j_{0}+(k-1)\mu}^{j_{0}+k\mu-1}\left (N^{\trun}(x_{j}, x_{j+1})-\E N^{\trun}(x_{j}, x_{j+1})\right )
\end{align*}
for each $k=1, \dots, j_2$, where  $j_2 = (j_1-j_0)\mu^{-1}$ is the number of blocks. Notice that $j_2 \ll T$. 

We have
\begin{eqnarray}
\E \left (S^{\trun}-\E S^{\trun}\right )^{4}&=&\E \left (\sum_{k=1}^{j_{2}}X_{k}\right )^{4}=\sum_{k=1}^{j_{2}}\E X_{k}^{4}+4\sum_{k\neq l}\E X_{k}^{3}X_{l}+6\sum_{k<l}\E X_{k}^{2}X_{l}^{2}\nonumber\\
&&+12\sum_{l<p;k\neq l,p}\E X_{k}^{2}X_{l}X_{p}+24\sum_{k<l<p<q}\E X_{k}X_{l}X_{p}X_{q}\nonumber\\
&=:&I_{1}+4I_{2}+6I_{3}+12I_{4}+24I_{5}.\nonumber
\end{eqnarray}
Note that each $X_{k}$ is of the form $S^{\trun}_{a',b',\delta}-\E S^{\trun}_{a',b',\delta}$ for some $a', b'$ that satisfy $  \log\frac{1-a'}{1-b'}\leq 1$. Thus, \eqref{st6} implies that $\E X_{k}^{4}\ll 1$ for all $k$. By H\"older's  inequality, each term in the summation of $I_1, \dots, I_5$ is of order $O(1)$ and so, 
$$I_1\ll j_2\ll T, \quad I_2+I_3\ll j_2^{2}\ll T^{2}.$$
To bound $I_{4}$ and $I_5$, we use the independence in Lemma \ref{lm:obser1} to conclude that if $k_2-k_1\ge 3\alpha\log\log n$ then $X_{k_2}$ and $(X_1, \dots, X_{k_1})$ are independent. Together with the fact that  $\E X_{k}=0$ for all $k$, we observe that most terms in the sums $I_4, I_5$ are zero. Ignoring these zero terms, we have 
$$I_{4} =   \sum_{l<p\leq l+C\log\log n}\E X_{k}^{2}X_{l}X_{p}\ll j_2^{2}\log \log n\ll T^{2}\log\log n,
$$
and
$$
I_5= \sum_{l-C\log\log n\leq k<l<p<q\leq p+C\log\log n} \E X_{k}X_{l}X_{p}X_{q}\ll T^{2}(\log\log n)^{2} .
$$
Putting the above bounds together, we obtain  Lemma \ref{lm:boundforthmoment}.
\end{proof}
\subsection{Proof of Lemma \ref{lm:CLT_gau}} 
Using the results in Sections \ref{subsection:approximate-sign} and \ref{subsection:truncate}, we shall approximate $N_n(\mj)$ by a sum of independent random variables to prove that it satisfies the CLT.
 We again recall that in this proof, the $\xi_i$ are iid standard Gaussian as mentioned at the beginning of this section. Recall the hypothesis \eqref{eq:cond:ab:CLT} that 
	\begin{equation}\label{eq:hypothesis:ab}
	(\log n)^{2}/n\le b_n<a_n\le \exp\left (-(\log n)^{c}\right ), \quad \log\frac{a_n}{b_n}=\Theta(\log n), \quad \text{and} \quad \Var  N_n(\mj) \gg \log n.
	\end{equation}
In particular, $a_n$ satisfies Condition \eqref{cond-a-log}. Let $\alpha, \beta$ be any constants satisfying 
	\begin{equation}\label{condab1}
	\beta\geq 3 \quad\text{and} \quad 2\beta+3\le c\alpha\rho'/24.
	\end{equation}
Let 
\begin{equation}\label{def:T:delta}
T:=\log\frac{a_n}{b_n}=\Theta(\log n), \quad \delta:=(\log n)^{-\beta},
\end{equation}
$$j_0:=\delta^{-1}\log  \frac{1}{a_n}\quad \text{and} \quad j_{1}:=\delta^{-1}\log  \frac{1}{b_n}.$$
We have $j_{1}-j_{0}=\delta^{-1}T$. Let
$$q:=\delta^{-1}T^{1/8} \quad \text{and} \quad p:=\delta^{-1}T^{1/2}.$$
Observe that $q=o(p)$ and $q$ grows with $n$.
For simplicity, we will assume that $j_{0}, j_{1}, p$ and $q$ are integers. In the case that they are not, we only need to replace them by their integer part. As before, let $x_{j}=1-\exp(-j\delta)$ for $j=j_{0},\ldots,j_{1}.$

Let $N^{\trun}_{P_n}(x_{j}, x_{j+1})$ be defined as in \eqref{def:Ntrun}. By Lemmas \ref{lmm:tosignchange} and \ref{lm:totruncation}, we can approximate $N_{n}(\mj \cap (0,1))$ by
$$
{S_1}^{\trun}:=S^{\trun}_{1-a_n,1-b_n,\delta}=\sum_{j=j_{0}}^{j_{1}-1}N^{\trun}_{P_n}(x_{j}, x_{j+1})
$$
and get an error term
\begin{equation}\label{}
\E \left(N_{n}(\mj \cap (0,1))-S_1^{\trun}\right)^{2}\ll T^{2}\delta^{1-\ep}+ T^{2}\delta^{-2}\left (\log \frac{1}{a_n}\right )^{-\alpha\rho'/3} =o(\log n)\nonumber
\end{equation}
where in the last inequality, we used \eqref{eq:hypothesis:ab} and \eqref{condab1}.

Combining this with the assumption that $\Var N_n(\mj)\gg \log n$, we get
 \begin{equation}\label{err1}
\E \left(N_{n}(\mj \cap (0,1))-S_1^{\trun}\right)^{2}=o(\log n)=o(\Var N_{n}(\mj)).
\end{equation}

Similarly, for the interval $\mj\cap (-1, 0)$, we approximate the number of real roots by 
$$S_2^{\trun} := \sum_{j=j_{0}}^{j_1-1} N^{\trun}_{P_n} (-x_{j+1}, -x_j).$$
And for the intervals $\mj\cap (1, \infty)$ and $\mj \cap (-\infty, -1)$, we respectively use 
$$S_3^{\trun} := \sum_{j=j_{0}}^{j_1-1} N^{\trun}_{R_n} (x_{j}, x_{j+1})\quad \text{and } S_4^{\trun} := \sum_{j=j_{0}}^{j_1-1} N^{\trun}_{R_n} (-x_{j+1}, -x_j)$$
where $R_n(x) = \frac{x^{n}}{c_n}P_n(x^{-1}) = \sum_{i=0}^{n} \frac{c_{n-i}}{c_n}\xi_{n-i} x^{i}$.
Let $S^{\trun} := \sum_{k=1}^{4} S^{\trun}_k$. We note that all of the lemmas proven earlier in this section hold for $R_n$ in place of $P_n$ (with  the value of $\rho$ being changed to  $0$ as in Section \ref{section:reduction:01}). From \eqref{err1} and its analog for $S_2^{\trun}, S_3^{\trun}, S_4^{\trun}$, we have
\begin{equation} 
\E (N_{n}(\mj)-S^{\trun})^{2}=o(\log n)=o(\Var N_{n}(\mj)).\nonumber
\end{equation}
Making use of Lemma \ref{lm:obser1}, we now approximate $S^{\trun}$ by a sum of independent random variables $Z_{k}, W_k$ as follows. Let
$$Z_{k}=  \sum_{j=j_{0}+kp+kq}^{j_{0}+(k+1)p+kq-1}\left (N^{\trun}_{P_n}(x_{j}, x_{j+1})+N^{\trun}_{P_n}(-x_{j+1}, -x_{j})\right ),$$
and
$$W_{k}=  \sum_{j=j_{0}+kp+kq}^{j_{0}+(k+1)p+kq-1}\left (N^{\trun}_{R_n}(x_{j}, x_{j+1})+N^{\trun}_{R_n}(-x_{j+1}, -x_{j})\right),\quad k=0, \dots, l-1$$

where
$$
l=\frac{j_{1}-j_{0}}{p+q}=\Theta(T^{1/2}) .
$$

By Lemma \ref{lm:obser1}, the random variables $Z_0, \dots, Z_{l-1}$ are mutually independent because $ q\delta=T^{1/8}\geq 2\alpha \log\log n$. Similarly for the random variables $W_0, \dots, W_{l-1}$. Moreover, all random variables $Z_0, \dots, Z_{l-1}, W_0, \dots, W_{l-1}$ are mutually independent because the $Z_s$ only involve the random variables $\xi_r$ where $r\le M_{1-b_n}\le n/2$ (by the definition \eqref{def:mx:Mx} and the left-most inequality in \eqref{eq:hypothesis:ab}) while the $W_s$ only involve the random variables $\xi_{n-r}$ where, again, $r\le M_{1-b_n}\le n/2$.

To evaluate the accuracy of the approximation of $S^{\trun}$ by $\sum_{k}(Z_k+W_k)$, consider 
$$
S^{\trun}-\sum_{k=0}^{l-1} (Z_{k}+W_k)=\sum_{k=0}^{l-1} (X_k+Y_{k})
$$
where
$$X_{k}=  \sum_{j=j_{0}+(k+1)p+kq}^{j_{0}+(k+1)p+(k+1)q-1}\left (N^{\trun}_{P_n}(x_{j}, x_{j+1})+ N^{\trun}_{P_n}(-x_{j+1}, -x_{j})\right ) , \quad \text{for }k=0, 1, \ldots , l-1,
$$
and  $Y_k$ are defined similarly with respect to $R_n$.

By Lemma \ref{lm:obser1}, the random variables $X_0, \dots, X_{l-1}, Y_0, \dots, Y_{l-1}$ are also mutually independent. Note that each $X_k, Y_{k}$ is of the form $S^{\trun}_{a,b,\delta}$ defined in Lemma \ref{lm:totruncation} for some $a$ and $b$ with $  \log\frac{1-a}{1-b}=q\delta=T^{1/8}.$ 
%
By \eqref{condab1} and the definition of $\delta$ in \eqref{def:T:delta}, $\delta=(\log n)^{-\beta}\geq \left (\log  \frac{1}{a_n}\right )^{-\alpha\rho'/24}$; this  allows us to use Lemma \ref{lm:boundforthmoment} to get
$$\E (X_{k}-\E X_{k})^{4}\ll  q^{2}\delta^{2} (\log\log n)^{2}  \quad\text{for all } k=0, \ldots, l-1.$$

 One can obtain a similar estimate for $Y_k$.
Thus, the error term of the approximation of $S^{\trun}$ by $  \sum_{k=0}^{l-1}(Z_{k}+W_k)$ has variance
\begin{eqnarray}
\Var \left(  \sum_{k=0}^{l-1}(X_k+Y_{k})\right)&=&\sum_{k=0}^{l-1}\Var X_{k}+\sum_{k=0}^{l-1}\Var Y_{k}\ll \sum_{k=0}^{l-1}q \delta  \log\log n  =o(\log n) .\nonumber
\end{eqnarray}

Combining this with \eqref{err1}, we get
\begin{equation}\label{err3}
\Var  \left(N_{n}(\mj)-  \sum_{k=0}^{l-1}(Z_k+W_k)\right)=o(\log n)=o\left (\Var N_{n}(\mj)\right ).
\end{equation}
%
The sum $\sum_{k=0}^{l-1}(Z_k+W_k)$ is a sum of independent random variables satisfying forth moment bound
\begin{eqnarray}
\sum_{k=0}^{l-1}\E (Z_{k}-\E Z_{k})^{4}+\sum_{k=0}^{l-1}\E (W_{k}-\E W_{k})^{4}&\ll& \sum_{k=0}^{l-1} p^{2}\delta^{2}(\log\log n)^{2}\nonumber\\
&=& o\left (\log ^{2}n\right ) =o\left(\Var \sum_{k=0}^{l-1}(Z_{k}+W_k)\right)^{2}\nonumber
\end{eqnarray}
where in the first inequality, we used Lemma \ref{lm:boundforthmoment}.
By the Lyapunov Central Limit Theorem (see for example, \cite{patrick1995probability}), the sum $\sum_{k=0}^{l-1}(Z_{k}+W_k)$ satisfies the Central Limit Theorem.

This and \eqref{err3} imply that $N_n(\mj)$ also satisfies the Central Limit Theorem, completing the proof of Lemma \ref{lm:CLT_gau}.
\qed
%
%
%

\section{Proof of Lemma \ref{lm:var_gau}} \label{sec:prop:var_gau}
 Since in  this section we only deal with Gaussian random variables, we again use $\xi_i$ to denote iid standard Gaussian variables (instead of $\tilde \xi_i$). This would help avoid complicated notation (such as double superscripts) later on.  By symmetry of the Gaussian distribution, we can assume that $c_i\ge 0$ for all $i$.

Let 
 \begin{equation}  \label{choice:an:bn:var}
 a_n=\exp\left (-2\log^{1/5}n\right ) \quad \text{and}\qquad b_n= \frac{1}{a_n n}
 \end{equation}
 and
 $$\mj := \pm (1-a_n, 1-b_n) \cup \pm (1-a_n, 1-b_n)^{-1}.$$
 Note that this $a_n$ satisfies Condition \eqref{cond-a-log}.

 By Proposition \ref{prop:boundedge}, 
 \begin{equation} \nonumber
 \E N_{n}^{2}\left (\R\setminus \mj\right )\ll  \log^{4}\frac{1}{a_n} =o (\log n).
 \end{equation}
Thus, to prove Lemma \ref{lm:var_gau}, it suffices to show that 
 \begin{equation}
 \Var N_{P_n}(\mj) \gg \log n.\nonumber
 \end{equation}
  We have
\begin{eqnarray}
N_{P_n} (\mj) &=&  N_{P_n} \left(\mj \cap [-1, 1]\right ) + N_{P_n}  \left(\mj \setminus [-1, 1]\right ) \nonumber\\
&=&  N_{P_n} \left(\mj \cap [-1, 1]\right ) + N_{R_n}  \left(\mj \cap [-1, 1]\right ) \nonumber,
\end{eqnarray}
where $R_n(x) = \frac{x^{n}}{c_n} P_n(x^{-1}) = \sum_{i=0}^{n} \frac{c_{n-i}}{c_n} \xi_{n-i} x^{i}$. 

Since $\Var (X+Y) =\Var X + \Var Y + \Cov (X,Y) \ge \Var X + \Cov (X,Y) $ for any two real random variables $X$ and $Y$, it suffices to show that 
\begin{equation}
\Var N_{R_n}  \left(\mj \cap [-1, 1]\right ) = \Omega(\log n) \label{varR}
\end{equation}
and
\begin{equation}
\Cov \left (N_{P_n}  \left(\mj \cap [-1, 1]\right ) , N_{R_n}  \left(\mj \cap [-1, 1]\right ) \right ) = o(\log n) \label{covPR}.
\end{equation}

\subsection{Universality for $R_n$}
In order to verify  \eqref{varR}, we  use the  universality method in a novel way. Instead of swapping the random variables $\xi_i$, we swap the deterministic 
coefficients $c_i$. This allow us to  couple $R_n$  with the Kac polynomial and the desired bound follows by known results concerning the variance of the Kac polynomial.  This swapping is possible thanks to the fact that the ``important" coefficients  are   $\frac{c_{n-i}}{c_n}$ which are close to 1 by \eqref{cond:ci:kac}.

Let 
$$\hat R_n(x) := \sum _{i=0}^{n} \xi_{n-i} x^{i}$$
be the corresponding Kac polynomial. We prove the following analogs of Theorem \ref{thm:distribution_universality} and Corollary \ref{cor:moment_universality} for $R_n$ and $\hat R_n$. 

\begin{proposition}\label{prop:Rdistribution}
	Assume that the $\xi_i$ are iid standard Gaussian. Let $\beta>0$ be any constant. There exists a constant $C>0$ such that for every function $F:\R\to \R$ whose derivative up to order 3 are bounded by 1 and for every $n$, we have
	$$\left |\E F\left (N_{R_n}\left (\mj\cap [-1, 1]\right )\right ) - \E F\left (N_{\hat R_n}\left (\mj\cap [-1, 1]\right )\right )\right |\le C(\log n)^{-\beta}.$$
\end{proposition}

\begin{proposition}\label{prop:Rmoment}
	Assume that the $\xi_i$ are iid standard Gaussian. Let $\beta>0$ be any constant. There exists a constant $C>0$ such that for every $n$, we have
	$$\left |\E \left (N_{R_n}^{k}\left (\mj\cap [-1, 1]\right )\right ) - \E \left (N_{\hat R_n}^{k}\left (\mj\cap [-1, 1]\right )\right )\right |\le C(\log n)^{-\beta}$$
for $k=1, 2$. In particular, 
	$$\left |\Var \left (N_{R_n} \left (\mj\cap [-1, 1]\right )\right ) -\Var \left (N_{\hat R_n} \left (\mj\cap [-1, 1]\right )\right )\right |\le C(\log n)^{-\beta}.$$
\end{proposition}

Proposition \ref{prop:Rdistribution} implies Proposition \ref{prop:Rmoment}, using  the same arguments as in the proof of Corollary \ref{cor:moment_universality}. \begin{proof}[Proof of  Proposition \ref{prop:Rdistribution}] We use the same arguments as in the proof of Theorem \ref{thm:distribution_universality} with the following modifications. First, $P_n$ is replaced by $R_n$ and $\tilde P_n$ is replaced by $\hat R_n$, and all of the $\delta^{\alpha}$ in the former for a small constant $\alpha$ with be replaced by $(\log n)^{-\beta'}$ for a large constant $\beta'$. For example, Lemma \ref{lm:repulsion} is replaced by the following variant that can be proved using the same argument.
	\begin{lemma}\label{lm:nonrepulsion} 
		Assume that the $\xi_i$ are iid standard Gaussian. Let $\delta\in [b_n, a_n]$. For any constant $\gamma>0$ and $x\in \R$ with $|x|\in  [1-\delta -\delta (\log n)^{-\gamma}, 1-\delta/2 +\delta  (\log n)^{-\gamma}]$, we have
		$$\P \left ( N_{R_n} B\left (x, \delta (\log n)^{-\gamma}\right )\ge 2\right )\ll (\log n)^{-3 \gamma/2}.$$
	\end{lemma}
	%
	%
	The only remaining difference compared to the proof of Theorem \ref{thm:distribution_universality} is in the proof of the analog of Lemma \ref{lm:log-universality-smooth}, namely for $\delta_0 = a_n, \delta_1 = a_n/2, \dots, \delta_{M-1}  = a_n/2^{M-1}$ and $\delta_M := \max\{1/n, b_n\}$ ($M$ is the largest integer such that $\delta_{M-1}> \max\{1/n, b_n\}$), and for $m_i = (\log n)^{\beta}$,
	\begin{lemma} \label{lm:log-universality-smooth:R}
		Assume that the $\xi_i$ are iid standard Gaussian. Let $\beta'$ be any positive constant. Let $L : \mathbb{C}^{m_{1}+\cdots+m_{M}}\rightarrow \mathbb{R}$ be a smooth function with all derivatives up to order 3 being bounded by $(\log n)^{\beta'}$. 
		Then for every $w_{ik}$  in $B(1-3\delta_{i}/2,2\delta_{i}/3)$, we have
		\begin{equation}\label{eq:lm:log-universality-smooth:R}
		\left|\E L\left (\frac{R_n(w_{ik})}{\sqrt{V(w_{ik})}}\right )_{\substack{i=1,\ldots,M\\k=1,\dots,m_{i}}}-\E L\left (\frac{\hat{R}_n(w_{ik})}{\sqrt{V(w_{ik})}}\right )_{\substack{i=1,\ldots,M\\k=1,\dots,m_{i}}}\right|\ll (\log n)^{-\beta'},
		\end{equation}
		where $V(w) := \Var R_n(w)$.
	\end{lemma}
	
	Assuming this lemma, the rest of the proof of Theorem \ref{thm:distribution_universality} can be adapted in a straightforward manner to complete the proof of Proposition \ref{prop:Rdistribution}.
\end{proof}

\begin{proof}[Proof of Lemma \ref{lm:log-universality-smooth:R}]	While for Lemma \ref{lm:log-universality-smooth}, going from $P_n$ to $\tilde P_n$, we need to swap the general random variables $\xi_i$ to the Gaussian ones $\tilde \xi_i$, here, going from $R_n$ to $\hat R_n$, we need to swap the coefficients $\frac{c_{n-i}}{c_n}$ to $1$ and keep the Gaussian random variables $\xi_i$ intact. Keeping that in mind, we set for each $0\le i_0\le n+1$,
	$$R_{i_0}(z) := \sum_{i=0}^{i_0-1} \xi_{n-i} z^{i} + \sum_{i=i_0}^{n} \frac{c_{n-i}}{c_n} \xi_{n-i} z^{i}.$$
	We have $R_{0} = R_n$, $R_{n+1} = \hat R_n$ and $R_{i_0+1}$ is obtained from $R_{i_0}$ by replacing the coefficient $\frac{c_{n-i_0}}{c_n}$ by 1.
	
	The difference $d_{i_0}$ in \eqref{def_di0} for $0\le i_0\le n$ now becomes 
	\begin{equation}
	d_{i_{0}}:=\left|\E _{\xi_{n-i_{0}}}\hat{L}\left(\frac{c_{n-i_{0}}\xi_{n-i_{0}}w_{ik}^{i_{0}}}{c_n\sqrt{V(w_{ik})}}\right)_{ik}-\E _{\xi_{n-i_{0}}}\hat{L}\left(\frac{{\xi}_{n-i_{0}}w_{ik}^{i_{0}}}{\sqrt{V(w_{ik})}}\right)_{ik}\right|\label{new_def_di0}
	\end{equation}
	where $\hat L$ is obtained from $L$ by translation and thus has all derivatives up to order 3 bounded by $(\log n)^{\beta'}$. The task is to show that 
	\begin{equation}\label{eq:universality:R:d}
	\sum_{i_0=0}^{n+1} \E _{\xi_0, \dots, \xi_n} d_{i_0} \ll (\log n)^{-\beta'}.
	\end{equation}
	By the Taylor expansion of order 2, we get
	\begin{equation} \label{eq:Taylor:R}
	\hat{L}\left(\frac{c_{n-i_{0}}\xi_{n-i_{0}}w_{ik}^{i_{0}}}{c_n\sqrt{V(w_{ik})}}\right)_{ik}=\hat{L}(0)+\hat{L}_{1}+\frac{1}{2}\hat{L}_{2}+\err_{2},
	\end{equation}
	where 
	\begin{eqnarray}
	\hat{L}_{1}&:=&\left.\frac{\d\hat{L}\left(\frac{c_{n-i_{0}}\xi_{n-i_{0}}w_{ik}^{i_{0}}}{c_n\sqrt{V(w_{ik})}}t\right)_{ik}}{\d t}\right|_{t=0}\nonumber\\
	&&=\sum_{ik}\frac{\partial\hat{L}(0)}{\partial\Re(z_{ik})}\Re\left (\frac{c_{n-i_{0}}\xi_{n-i_{0}}w_{ik}^{i_{0}}}{c_n\sqrt{V(w_{ik})}}\right )+\sum_{ik}\frac{\partial\hat{L}(0)}{\partial \Im(z_{ik})}
	\Im\left (\frac{c_{n-i_{0}}\xi_{n-i_{0}}w_{ik}^{i_{0}}}{c_n\sqrt{V(w_{ik})}}\right ),\nonumber\\
	\newline
	\hat{L}_{2}&:=&\left.\frac{\d ^{2}\hat{L}\left (\frac{c_{n-i_{0}}\xi_{n-i_{0}}w_{ik}^{i_{0}}}{c_n\sqrt{V(w_{ik})}}t\right )_{ik}}{\d t^{2}}\right|_{t=0},\nonumber
	\end{eqnarray}
	
	and 
	\begin{equation}
	| \err_{2}| \ll (\log n)^{3\beta'}\frac{c^{3}_{n-i_0}}{c_n^{3}}|\xi_{n-i_{0}}|^{3}\left(\sum_{ik}|w_{ik}|^{i_0}\delta_{i}^{1/2}\right)^{3}.\nonumber
	\end{equation}
	where we used  \eqref{cond:ci:kac} to get that
	$$V(w_{ik}) = \sum_{j=0}^{n} \frac{c_{n-j}^{2}}{c_n^{2}} |w_{ik}|^{2j}\gg \sum_{j=\delta_i^{-1}/4}^{\delta_i^{-1}/2} |w_{ik}|^{2i} \gg \delta_i^{-1}.$$
	Similarly, we get the expansion for $\hat{L}\left(\frac{ \xi_{n-i_{0}}w_{ik}^{i_{0}}}{\sqrt{V(w_{ik})}}\right)_{ik}$. Subtracting the two expansions and taking expectation both sides (noting again that all of the $\xi_i$ are iid standard Gaussian and in particular, have mean 0 and variance 1), we obtain
	\begin{equation}\label{eq:di0:R}
	\begin{split}
	&(\log n)^{-3\beta'} d_{i_0}\ll \left |\frac{c^{2}_{n-i_0}}{c_n^{2}}-1\right |\left (\sum_{ik}  |w_{ik}|^{i_0}\delta_i^{1/2}\right )^{2} +  \left (\frac{c^{3}_{n-i_0}}{c_n^{3}}+1 \right )\left (\sum_{ik} |w_{ik}|^{i_0}\delta_i^{1/2}\right )^{3}\\
	&\ll   (\log n)^{O(\beta')}\left |\frac{c^{2}_{n-i_0}}{c_n^{2}}-1\right | \sum _{i=1}^{M} \delta_i (1-\delta_i/2)^{2i_0} +  (\log n)^{O(\beta')}\left (\frac{c^{3}_{n-i_0}}{c_n^{3}}+1 \right )  \sum _{i=1}^{M} \delta_i^{3/2} (1-\delta_i/2)^{3i_0}
	\end{split}
	\end{equation}
	where in the last inequality, we used $|w_{ik}|\le 1 - \delta_i/2$, Cauchy-Schwartz inequality and the fact that $M \ll \log n$.
	Note that for each $i$,
	\begin{eqnarray}
	\sum_{i_0=0}^{n} \frac{c^{3}_{n-i_0}}{c_n^{3}} (1-\delta_i/2)^{3i_0}&\le & \sum_{i_0=0}^{n/2} \frac{c^{3}_{n-i_0}}{c_n^{3}} (1-\delta_i/2)^{2i_0}+ \sum_{i_0=n/2}^{n} \frac{c^{3}_{n-i_0}}{c_n^{3}} (1-\delta_i/2)^{2i_0} \nonumber\\
	&\ll& \sum_{i_0=0}^{n/2}  (1-\delta_i/2)^{2i_0}  +  n^{O(1)}(1-\delta_i/2)^{n} \ll \delta_i^{-1}\nonumber
	\end{eqnarray}
	where in the second to last inequality, we used Condition \eqref{cond_ci_rho} and in the last inequality, we used $n^{O(1)} (1-\delta_i/2)^{n}\le n^{O(1)} (1-b_n/2)^{n}\le n^{O(1)}e^{-b_n n/2}\ll 1$ by the choice of $b_n$ in \eqref{choice:an:bn:var}.
	
	Thus, plugging this into \eqref{eq:di0:R} and using $\sum_{i=1}^{M} \delta_i^{1/2}\ll a_n^{1/2} \ll (\log n)^{-C}$ for any constant $C$,  
	\begin{eqnarray}
	(\log n)^{-3\beta'} \sum_{i_0=0}^{n} d_{i_0} &\ll&  (\log n)^{O(\beta')} \sum_{i=1}^{M} \sum_{i_0=0}^{n} \left |\frac{c^{2}_{n-i_0}}{c_n^{2}}-1\right |  \delta_i (1-\delta_i/2)^{2i_0} +  (\log n)^{O(\beta')} \sum_{i=1}^{M} \delta_i^{1/2}\nonumber.
	\end{eqnarray}
	Let 
	$$I_0 := a_n^{-1/2} = \exp\left (\log^{1/5}n\right )\quad \text{and}\quad I_1 := \frac{(\log n)^{2}}{b_n}\le n \exp\left (-\log^{1/5}n\right ).$$
	Splitting the double sum 
	$$ \sum_{i=1}^{M} \sum_{i_0=0}^{n} \left |\frac{c^{2}_{n-i_0}}{c_n^{2}}-1\right |  \delta_i (1-\delta_i/2)^{2i_0}$$ 
	into $\sum_{i=1}^{M}\sum_{i_0=I_0}^{I_1} $, $\sum_{i=1}^{M}\sum_{i_0=0}^{I_0-1}$, and $\sum_{i=1}^{M}\sum_{i_0=I_1+1}^{n}$ and denoting the corresponding sums by $S_1, S_2, S_3$, we obtain 
	\begin{eqnarray}
	(\log n)^{-3\beta'} \sum_{i_0=0}^{n} d_{i_0} &\ll&  (\log n)^{O(\beta')} (S_1+ S_2+S_3)+ (\log n)^{-4\beta'}\nonumber.
	\end{eqnarray}

	By assumption \eqref{cond:ci:kac}, we have for every $i_0\in [ I_0, I_1]$, 
	$$\frac{c^{2}_{n-i_0}}{c_n^{2}} -1 \ll \exp\left (-(\log \log n)^{1+\ep}\right ) .$$
	
	Hence, 
	$$S_1 \ll\exp\left (-(\log \log n)^{1+\ep}\right ) \sum_{i=1}^{M}\sum_{i_0=0}^{n}   \delta_i (1-\delta_i/2)^{2i_0}\ll M\exp\left (-(\log \log n)^{1+\ep}\right ).$$
	
	For $S_2$, we observe that $\frac{c^{2}_{n-i_0}}{c_n^{2}} \ll 1$ for all $i_0\le I_0\le n/2$ by Condition \eqref{cond_ci_rho} and so
	$$S_2 \ll \sum_{i=1}^{M}\sum_{i_0=0}^{I_0-1} \delta_i  \ll I_0 a_n = a_n^{1/2}.$$
	
	For $S_3$, we observe that $\frac{c^{2}_{n-i_0}}{c_n^{2}} \ll n^{O(1)}$ for all $i_0$ by Condition \eqref{cond_ci_rho} and that for all $i_0\ge I_1$,
	$$(1-\delta_i/2)^{i_0}\ll (1 - b_n/2)^{I_1} \ll \exp\left (-b_n I_1/2\right ) = \exp(-\log^{2}n).$$
	And so, 
	$$S_3\ll n^{O(1)}\sum_{i=1}^{M}\sum_{i_0=I_1+1}^{n} \exp(-\log^{2}n)\ll  n^{O(1)}\exp(-\log^{2}n).$$
	Combining these bounds, we obtain 
	\begin{eqnarray}
	(\log n)^{-3\beta'} \sum_{i_0=0}^{n} d_{i_0} &\ll&  (\log n)^{-4\beta'}\nonumber
	\end{eqnarray}
	proving \eqref{eq:universality:R:d} and completing the proof of Lemma \ref{lm:log-universality-smooth:R}.
\end{proof}

\subsection{Proof of  \eqref{varR}}
	As shown in \cite{Mas1}, for the Kac polynomial $\hat R_n$ (recall that the random variables $\xi_i$ are iid standard Gaussian), $\Var\left (N_{\hat R_n} \left (-1, 1\right )\right )\gg \log n$.
	
	By Proposition \ref{prop:boundedge} for the Kac polynomial $\hat R_n$ and the choice of $a_n, b_n$ in \eqref{choice:an:bn:var},
	$$\E N^{2}_{\hat R_n} ([-1, 1]\setminus \mj) \ll \E N^{2}_{\hat R_n} (\R\setminus \mj) =o\left (\log n\right ).$$
	So, by the triangle inequality, 
	$$\sqrt{\Var \left (N_{\hat R_n}\left (\mj \cap [-1, 1]\right )\right )}\ge \sqrt{\Var \left (N_{\hat R_n}\left (-1, 1\right )\right )} -o(\sqrt{\log n})\gg \sqrt{\log n}.$$            
This together with Proposition \ref{prop:Rmoment} imply \eqref{varR}.
\qed

\subsection{Proof of \eqref{covPR}} By a classical  formula \cite[Theorem 1]{Kac2}, we have that for every $a<b$ and for every nonzero polynomial $f$, 
\begin{eqnarray}
N_{f}(a, b)&=&\frac{1}{2\pi}\int _{\R}\int_{a}^{b} |f'(x)|\cos(sf(x))dxds\nonumber\\
&=&\frac{1}{2\pi^{2}}\int _{\R}\int_{a}^{b} \int_{\R}\frac{1}{u^{2}} (1-\cos(uf'(x)))\cos(sf(x))dudxds.\nonumber
\end{eqnarray}
%

We will apply this formula for both $P_n$ and $R_n$. To avoid the improper integrals, we need to cut off the domain of integration. Let $D := \exp\left (a_n^{-1}/100\right )$, $\gamma := D^{-3}$ and approximate $N_f(a, b)$ by 
$$N_{f}^{(1)}(a, b) := \frac{1}{2\pi}\int_{-D}^{D} \int_{a}^{b} |f'(x)|\cos(sf(x)) dxds,$$
and
$$N_{f}^{(2)}(a, b) := \frac{1}{2\pi^{2}}\int_{-D}^{D}\int_{a}^{b}\int_{\gamma}^{D^{2}} \frac{1}{u^{2}} \left (1- \cos(uf'(x))\right ) \cos(sf(x))dudxds.$$

We first  show that $N_f^{(1)}$ is a good approximation of $N_f$. We claim that for any $(a, b)\subset \mj\cap [-1, 1]$, 
\begin{equation}
\E\left |N_{P_n}(a, b) - N^{(1)}_{P_n}(a, b)\right |^{2} \ll \exp\left (-\Omega(a_n^{-1})\right )\label{eqn:diff1}.
\end{equation}
To show this, let $x_1<\dots<x_k$ be all the roots of $P_n'(x)$ in the interval $(a, b)$ and let $x_0 = a, x_{k+1} = b$. We have $k\le n$. Since $P'_n$ keeps the same sign on each interval $(x_i, x_{i+1})$, it holds that
\begin{eqnarray}
N_{P_n}(a, b) - N^{(1)}_{P_n}(a, b) &\le& \frac{1}{2\pi} \sum_{i=0}^{k}\left |\int_{|s|\ge D} \frac{\sin\left (sP_n(x_{i+1})\right )- \sin\left (sP_n(x_{i})\right )}{s}ds\right |\nonumber\\
&\ll & \sum_{i=0}^{k+1} \min\left \{1, \frac{1}{D|P_n(x_i)|}\right \} \nonumber
\end{eqnarray}
where we used \eqref{eq:sin:bound2}.
Thus, 
\begin{equation}\label{eq:N:N1}
\left (\E\left |N_{P_n}(a, b) - N^{(1)}_{P_n}(a, b)\right |^{2} \right )^{1/2} \ll \sum_{i=0}^{k+1} \left (\E \min\left \{1, \frac{1}{D|P_n(x_i)|}\right \}^{2}\right )^{1/2}.
\end{equation}

Divide the interval $(a, b)$ into $D^{1/2}$ equal intervals by the points $a=a_0<a_1<\dots<a_{D^{1/2}} = b$.

Let $p=1/4$ (or any small constant). For each $1\le i\le k$, assume that $x_i\in (a_j, a_{j+1}]$ for some $j$. If $|P_n(x_i)|\le D^{p-1}$ and $|P_n(a_{j+1})| \ge 2D^{p-1}$ then
$$\left |P_n(a_{j+1}) - P_n(x_i) \right |=\left |\int_{x_i}^{a_{j+1}} \int_{x_i}^{t} P_n''(u) dudt\right |\ge D^{p-1}$$
and so 
$$\int_{a_j}^{a_{j+1}} \int_{a_j}^{a_{j+1}} |P_n''(u)| dudt\ge D^{p-1}.$$
This happens with small probability
\begin{equation}\label{eq:corPR:multipleroot}
\P\left (\int_{a_j}^{a_{j+1}} \int_{a_j}^{a_{j+1}} |P_n''(u)| dudt\ge D^{p-1}\right )\ll D^{-1}.
\end{equation}
We defer the proof of \eqref{eq:corPR:multipleroot} to Appendix \ref{app:corPR:multipleroot} as it is similar to the proof of Lemma \ref{multipleroot}.
Using this and the union bound over all $D^{1/2}$ possible values of $j$, we get 
\begin{eqnarray}
P\left (|P_n(x_i)|\le D^{p-1}\right )&\le& \P\left (\exists j: |P_n(a_j)|\le 2D^{p-1}\right ) + D^{1/2} O\left (D^{-1}\right )\nonumber\\
&\ll & D^{1/2} D^{p-1}+ D^{-1/2}\ll D^{-1/4}\nonumber
\end{eqnarray}
where we used the fact that $P_n(a_j)$ is a Gaussian random variable with variance $\Omega(1)$.
%
Plugging this into \eqref{eq:N:N1} and using $k\le n, p=1/4$ yield
\begin{equation}
\left (\E\left |N_{P_n}(a, b) - N^{(1)}_{P_n}(a, b)\right |^{2} \right )^{1/2} \ll n. \left (D^{-1/4} +D^{-1/8}\right ) \ll \exp\left (-\Omega(a_n^{-1})\right ).\nonumber
\end{equation}
This proves \eqref{eqn:diff1} which means that $N^{(1)}_{P_n}(a, b)$ is a good approximation of $N_{P_n}(a, b)$.

Next, we show that for all $(a, b)\subset \mj\cap [-1, 1]$, $N^{(2)}_{P_n}(a, b)$ is also a good approximation of $N^{(1)}_{P_n}(a, b)$, namely,
\begin{equation}\label{eq:E:1:2}
\E\left |N^{(1)}_{P_n}(a, b) - N^{(2)}_{P_n}(a, b)\right |^{2} \ll \exp\left (-\Omega(a_n^{-1})\right ).
\end{equation}

To start, using the fact that $0\le 1-\cos x\le x^{2}$ for every real number $x$, we have
\begin{eqnarray}
\left |N^{(1)}_{P_n}(a, b) - N^{(2)}_{P_n}(a, b)\right | &\ll&   \int_{-D}^{D}\int_{a}^{b}\int_{0}^{\gamma} |P'_n(x)|^{2}dudxds+ \int_{-D}^{D}\int_{a}^{b}\int_{D^{2}}^{\infty} \frac{1}{u^{2}}dudxds\nonumber\\
&\ll& D^{-2} \int_{a}^{b} |P_n'(x)|^{2}dx +  D^{-1}.\nonumber
\end{eqnarray}
Taking the second moment of both sides, we get  
\begin{eqnarray*}
\E\left |N^{(1)}_{P_n}(a, b) - N^{(2)}_{P_n}(a, b)\right |^{2} &\ll& D^{-1}+ D^{-2} \int_{a}^{b} \E|P_n'(x)|^{4}dx\\
& \ll& D^{-1}+ D^{-2} n^{O(1)} \ll \exp\left (-\Omega(a_n^{-1})\right )\nonumber
\end{eqnarray*}
where we again used the fact that $a_n$ satisfies \eqref{cond-a-log}. This proves \eqref{eq:E:1:2}.

Combining this with \eqref{eqn:diff1}, we conclude that  for any $(a, b)\subset  \mj\cap [-1, 1]$, 
$$ \E\left |N_{P_n}(a, b) - N^{(2)}_{P_n}(a, b)\right |^{2} \ll \exp\left (-\Omega(a_n^{-1})\right ).$$

We can obtain a similar estimate for $R_n$. Therefore,  in order to prove \eqref{covPR}, it suffices to show 
\begin{equation} 
\Cov \left (N^{(2)}_{P_n}  \left(\mj \cap [-1, 1]\right ) , N^{(2)}_{R_n}  \left(\mj \cap [-1, 1]\right ) \right )= o(\log n) \label{covPR:2}.
\end{equation}
To prove this bound, we  need to make a critical use of a property of Gaussian variable. For a standard Gaussian random variable $Z$ and any real number $a$, $\E \cos(aZ) = E e^{iaZ} = e^{-a^{2}/2}$. 
Since $P_n(x), R_n (x)$ are Gaussian for any value of $x$, we have for $(a, b), (c, d)\subset \mj\cap [-1, 1]$,
\vspace{4mm} 
\begin{eqnarray}\label{eq:N2:cov}
&&\Cov \left (N^{(2)}_{P_n}  \left(a, b\right ) , N^{(2)}_{R_n}  \left(c, d\right ) \right )\nonumber\\
&&= \frac{1}{4\pi^{4}}\int_{a}^{b}\int_{c}^{d}\int_{\gamma}^{D^{2}}\int_{\gamma}^{D^{2}}\int_{-D}^{D}\int_{-D}^{D} \frac{1}{u^{2}v^{2}} (F_1+F_2+F_3+F_4) dtdsdvdudydx 
\end{eqnarray}
\vspace{4mm} 
where 
\begin{eqnarray}
F_1(x, y, u, v, s, t) &:=& \frac{1}{8} \sum \exp\left (-\frac{1}{2}\Var \left (sP_n(x)\pm uP_n'(x)\pm tR_n(y)\pm vR_n'(y)\right )\right )\nonumber\\
&-&  \frac{1}{4} \sum \exp\left (-\frac{1}{2}\Var \left (sP_n(x)\pm uP_n'(x)\right )-\frac{1}{2}\Var\left ( tR_n(y)\pm vR_n'(y)\right )\right )\nonumber
\end{eqnarray}
in which the sums are taken over all possible assignments of $+$ and $-$ signs in place of the $\pm$ and 
\begin{eqnarray}
F_2(x, y, u, v, s, t) &:=& -F_1(x, y, 0, v, s, t) \nonumber,\\
F_3(x, y, u, v, s, t) &:=& -F_1(x, y, u, 0, s, t) \nonumber,\\
F_4(x, y, u, v, s, t) &:=& F_1(x, y, 0, 0, s, t) \nonumber.
\end{eqnarray}

  These formulas follow directly from the definition of $N^{(2)}$; we provide the tedious derivation in Appendix \ref{app:N2:cov} for the reader's convenience.

 We now show that for $(a, b), (c, d)\subset \mj\cap [-1, 1]$ and for all $i=1, 2, 3, 4$, 
\begin{equation}\label{eq:6:int}
\int_{a}^{b}\int_{c}^{d}\int_{\gamma}^{D^{2}}\int_{\gamma}^{D^{2}}\int_{-D}^{D}\int_{-D}^{D} \frac{1}{u^{2}v^{2}} F_i dtdsdvdudydx = o(1).
\end{equation}
We will show it for $i=4$. The cases $i=1, 2, 3$ are completely similar. We have
\begin{equation}\label{eq:F4}
F_4(x, y, u, v, s, t) = \exp\left (-\frac{s^{2}}{2}\sum_{i=0}^{n} c_i^{2}x^{2i}\right )\exp\left (-\frac{t^{2}}{2}\sum_{i=0}^{n} \frac{c_i^{2}y^{2n-2i}}{c_n^{2}}\right )\left (\frac{e^{st\Delta}+e^{-st\Delta}}{2}-1\right )
\end{equation}
where
$$\Delta = \sum_{i=0}^{n} \frac{c_i^{2}}{c_n} x^{i}y^{n-i}.$$
Since $|x|, |y|\le 1-b_n$ and $nb_n \ge a_n^{-1}/2\gg C\log n$ for any constant $C$, we have
$$\Delta  \ll  n^{O(1)}\sum_{i=0}^{n}  (1-b_n)^{n}  \ll   n^{O(1)}\exp\left (-n b_n\right ) \ll \exp\left (-a_n^{-1}/4\right ).$$ 
Thus, bounding the first two exponents in \eqref{eq:F4} by 1, using $D = \exp\left (a_n^{-1}/100\right )$ and $|s|, |t|\le D$, we get that on the domain of integration in \eqref{eq:6:int},
$$F_4(x, y, u, v, s, t) \ll \exp\left (O(1) D^{2}\exp\left (-a_n^{-1}/4\right )\right )-1\ll D^{2}\exp\left (-a_n^{-1}/4\right )\ll \exp\left (-a_n^{-1}/5\right ).$$
Finally, using $\gamma = D^{-3}$, we have
\begin{equation} 
\int_{a}^{b}\int_{c}^{d}\int_{\gamma}^{D^{2}}\int_{\gamma}^{D^{2}}\int_{-D}^{D}\int_{-D}^{D} \frac{1}{u^{2}v^{2}} F_4 dtdsdvdudydx \ll D^{8}\exp\left (-a_n^{-1}/5\right ) = o(1)\nonumber
\end{equation}
proving \eqref{eq:6:int} and completing the proof of \eqref{covPR}.
\qed

{\bf Acknowledgements.} Oanh would like to thank Fran\c{c}ois Baccelli and Terence Tao for valuable conversations concerning random polynomials. We thank the anonymous referees for their helpful suggestions.

\bibliographystyle{plain}
\bibliography{polyref}

\section{Appendix}

\subsection{Proof of the Jensen's inequality \eqref{eq:Jensen}}\label{app:jensen}
By setting $g(w) = f\left (R(w+z)\right )$ and prove the corresponding inequality for $g$, it suffices to assume that $z = 0$ and $R=1$. Let $a_1, \dots, a_N$ be the zeros of $f$ in $\bar B(0, r)$. For each $a$ inside the unit disk $D$, consider the map 
$$T_a(w) = \frac{w-a}{\bar a w - 1}.$$

For $|a|\le r$ and $|w|\le r$, one can show by algebraic manipulation that 
$$|T_a(w)|\le \frac{2r}{1+r^{2}}<1.$$

Moreover, for all $|a|<1$ and $|w|=1$, we have
$$|T_a(w)| = |\bar w|\left |\frac{w-a}{\bar a w - 1}\right | = \left |\frac{1-a \bar w}{\bar a w - 1}\right |=1.$$

Let $h(w) = \frac{f(w)}{\prod_{k=1}^{N} T_{a_k}(w)}$. Then $h$ is an analytic function on $D$. By maximum principle, we have for every $w_0\in rD$,
\begin{eqnarray}
\frac{|f(w_0)|(1+r^{2})^{N}}{(2r)^{N}}\le \max_{w\in rD} |h(w)| \le \max_{w\in D} |h(w)| = \max_{w\in \partial D} |h(w)|=\max_{w\in \partial D} |f(w)|= M_1\nonumber.
\end{eqnarray}

Thus, $N\le \frac{\log \frac{M_1}{|f(w_0)|}}{\log\frac{1+r^{2}}{2r}}$ for all $w_0\in rD$, completing the proof.\qed

\subsection{Proof of \eqref{eq:S23}}\label{app:proof:S23}
We first reduce to the hyperbolic polynomials for which the Kac-Rice formula \eqref{eq:KacRice} is easier to handle. Consider the hyperbolic polynomial with coefficients $c_{j, \rho}:=\sqrt{\frac{(2\rho+1)\dots (2\rho+j)}{j!}}$, $0\le j\le n$. 

By condition \eqref{cond:ci:kac}, $c_{j, \rho} =\Theta (c_j)$ for all $j\ge N_0$. Using the Kac-Rice formula \eqref{eq:KacRice}, we have 
\begin{equation}\label{eq:KR1}
\E \left (\tilde N_n(S_2\cup S_3)\right )  \ll   \int_{S_2\cup S_3} \frac{\sqrt{\sum_{j=0}^{n}\sum_{k=j+1}^{n} c_{j, \rho}^{2} c_{k, \rho}^{2}(k-j)^{2} t^{2j+2k-2}}}{\sum_{j=0}^{n} c_{j, \rho}^{2}t^{2j}}dt.
\end{equation}   

We use \cite[Lemma 10.3]{DOV} with $h(k) = c_{k, \rho}^{2}$ which estimates the above integrand uniformly over the interval $\left (1-\frac{1}{C}, 1-\frac{C}{n}\right )$ for some sufficiently large constant $C$ and asserts that 
\begin{eqnarray}
\frac{\sqrt{\sum_{j=0}^{n}\sum_{k=j+1}^{n} c_{j, \rho}^{2} c_{k, \rho}^{2}(k-j)^{2} t^{2i+2k-2}}}{\sum_{j=0}^{n} c_{j, \rho}^{2}t^{2i}} &\ll & \frac{\sqrt{2\rho+1}}{2\pi(1-t)} +   (1-t)^{\rho-1/2} + \frac{1}{n(1-t)^{2}}\ll \frac{1}{1-t}\nonumber.		
\end{eqnarray}

This together with \eqref{eq:KR1} give \eqref{eq:S23} for $\delta_i\ge \frac{2C}{n}$ as in this case, $S_2\cup S_3\subset \left (1-\frac{1}{C}, 1-\frac{C}{n}\right )$.

If $\delta_i \le \frac{2C}{n}$, since $k-j\le n$, for all $t\in S_2\cup S_3$, we have
\begin{equation}\label{d2}
\frac{\sqrt{\sum_{j=0}^{n}\sum_{k=j+1}^{n} c_{j, \rho}^{2} c_{k, \rho}^{2}(k-j)^{2} t^{2i+2k-2}}}{\sum_{j=0}^{n} c_{j, \rho}^{2}t^{2i}} \ll n.
\end{equation}
Plugging this into \eqref{eq:KR1} and using the fact that $2C/n\ge \delta_i\ge \delta_M\ge 1/n$ give
$$\E \left (\tilde N_n(S_2\cup S_3)\right )  \ll n\delta_i^{1+\alpha}\ll n^{-\alpha}\ll \delta_i^{\alpha}$$
and hence \eqref{eq:S23} for $\delta_i \le \frac{2C}{n}$, completing the proof of \eqref{eq:S23} for all values of $\delta_i$.
\qed

\subsection{Proof of Lemma \ref{lm:log-universality-general}}\label{app:log_universality}
 In this section, we deduce Lemma \ref{lm:log-universality-general} from Lemma \ref{lm:log-universality-smooth}.

The constant $\alpha_{0}$ in this proof will be a small fraction of the $\alpha_{0}$ in Lemma \ref{lm:log-universality-smooth}. Let $\bar{K}(x_{ik})_{ik}:=K(x_{ik}+\frac{1}{2}\log V(w_{ik}))_{ik}$. Then, $\bar{K}$ still satisfies \eqref{eq:derivativebound} and we can reduce the problem to showing that
 	$$
 	\left|\E \bar{K}\left(\log\frac{|P(w_{ik})|}{\sqrt{V(w_{ik})}}\right)_{ik}-\E \bar{K}\left(\log\frac{|\tilde{P}(w_{ik})|}{\sqrt{V(w_{ik})}}\right)_{ik}\right|=O(\delta_0^{\alpha}).
 	$$
 	Ideally, we would like to set $L(z_{ik})_{ik}:= \bar K(\log |z_{ik}|)_{ik}$ and apply Lemma \ref{lm:log-universality-smooth} for this function $L$. However, the singularity of the log function at 0 prevents $L$ from satisfying \eqref{eq:derivativebound}. To handle this difficulty, we split the space of $(\log|z_{ik}|)_{ik}$ into two regions $\Omega_1$ and $\Omega_2$ where $\Omega_1$ is the image of the log function around 0 and show that the contribution from $\Omega_1$ is insignificant. On $\Omega_2$, the log function is well-behaved and we can then apply Lemma \ref{lm:log-universality-smooth} there.
 	
 	More specifically, for $M_{i}:=\log\left (\delta_{i}^{-12\alpha}\right )$, let 
 	$$\Omega_{1}=\{ (x_{ik})_{ik}\in \mathbb{R}^{m_{1}+\cdots+m_{M}} : x_{ik}\le-M_{i} \text{ for some } i,k\}$$ 
 	and
 	$$\Omega_{2}=\{(x_{ik})_{ik}\in \mathbb{R}^{m_{1}+\cdots+m_{M}} : x_{ik}\ge-M_{i}-1 \text{ for all }i,k\}.$$
 	
 	
 	Let $\psi : \mathbb{R}^{m_{1}+\cdots+m_{M}}\rightarrow[0, 1]$ be a smooth function taking values in $[0, 1]$ such that $\psi$ is supported in $\Omega_{2}$, $\psi=1$ on the complement of $\Omega_{1}$ and $\Vert\partial^{a}\psi\Vert_{\infty}=O(1)$ for all $0\leq a\leq 3$. Put $\phi:=1-\psi,  K_{1}:=\bar{K}.\phi$, and $ K_{2}:=\bar{K}.\psi$. We have $\bar{K}=K_{1}+K_{2}$ and both $K_{1}, K_{2}$ satisfy \eqref{eq:derivativebound} with $\supp K_1 \subset {\Omega}_{1}$, $\supp K_2\subset {\Omega}_{2}.$
 	
 	We now show that the contribution from $K_{1}$ is negligible. Set $\tilde{K}_{1}:=\Vert \bar{K}\Vert_{\infty}\phi$ and $$L_{1}(z_{ik})_{ik}:=\tilde{K}_{1}(\log|z_{ik}|)_{ik}.$$ 
 	Since $\Vert K_{1}\Vert_{\infty}\leq\Vert \bar{K}\Vert_{\infty}\ll 1$,  we observe that $L_{1}$ satisfies
 	\begin{itemize}
 		\item $|K_{1}(\log|z_{ik}|)_{ik}|\leq L_{1}(z_{ik})_{ik}$,
 		\item $\supp(L_{1})\subset \{ (z_{ik})_{ik}\in \mathbb{C}^{m_{1}+\cdots+m_{M}} : |z_{ik}|\leq e^{-M_{i}} \text{ for some } i,k\}$,
 		\item $L_1$ is constant on $ \{ (z_{ik})_{ik}\in \mathbb{C}^{m_{1}+\cdots+m_{M}} : |z_{ik}|\le e^{-M_{i}-1} \text{ for some } i,k\}$,
 		\item $L_{1}$ satisfies \eqref{eq:derivativebound} (with the power $ 2\alpha$ being replaced by $ 14\alpha$ but that doesn't affect the argument).
 	\end{itemize}
 	
 	Choose $\alpha_{0}$ to be small enough such that $C \alpha_{0}$ is at most the constant $\alpha_{0}$ in Lemma \ref{lm:log-universality-smooth} where $C $ is some sufficiently large absolute constant. Applying Lemma \ref{lm:log-universality-smooth}, we get
 	$$
 	\E \left|K_{1}\left(\log\frac{|P(w_{ik})_{ik}|}{\sqrt{V(w_{ik})_{ik}}}\right)\right| \leq \E L_{1}\left(\frac{P(w_{ik})}{\sqrt{V(w_{ik})}}\right)_{ik}\leq \E L_{1}\left(\frac{\tilde{P}(w_{ik})}{\sqrt{V(w_{ik})}}\right)_{ik}+O\left (\delta_0^{C \alpha}\right ).
 	$$
 	
 	Since the variables $\tilde{\xi}_{i}$ are Gaussian, we have
 	\begin{align*}
 	\E L_{1}\left(\frac{\tilde{P}(w_{ik})}{\sqrt{V(w_{ik})}}\right)_{ik} &\quad\ll \quad  \P \left(\exists ik: \frac{|\tilde{P}(w_{ik})|}{\sqrt{V(w_{ik})}}\leq e^{-M_{i}}\right) \ll \quad  \sum_{i=1}^{M}m_{i}\delta_{i}^{12\alpha} \ll \delta_0^{\alpha}.
 	\end{align*}
 	
 	Thus, $  \E \left|K_{1}\left (\log\frac{|P(w_{ik})|}{\sqrt{V(w_{ik})}}\right )_{ik}\right|\ll \delta_0^{\alpha}$. Finally, we will show that
 	$$
 	\left|\E K_{2} \left(\log\frac{|P(z_{1})|}{\sqrt{V(z_{1})}}, . . . , \log\frac{| P(z_{m})|}{\sqrt{V(z_{m})}}\right)-\E K_{2}\left(\log\frac{|\tilde P(z_{1})|}{\sqrt{V(z_{1})}}, \ldots , \log\frac{|\tilde P(z_{m})|}{\sqrt{V(z_{m})}}\right)\right|\ll \delta^{\alpha}_{0}.
 	$$
 	
 	Define $L_{2} : \mathbb{C}^{m_{1}+\cdots+m_{M}}\rightarrow \mathbb{R}$ by $L_{2}(z_{ik})=K_{2}(\log|z_{ik}|)$. Since $\supp  K_2 \subset {\Omega}_{2}$, $$\supp  L_2\subset \{ (z_{ik})_{ik}: |z_{ik}|\ge e^{-M_{i}-1}\gg  \delta_{i}^{12\alpha} \text{ for all } i,k\}.$$ 
 	Thus, $L_{2}$ is well-defined and  satisfies \eqref{eq:derivativebound} (with the power $2\alpha$ being replaced by $14\alpha$). Applying Lemma \ref{lm:log-universality-smooth} gives
 	\begin{align*}
 	& \E K_{2}\left(\log\frac{|P(w_{ik})|}{\sqrt{V(w_{ik})}}\right)_{ik}-\E K_{2}\left(\log\frac{|\tilde P(w_{ik})|}{\sqrt{V(w_{ik})}}\right)_{ik} \\
 	&\quad =\quad  \E L_{2}\left(\frac{|P(w_{ik})|}{\sqrt{V(w_{ik})}}\right)_{ik}-\E L_{2}\left(\frac{|\tilde P(w_{ik})|}{\sqrt{V(w_{ik})}}\right)_{ik} \ll \delta_0^{\alpha} .
 	\end{align*}
 	This completes the proof of Lemma \ref{lm:log-universality-general}.
 \qed

\subsection{Proof of \eqref{varbound}} \label{app:proof:varbound}
In this section, we prove \eqref{varbound}, namely, for a sufficiently large constant $C$, we have
\begin{equation}\label{varbound:2}
V(x) := \sum_{i=0}^{n}c_{i}^{2}x^{2i}=\frac{\Theta(1)}{(1-x+1/n)^{2\rho+1}}\quad \forall x\in(1-1/C, 1).
\end{equation}
To this end, we will repeatedly use \eqref{cond_ci_rho} and the assumption that $\rho>-1/2$.

If $x\ge 1-\frac{1}{n}$, we have 
\begin{equation}
V(x)\le \sum_{i=0}^{n}c_{i}^{2}\ll \sum_{i=0}^{N_0} 1 + \sum_{i=0}^{n}i^{2\rho} \ll n^{2\rho+1}\nonumber.
\end{equation}
For the lower bound, we have $x^{2i}\ge \left (1-\frac{1}{n}\right )^{2n}\gg 1$ and so
\begin{equation}
V(x)\gg \sum_{i=0}^{n}c_{i}^{2} \gg \sum_{i=N_0}^{n}i^{2\rho} \gg n^{2\rho+1}\nonumber.
\end{equation}
These bounds prove \eqref{varbound:2} for $x\ge 1-\frac{1}{n}$.

If $1-\frac{1}{n}<x<1-\frac{1}{C}$, letting $L = \frac{1}{1-x} \in (C, n)$, we have $\frac{1}{(1-x+1/n)^{2\rho+1}} =\Theta\left (L^{2\rho+1}\right )$ and
\begin{equation}
V(x)\gg \sum_{i=L}^{2L}c_{i}^{2} x^{2L} \gg \sum_{i=L}^{2L} i^{2\rho}  \gg L^{2\rho+1}\nonumber.
\end{equation}
As for the upper bound, we have for any constant $C'$,
\begin{equation}\label{eq:varbound:L:1}
V(x)\ll \sum_{i=0}^{N_0} 1+ \sum_{i=N_0}^{\infty} i^{2\rho} x^{2i}\ll 1 + \sum_{i=0}^{C'L}i^{2\rho} x^{2i} + \sum_{i=C'L}^\infty i^{2\rho} x^{2i}\ll L^{2\rho+1}+ \sum_{i=C'L}^\infty i^{2\rho} x^{2i}.
\end{equation}
Since $x=1-\frac{1}{L} \le e^{-1/L}$, the right-most sum is at most
\begin{equation}\label{eq:varbound:L:2}
\sum_{i=C'L}^\infty i^{2\rho} x^{2i} \le \sum_{i=C'L}^\infty i^{2\rho} e^{-2i/L} = L^{2\rho} \sum_{i=C'L}^\infty \left (\frac{i}{L}\right )^{2\rho}  e^{-2i/L}.
\end{equation}
By choosing $C'$ sufficiently large (depending only on $\rho$) such that the function $t\to t^{2\rho} e^{-2t}$ is decreasing on $(C'-1, \infty)$, we have
\begin{equation}\label{eq:varbound:L:3}
\sum_{i=C'L}^\infty \left (\frac{i}{L}\right )^{2\rho}  e^{-2i/L}\le \int_{C'L-1}^{\infty} \left (\frac{s}{L}\right )^{2\rho}  e^{-2s/L}ds = L\int_{C'-1/L}^{\infty} t^{2\rho} e^{-2t} dt\ll L.
\end{equation}
Plugging this into \eqref{eq:varbound:L:1} and \eqref{eq:varbound:L:2}, we obtain $V(x)\ll L^{2\rho+1}$ which is the desired upper bound.
\qed

\subsection{Proof of \eqref{eq:1-r:rho} and \eqref{boundvariance}}\label{proof:boundavariance}
\begin{proof} [Proof of \eqref{eq:1-r:rho}]
We need to show that
\begin{equation}\label{eq:1-r:rho:2}
\sum_{0\le i<k\le n} c_i^{2}c_k^{2} (x^{i}y^{k} - x^{k}y^{i})^{2} = \Theta\left (\sum_{0\le i<k\le n} c_{i, \rho}^{2}c_{k, \rho}^{2} (x^{i}y^{k} - x^{k}y^{i})^{2}\right ).
\end{equation}
	By Condition \eqref{cond_ci_rho}, $c_i\ll c_{i, \rho}$ for all $i\ge 0$, and so the left-hand side of \eqref{eq:1-r:rho:2} is at most the order of the right-hand side. To prove the reverse, by Condition \eqref{cond_ci_rho}, $c_i\gg c_{i, \rho}$ for all $i\ge N_0$, so 
	\begin{equation}
	\sum_{0\le i<k\le n} c_i^{2}c_k^{2} (x^{i}y^{k} - x^{k}y^{i})^{2} \ge \sum_{N_0\le i<k\le n} c_i^{2}c_k^{2} (x^{i}y^{k} - x^{k}y^{i})^{2} \gg \sum_{N_0\le i<k\le n} c_{i, \rho}^{2}c_{k, \rho}^{2} (x^{i}y^{k} - x^{k}y^{i})^{2} \nonumber.
	\end{equation}
	Thus, it remains to show that the remaining terms on the right-hand side of \eqref{eq:1-r:rho:2} are of smaller order, namely, for all $1\le i<N_0$,
	\begin{equation}\label{eq:1-r:rho:3}
	\sum_{k=0}^{n} c_{i, \rho}^{2}c_{k, \rho}^{2} (x^{i}y^{k} - x^{k}y^{i})^{2} \ll \sum_{N_0\le j\le n}\sum_{N_0\le k\le n} c_{j, \rho}^{2}c_{k, \rho}^{2} (x^{j}y^{k} - x^{k}y^{j})^{2}.
	\end{equation}
	Since $N_0$ is a constant, $c_{k, \rho} =\Theta (c_{k+N_0, \rho})$ for all $k\ge 0$ and since $xy = \Theta(1)$, we have for $j' = i+N_0$,
		\begin{eqnarray} \label{eq:1-r:rho:4}
	\sum_{k=0}^{n} c_{i, \rho}^{2}c_{k, \rho}^{2} (x^{i}y^{k} - x^{k}y^{i})^{2} &\ll& \sum_{k=0}^{n}  c_{j', \rho}^{2}c_{k+N_0, \rho}^{2} (x^{j'}y^{k+N_0} - x^{k+N_0}y^{j'})^{2}\nonumber\\
	&=& \sum_{k=N_0}^{n+N_0}  c_{j', \rho}^{2}c_{k, \rho}^{2} (x^{j'}y^{k} - x^{k}y^{j'})^{2}.
	\end{eqnarray}
	Assume without loss of generality that $x<y$. Using the simple observation that 
	$$0\le y^{j+1} - x^{j+1} \le 2(y^{j} - x^{j})\quad\forall j\ge 1,$$
	we have
	\begin{equation}
	\sum_{k=n+1}^{n+N_0}  c_{j', \rho}^{2}c_{k, \rho}^{2} (x^{j'}y^{k} - x^{k}y^{j'})^{2}\ll \sum_{k=n+1-N_0}^{n}  c_{j', \rho}^{2}c_{k, \rho}^{2} (x^{j'}y^{k} - x^{k}y^{j'})^{2}.\nonumber
	\end{equation}
	And so, the right-most side of \eqref{eq:1-r:rho:4} is of order at most the right-most side of \eqref{eq:1-r:rho:3}, proving \eqref{eq:1-r:rho:2}.
\end{proof}
\begin{proof}[Proof of \eqref{boundvariance}]
	We want to show that for every $x\in[1-a_n, 1-b_n]$,  
\begin{equation} 
\sum_{k=0}^{n} c_{k, \rho}^{2} x^{2k}=  \frac{1+O(\ep_0)}{(1-x^{2})^{2\rho+1}}.\nonumber
\end{equation}
where $ \varepsilon_0=\exp\left (-(\log\log n)^{1+2\varepsilon}\right )$. 
By Taylor's expansion, we have
 $$S: = \sum_{k=0}^{\infty} c_{k, \rho}^{2} x^{2k}=  \frac{1}{(1-x^{2})^{2\rho+1}}.$$
 Thus, it suffices to show that 
 $$\sum_{k=n+1}^{\infty} c_{k, \rho}^{2} x^{2k}\ll \ep_0 S.$$
 We have 
 \begin{eqnarray}
 \sum_{k=n+1}^{\infty} c_{k, \rho}^{2} x^{2k}  = x^{2n+2}\sum_{k=0}^{\infty} \frac{c_{n+1+k, \rho}^{2}}{c_{k, \rho}^{2}} c_{k, \rho}^{2} x^{2k}\nonumber
 \end{eqnarray}
and so, it is left to verify that for all $k\ge 0$, 
\begin{equation} 
x^{2n}\frac{c_{n+1+k, \rho}^{2}}{c_{k, \rho}^{2}}\ll \ep_0.\nonumber
\end{equation}
Indeed, we have
\begin{eqnarray}
x^{2n}\frac{c_{n+1+k, \rho}^{2}}{c_{k, \rho}^{2}} &=& x^{2n} \prod_{i=1}^{n+1}\frac{2\rho + k+i}{k+i}\le  x^{2n} \prod_{i=1}^{n+1}\frac{2\rho + i}{i} \le  x^{2n} \prod_{i=1}^{n+1}\frac{2\left[\rho\right] + i+1}{i}   \nonumber\\
&=& x^{2n} \frac{(n+2)\dots (n+2+2\left[\rho\right] )}{(2\left[\rho\right] +1)!}\ll x^{2n}n ^{2\rho+1}.\nonumber
\end{eqnarray}
Using $x\le 1-b_n\le 1 - \frac{(\log n)^{2}}{n}$ by the assumption \eqref{eq:cond:ab:CLT}, we obtain
$$x^{2n}\frac{c_{n+1+k, \rho}^{2}}{c_{k, \rho}^{2}} \ll  \left (1 - \frac{(\log n)^{2}}{n}\right )^{2n}n ^{2\rho+1}\ll \exp\left (-2(\log n)^{2}\right )n ^{2\rho+1}\ll \ep_0 .$$
\end{proof}

\subsection{Proof of \eqref{eq:corPR:multipleroot}}\label{app:corPR:multipleroot}
	Let $(c, d):= (a_j, a_{j+1})$. We want to show that for any interval $(c, d)\subset [1-1/C, 1]$ with $d-c \le D^{-1/2}$,
	\begin{equation}\label{eq:corPR:multipleroot:1}
	\P\left (\int_{c}^{d} \int_{c}^{d} |P_n''(u)| dudt\ge D^{p-1}\right )\ll D^{-1}.
	\end{equation}
	Let $I$ denote the above double integral. By Markov's inequality and H\H{o}lder's inequality, for a large constant $h$ to be chosen, we have
	$$\P(I\ge D^{p-1})\le D^{(1-p)h}\E I^{h}\ll D^{(1-p)h} (d-c)^{2(h-1)}\E \int_{c}^{d} \int_{c}^{d} |P_n''(u)|^{h} dudt$$
	and so, 
	$$\P(I\ge D^{p-1})\ll D^{(1-p)h} (d-c)^{2h}\max_{u\in [c, d]} \E |P_n''(u)|^{h} \le D^{-ph}\max_{u\in [c, d]} \E |P_n''(u)|^{h}.$$
	Since $P''(u)$ is a Gaussian random variable, by the hypercontractivity of Gaussian distribution, we have
	$$\E |P_n''(u)|^{h}\ll \left (\E |P_n''(u)|^{2}\right )^{\frac{2+h}{2}}\ll n^{O(1)}.$$ 
	Thus, by choosing $h = 2/p$,
	$$\P(I\ge D^{p-1})\ll D^{-ph} n^{O(1)} \ll D^{-1}$$
	where we used $D = \exp\left (a_n^{-1}/100\right ) \gg n^{C}$ for any constant $C$ as $a_n$ satisfies Condition \eqref{cond-a-log}.
	\qed
	
\subsection{Proof of \eqref{eq:N2:cov}}\label{app:N2:cov}	
We have
\begin{eqnarray}
\Cov \left (N^{(2)}_{P_n}  \left(a, b\right ) , N^{(2)}_{R_n}  \left(c, d\right ) \right )=\E N^{(2)}_{P_n}  \left(a, b\right ) N^{(2)}_{R_n}  \left(c, d\right ) - \E N^{(2)}_{P_n}  \left(a, b\right )\cdot \E N^{(2)}_{R_n}  \left(c, d\right ) \nonumber.
\end{eqnarray}
Thus, we have by definition of $N^{(2)}$ that
\begin{eqnarray} 
 &&\Cov \left (N^{(2)}_{P_n}  \left(a, b\right ) , N^{(2)}_{R_n}  \left(c, d\right ) \right ) =  \frac{1}{4\pi^{4}}\int_{a}^{b}\int_{c}^{d}\int_{\gamma}^{D^{2}}\int_{\gamma}^{D^{2}}\int_{-D}^{D}\int_{-D}^{D} \frac{1}{u^{2}v^{2}}\cdot \label{eq:cov:N2:1}\\
 && \qquad\qquad \qquad \cdot\left [\E\mathcal F_1(x, u, s) \mathcal F_2 (y, v, t) -\E\mathcal F_1(x, u, s) \E \mathcal F_2 (y, v, t) \right ]dtdsdvdudydx \nonumber
\end{eqnarray}
where 
\begin{eqnarray}
\mathcal F_1(x, u, s) :=  \left (1 - \cos\left (u P_n'(x)\right )\right )\cos(sP_n(x)) ,    \mathcal F_2 (y, v, t) := \left (1 - \cos\left (v R_n'(y)\right )\right )\cos(tR_n(y)).\nonumber
\end{eqnarray}
Note that we can use Fubini's theorem in the above calculation because the integrands are absolutely integrable.

We have
\begin{eqnarray}
 \mathcal F_1(x, u, s) \mathcal F_2 (y, v, t) &=&  \left (1 - \cos\left (u P_n'(x)\right )\right )\cos(sP_n(x))  \left (1 - \cos\left (v R_n'(y)\right )\right )\cos(tR_n(y))\nonumber\\
&=& \cos(sP_n(x))  \cos(tR_n(y)) -  \cos\left (u P_n'(x)\right )\cos(sP_n(x)) \cos(tR_n(y)) \nonumber\\
&&-\cos\left (v R_n'(y)\right )\cos(sP_n(x)) \cos(tR_n(y)) \nonumber\\
&&+ \cos\left (u P_n'(x)\right )\cos\left (v R_n'(y)\right )\cos(sP_n(x)) \cos(tR_n(y))\nonumber
\end{eqnarray}
and so
\begin{eqnarray}
\mathcal F_1(x, u, s) \mathcal F_2 (y, v, t) &=& \frac{1}{2} \sum \cos\left (sP_n(x) \pm tR_n(y) \right ) - \frac{1}{4} \sum \cos\left (u P_n'(x)\pm sP_n(x) \pm tR_n(y) \right )  \nonumber\\
&&- \frac{1}{4} \sum \cos\left (v R_n'(y)\pm sP_n(x) \pm tR_n(y) \right ) \nonumber\\
&&+ \frac{1}{8} \sum \cos\left (u P_n'(x)\pm v R_n'(y)\pm sP_n(x) \pm tR_n(y) \right ) \nonumber.
\end{eqnarray}

We recall that the random variables $\xi_i$ are iid standard Gaussian and for a standard Gaussian random variable $Z$ and any real number $a$, $\E \cos(aZ) = E e^{iaZ} = e^{-a^{2}/2} = \exp\left (-\frac{1}{2}\Var (aZ)\right )$. 
Thus, 
\begin{eqnarray}
&&\E \mathcal F_1(x, u, s) \mathcal F_2 (y, v, t) =   \frac{1}{2} \sum \exp\left (-\frac{1}{2}\Var \left (sP_n(x) \pm tR_n(y) \right ) \right )\nonumber\\
&&- \frac{1}{4} \sum \exp\left (-\frac{1}{2}\Var\left (u P_n'(x)\pm sP_n(x) \pm tR_n(y) \right )\right )  \nonumber\\
&& - \frac{1}{4} \sum \exp\left (-\frac{1}{2}\Var\left (v R_n'(y)\pm sP_n(x) \pm tR_n(y) \right )\right )\nonumber\\
&& + \frac{1}{8} \sum \exp\left (-\frac{1}{2}\Var\left (u P_n'(x)\pm v R_n'(y)\pm sP_n(x) \pm tR_n(y) \right ) \right )\nonumber.
\end{eqnarray}

Similarly, 
\begin{eqnarray}
&&\E \mathcal F_1(x, u, s)=   \exp\left (-\frac{1}{2}\Var \left (sP_n(x) \right ) \right ) - \frac{1}{2} \sum \exp\left (-\frac{1}{2}\Var\left (u P_n'(x)\pm sP_n(x)  \right )\right )  \nonumber
\end{eqnarray}
and 
\begin{eqnarray}
&&\E \mathcal F_2(y, v, t)=   \exp\left (-\frac{1}{2}\Var \left (tR_n(y) \right ) \right ) - \frac{1}{2} \sum \exp\left (-\frac{1}{2}\Var\left (v R_n'(y)\pm tR_n(y)  \right )\right )  \nonumber.
\end{eqnarray}
Plugging these formulas into \eqref{eq:cov:N2:1}, we obtain \eqref{eq:N2:cov}.
\qed

\end{document}